\documentclass[a4paper,12pt]{amsart}
\usepackage{fullpage}
\usepackage[pdftex]{hyperref}

\usepackage{tikz}
\usepackage{amsaddr}
\usepackage{amsthm, amssymb, amsmath}
\usepackage{color}
\usepackage{enumerate}
\usepackage{thmtools,thm-restate}
\usepackage{mathtools}

\newtheorem{thm}{Theorem}[subsection]
\newtheorem*{thm*}{Theorem}
\newtheorem{cor}[thm]{Corollary}
\newtheorem{claim}[thm]{Claim}
\newtheorem*{claim*}{Claim}
\newtheorem{lem}[thm]{Lemma}
\newtheorem{prop}[thm]{Proposition}

\theoremstyle{definition}
\newtheorem{defn}[thm]{Definition}

\newtheorem{nt}[thm]{Notation}

\newtheorem{rem}[thm]{Remark}
\newtheorem{ex}[thm]{Example}
\newtheorem*{rem*}{Remark}
\newtheorem*{ack}{Acknowledgements}

\tikzset{every picture/.style={line width=0.75pt}} 

\def \p{{\mathbb P}}          

\def \Z{{\mathbb Z}}

\def\N{{\mathbb N}}           
\def\R{{\mathbb R}}           
\def\O{{\mathcal O}}          
\def\G{{\mathcal G}}          
\def\HH{{\mathcal H}}          
\def\e{\varepsilon}           
\def\f{\varphi}           

\def\Aut{\operatorname{Aut}}
\def\Out{\operatorname{Out}}
\def\Inn{\operatorname{Inn}}
\def\Lip{\operatorname{Lip}}

\def\Hyp{\operatorname{Hyp}}

\def\stab{\operatorname{Stab}}
\def\Cand{\operatorname{Cand}}

\def\vol{\operatorname{vol}}
\def\Min{\operatorname{Min}}

\def\rank{\operatorname{rank}}

\def\simfk{\langle \sim_{f^i}\rangle}

\newcommand{\surv}{\operatorname{surv}}
\newcommand{\ri}{\operatorname{right}}
\newcommand{\lef}{\operatorname{left}}
 
\title[On the action of $\phi$ on $\Min(\phi)$]{On the action of relatively irreducible automorphisms on their train tracks}
\author{Stefano Francaviglia}
\address{Dipartimento di Matematica of the University of
Bologna}
\email{stefano.francaviglia@unibo.it}
\author{Armando Martino}
\address{Mathematical Sciences, University of Southampton }
\email{A.Martino@soton.ac.uk}

\author{Dionysios Syrigos}
\address{Mathematical Sciences, University of Southampton }
\email{D.Syrigos@soton.ac.uk}

\begin{document}
	\begin{abstract}
Let $G$ be a group and let $\G$ be a free factor system of $G$, namely a free splitting of $G$
as $G=G_1*\dots*G_k*F_r$.
In this paper, we study the
set of train track points for $\G$-irreducible
automorphisms $\phi$ with exponential growth. 
Such set is known to coincide with the minimally displaced set $\Min(\phi)$ of $\phi$, in the
relative deformation space corresponding to the splitting. The theory of such relative spaces,
even if it is more general by its own nature, is crucial  to understanding reducible
automorphisms of free groups, as any automorphism is relatively irreducible with respect to \emph{some} free factor system $\mathcal{G}$.

Our main result is that $\Min(\phi)$ is co-compact, under the action of the cyclic subgroup
generated by $\phi$. 

Along the way we obtain other results that could be of independent interest. For instance,
we prove that any point of $\Min(\phi)$ is in uniform distance from $\Min(\phi^{-1})$.  We also
prove that the action of $G$ on the product of the attracting and the repelling trees for
$\phi$, is discrete. Finally, we get some fine insight about the local topology of relative
outer space. 

Some applications of co-compactness are discussed. In particular we
generalise a classical result of Bestvina, Feighn and Handel for the centralisers of 
irreducible automorphisms of free groups, in the more general context of relatively irreducible
automorphisms of a free product. From this, we deduce that centralisers of elements of $\Out(F_3)$
are finitely generated, which was previously unknown. Finally, we mention that an immediate
corollary of co-compactness is that the set $\Min(\phi)$ is always quasi-isometric to a line.		
	\end{abstract}

		
\maketitle

\tableofcontents

\section{Introduction}

\textbf{Overview.}
Automorphisms of free groups play a central role in Geometric Group Theory. Culler-Vogtmann Outer space is one of the main methods that are currently used for the study of automorphisms of free groups. Irreducible automorphisms  have been studied the most, as there are available many different tools for them (for instance, train tracks representatives~\cite{BH-TrainTracks}).

More recently, Guirardel and Levitt introduced in~\cite{GuirardelLevitt} the notion of a
relative outer space of a group corresponding to a free factor system. These relative spaces
have been used for the study of automorphisms of general free products but also for reducible
automorphisms of free groups, as any automorphism is relatively irreducible in the
appropriate relative outer space. Note that many of the classical tools that are available for
irreducible automorphisms, are also available in relative outer spaces; for instance, existence of train track representatives in the general context is proved in~\cite{FM15}.

In this paper, together with our companion paper~\cite{FMS}, we study relatively irreducible
automorphisms. In particular, we focus on their minimally displaced set (with respect to
the Lipschitz metric) which (by~\cite{FM15}) can be seen as the set of train track points
(and train track properties will be crucially used in our arguments).

\medskip

\textbf{Main Results of the Paper.} (See section~\ref{sec2} for a more complete explanation of our terminology.)
 
Let $G$ be a group, with a free factor system $\G$. Let $\Aut(\G)$ be the group of automorphisms of $G$ which preserve $\G$ (up to conjugacy) and $\Out(\G)=\Aut(\G)/ \Inn(G)$.  Denote by
$[\phi]\in\Out(\G)$ the outer class of $\phi\in\Aut(\G)$.

We then let $\O(\G)$ be the relative outer space
corresponding to $\G$; that is, the space of minimal edge-free actions of $G$ on metric trees
with stabilisers giving rise to $\G$, up to equivariant isometry. We denote by $\O_1(\G)$ the
co-volume one subspace of $\O(\G)$. Both $\Aut(\G)$ and $\Out(\G)$ act on $\O(\G)$ and
$\O_1(\G)$ by twisting the action. Inner automorphisms act trivially, hence the action
of $\phi\in\Aut(\G)$ on $\O(\G)$ depends only on $[\phi]$.

We denote by $\Min(\phi)$ the set of points in $\O(\G)$ which are minimally displaced by
$\phi$, with respect to the Lipschitz metric, and set $\Min_1(\phi)=\Min(\phi)\cap O_1(\G)$.

By \cite{FM15}, in the case where $[\phi]$ is irreducible, this is exactly the set of points which support train 
track maps representing $[\phi]$ (this is explained in more detail in Section~\ref{sec2} and Theorem~\ref{PropertiesOfIrreducibles}). 

The main result of this paper is that if $[\phi]$ is $\G$-irreducible with
exponential growth, then   $\Min_1(\phi)$ is co-compact, under the action of the cyclic group
generated by $\phi$.

We first prove, in Section~\ref{s7.1}, our main result under the extra hypothesis that $[\phi]$ is
primitive (see Theorem \ref{CocompactnessPrimitive}), i.e. it has a train track representative
with primitive transition matrix. Then, in Section~\ref{generalautos} we drop primitivity
condition, proving:

\begin{restatable*}{thm}{Cocompactness}
  \label{CoCompactness}
  Let $[\phi]\in\Out(\G)$ be $\G$-irreducible and with $\lambda(\phi)>1$ (that is,
  a relatively irreducible automorphism with exponential growth).
  Then the action of $\langle\phi\rangle$ on $\Min_1(\phi) = \Min(\phi) \cap \O_1$ is co-compact.
\end{restatable*}

We notice that co-compactness was already known for genuine irreducible automorphisms of a free
group 
(see~\cite{HM}, for the original proof of Handel and Mosher, and~\cite{FMS2}, for a recent
elementary proof which was given by the authors). The main difference between the classical
case and our general case is that  the Culler-Vogtmann Outer Space of a free group has a
locally finite simplicial structure, while our general relative
deformation spaces are not even locally compact.

\begin{rem*}
In our companion paper \cite{FMS}, we prove that the minimally displaced set of an irreducible
automorphism of exponential growth is locally finite. It may seem quite intuitive to the
reader  that as $\Min(\phi)$ is locally finite, its co-compactness is equivalent to the
existence of a  fundamental domain contained in the union of finitely many simplices.

In fact, our strategy, and the thrust of this paper, is to show that the action is 
co-bounded and then deduce co-compactness from the local finiteness result. This seems intuitive, 
but presents some challenges since our main tool is the Lipschitz metric which is an asymmetric metric and whose properties 
can sometimes fail to be well behaved; for instance, a general relative outer space is usually a locally infinite space
and so is not locally compact. 

For the experts, one important question is - given we already have that $\Min(\phi)$ is locally finite - why is this current paper so long? There are a couple of reasons. The first, and most important, is that the proofs of this fact for Culler-Vogtmann space divide into two main arguments according to whether the irreducible automorphism is non-geometric or geometric. The second case corresponds precisely to the existence of a closed periodic Nielsen path, and then it is a well known result of \cite{BH-TrainTracks} that the corresponding automorphism is induced from a surface homeomorphism. The arguments then proceed by appeals to surface theory in the latter case and to other arguments in the non-geometric case (which are not valid for the geometric case). A direct ancestor of our result for Culler-Vogtmann space can be found in \cite{HM}, where they prove that the axis bundle is co-compact (amongst other results). As they note: 

``\textit{We shall do this only in the case where $\phi$ is nongeometric, meaning that it does not arise
	from a homeomorphism of a compact surface with boundary. The geometric case, while
	conceptually much simpler and more well understood, has some peculiarities
	whose inclusion in our theory would overburden an already well laden paper.}''

However, in the free product case the appeal to surface theory is more delicate and less obviously valid. That is, realising a homotopy equivalence of a graph as a homeomorphism of a surface is a core part of the theory for free group automorphisms going back to \cite{BH-TrainTracks}, but it is not just that this analogue is absent in the free product or relative case, there is good reason to think that it isn't entirely valid. For instance, any irreducible automorphism of exponential growth of a free group has a train track representative whose transition matrix is primitive. (In \cite{BFH-Laminations0} it is proved that any irreducible automorphism has a locally connected Whitehead graph - see also section~\ref{s5}. In \cite{Dowdall2015} this is called `weakly clean' and in that paper, Proposition B.2, it is shown that this implies clean, which means having a primitive transition matrix. See also Remark 2.10 of \cite{Mut} for a discussion of this.)

But this is no longer true for free products, since one can write examples of train track representative of irreducible automorphisms which are not fully irreducible (that is, some power is reducible) and whose transition matrices are not primitive (this cannot happen in the absolute case).   


It is possible that one can resolve these issues and introduce surfaces into the free product situation - using something like the improved relative train tracks of \cite{BHH-Laminations1}, especially when the underlying group is free, and the given automorphism is relatively irreducible with respect to a free factor system - which one can think of as the main case of interest for the free product theory. Taking that route would simplify our Section~\ref{s6} a little, which is a generalisation of similar results in \cite{BFH-Laminations0} and \cite{DahLi}, but we have chosen not to do that since it seems to us that fully taking into account the subtleties of \emph{relative surfaces} would add length and complexity in a different way. We would reiterate here that this process has not formally been done, and we feel that the subtleties and differences between free groups and free products require some caution in simply hand-waving through techniques which may not be entirely valid.

Instead, we actually deal with the geometric and non-geometric cases (where here the distinction being made is whether or not there is a closed periodic Nielsen path) at the same time. This is somewhat new, conceptually, and seems appropriate for the general free product case. (Also, Theorem~\ref{DiscretnessOfMinSet} seems interesting and important to us, and proving it for any irreducible automorphism without adding technical conditions seems worthwhile. One of our goals was to prove Lemma 2.13 of \cite{DahLi} but without the `no twinned subgroups' hypothesis, and in this we were successful.)

The second reason for the length of this current paper is that we do not pass to powers of the automorphism to make hypotheses and arguments simpler. For instance, we do not assume that our automorphisms are fully irreducible. (Proofs in the fully irreducible case are always easier, but the statements always seem to hold more generally.) This may seem a very minor difference, and in some situations, for instance in \cite{DahLi} where the goal is to prove relative hyperbolicity, passing to a power is a tame procedure. However, for more algebraic applications, like looking at centralisers, passing to powers is not benign. (Knowing something about the centraliser of a power does not yield strong information about the original one.) To take another example; in \cite{Cohen1999} there is a solution of the conjugacy problem for Dehn Twists and in \cite{KrsticLustigVogtmann}, this is extended to roots of Dehn Twists. However, one notes that the latter paper is longer than the former because the additional complexities of looking at roots are considerable. It is probably true that our arguments can be shortened by passing to suitable powers, but this would certainly invalidate our application and would definitely be a weaker result with less scope for further applications.

Ultimately, the argument proceeds as one would expect, but there are many places where the intuition one would get from the free group situation would be slightly incorrect. Thus some of our text is expository since experts familiar with the free group setting would tend to assume that many things which are true for free groups are also true for free products, and this is not always the case.

%
%
%
\end{rem*}

In the process of the proof of our main result, we obtain some new results that could
be of independent interest. For instance, we show that $\Min(\phi)$ is quasi-isometric to a line:

\begin{restatable*}{cor}{quasiline}
	\label{quasiline}
	Let $[\phi]\in\Out(\G)$ be $\G$-irreducible and with $\lambda(\phi)>1$.
	Then $\Min_1(\phi)$, equipped with the symmetric Lipschitz metric, is quasi-isometric to a line.
\end{restatable*}

\begin{rem*}
Note that this is also true with respect to the path Euclidean metric, since the Svarc-Milnor Lemma also applies for that metric. 
\end{rem*}

Moreover, we show that $\Min(\phi)$ is  undistorted in the relative outer space
(Theorem~\ref{undistorted}).

We also show that for any $\G$-irreducible $[\phi]$ with exponential growth, any point of $\Min(\phi)$ is at uniform distance from $\Min(\phi^{-1})$:

\begin{restatable*}{thm}{UnifDist} 
	\label{UnifDistFromMinSet}
	Let $[\phi]\in\Out(\G)$ be $\G$-irreducible, with $\lambda(\phi) > 1$. 
	Then there is a $D$-neighbourhood (with respect to the symmetrised stretching factor, see Section~\ref{lipschitzsection}) of $\Min_1(\phi)$ containing $\Min_1(\phi^{-1})$.

	More precisely, for any $L$ there is a constant $D$
	(depending only on $[\phi]$ and $L$) such that for any volume-$1$ point, $X$ with
        $\lambda_\phi(X)\leq L$, 
	there is a volume-$1$ point, $Y\in\Min(\phi^{-1})$ such that $\Lambda(X,Y)\Lambda(Y,X)<D$.  In particular, for any $X \in \Min(\phi)$ there is $Y \in \Min(\phi^{-1})$ such that $\Lambda(X,Y)\Lambda(Y,X) <D$.
\end{restatable*}

Another interesting result is the following.  For a $\G$-irreducible automorphism $\phi$ with
exponential growth, we can define the attracting and repelling  trees (starting from a train
track point $X$). We prove a discreteness result for the product of the limit trees.

\begin{restatable*}{thm}{Discretness}
	\label{DiscretnessOfMinSet}
  Let $[\phi]\in\Out(\G)$ be $\G$-irreducible and with $\lambda(\phi)>1$ (that is,
  a relatively irreducible automorphism with exponential growth).
	 Let $X \in \Min(\phi)$ and $Y \in \Min(\phi^{-1})$, and denote by $X_{+\infty}$ and
         $Y_{-\infty}$  the corresponding attracting tree and repelling tree for $\phi$,
         respectively 
         (Definition~\ref{stabletree}). Then there exists an $\epsilon > 0$ such that for all
         $g \in G$, either 
	\begin{itemize}
		\item $\ell_{X_{+\infty}}(g) = \ell_{Y_{-\infty}}(g) = 0$ or,
		\item $\max \{   \ell_{X_{+\infty}}(g) , \ell_{Y_{-\infty}}(g)        \} \geq \epsilon$.
	\end{itemize}
\end{restatable*}

\textbf{Applications.}
As a first application of co-compactness, we get
Theorem~\ref{RelativeCentraliser}, which describes the structure of centralisers of relatively irreducible automorphisms, in the spirit
of the classical result of Bestvina, Feighn and Handel for irreducible automorphisms of free
groups (see \cite{BFH-Laminations0}).

Also, our results for general deformation spaces, have applications to classical cases. An
example is the following fact for centralisers of elements
in $\Out(F_3)$, which was previously unknown: 

\begin{restatable*}{thm}{Centralisers}
	\label{Out(F_3)}
	Centralisers of elements in $\Out(F_3)$ are finitely generated.
\end{restatable*}	

\medskip

\noindent
\textbf{Strategy of the Proof of Theorem~\ref{CoCompactness}.}

Associated to any $\mathcal{G}$-irreducible automorphism class of exponential growth $[\phi]$ we can define the attracting (or stable) tree, $X_{+\infty}$, which is the forward limit of some point, $X_{+\infty} = \lim_{n \to \infty} \frac{X \phi^n}{\lambda(\phi)^n}$, where $X \in \Min(\phi)$. Note that since $X$ is minimally displaced, we get that the Lipschitz distance from $X$ to $\frac{X \phi^n}{\lambda(\phi)^n}$ is also 1. In fact, the Lipschitz distance from $X$ to $X_{+\infty}$ is again, 1, and this same calculation holds for any minimally displaced point and its forward limit. Thus any point in $\Min(\phi)$ is uniformly close to its forward limit. 

North-South Dynamics then tells us that all points have the same forward limit, up to positive scaling constants, except for the {\em repelling (or unstable) tree}. 

If the action of $\langle \phi \rangle$ on $\Min_1(\phi)$ is not co-compact, we can then find a sequence of points $X_n \in \Min(\phi)$ whose limit $T= \lim_{n \to \infty} X_n$ is very far (at infinite distance) from both the attracting and repelling trees for $\phi$. This leads to the following contradiction: 

\begin{itemize}
	\item $T$ is at finite distance from its forward limit $\lim_{n \to \infty} \frac{T\phi^n}{\lambda(\phi)^n}$, since $T$ is a limit of minimally displaced points, each of which is uniformly close to its forward limit. 
	\item $T$ is at infinite distance from its forward limit $\lim_{n \to \infty} \frac{T\phi^n}{\lambda(\phi)^n}$, since the latter is $cX_{+\infty}$ for some $c>0$, by North-South dynamics. 
\end{itemize} 

\medspace 
 
Implementing this argument requires a careful verification that one's geometric intuition concerning limits and distance are correct. We do this, by proving results that we believe are of independent interest along the way. 
In more detail, we start with a $\G$-irreducible automorphism class $[\phi]$ with exponential growth rate $\lambda(\phi) >1$. 
\begin{enumerate}
	\item We fix a ``basepoint", $X \in \Min_1(\phi)$ and define the attracting tree,
          $X_{+\infty} = \lim_n \frac{X\phi^n}{\lambda(\phi)^n}$, which exists due to train track properties, Lemma~\ref{lxinf}. 
	\item We argue by contradiction, and suppose that  $\Min_1(\phi)/\langle \phi \rangle$ is not compact. 
	\item We thus produce a sequence -- justified in Theorem~\ref{cocompactnessdefs} -- $X_n \in \Min_1(\phi)$, such that the distance from $X$ to the $\phi$-orbit of $X_n$ tends to infinity. The distance we use here is the Lipschitz distance (where we can use either the symmetric or non-symmetric ones, since they are equivalent on the thick part, and any point of $\Min_1(\phi)$ must be thick). 
	\item By replacing each $X_n$ with a suitable element of its $\phi$-orbit, we can assume that $1 \leq \Lambda(X_n,X_{+\infty}) \leq \lambda(\phi) $ and $\Lambda(X,X_n)$ is unbounded. (In fact, Theorem~\ref{cocompactnessdefs} has a long list of equivalent statements of co-compactness that includes this one.)
	\item As $\overline{\mathbb{P}\O(\G)}$ is compact, we may find constants $\mu_n$ and a subsequence of $X_n$ such that 
	$\lim_n\frac{X_n}{\mu_n} \to T$ (this is convergence as length functions, and occurs in
        $\overline{\O(\G)}$. In case $G$ is not countable we can use ultralimits instead of
        classical limits).  
	\item Since $T$ is the limit of points displaced by $\lambda(\phi)$, $T$ itself is displaced by at most $\lambda(\phi)$ under $\phi$, Lemma~\ref{displace}. 
	\item We then argue, in Proposition~\ref{Tfar}, that $\Lambda(T, X_{+\infty}) = \infty$, which in particular implies that $T$ is not in the same homothety class as $X_{+\infty}$.
	\item Symmetrically, we argue $T$ is not in the same homothety class as the repelling tree. However, since many aspects of the theory are not symmetrical, this requires two important ingredients: 
	\begin{enumerate}[(i)]
		\item Theorem~\ref{UnifDistFromMinSet} shows that there is a uniform distance between $\Min_1(\phi)$ and $\Min_1(\phi^{-1})$. That is, one is contained in a Lipschitz neighbourhood of the other, and so $T$ is also a limit of points which are minimally displaced by $\phi^{-1}$, even though $\Min_1(\phi)$ and $\Min_1(\phi^{-1})$ are different. (More precisely, $T$ is bi-Lipschitz equivalent to a limit of such points.)
		\item Theorem~\ref{DiscretnessOfMinSet} shows that if we have a bound on the Lipschitz distance to the attracting tree, we also get a bound on the Lipschitz distance to the (in fact, any) repelling tree. Thus $T$ is also a limit (or bi-Lipschitz equivalent to a limit) of points, minimally displaced by $\phi^{-1}$, whose distance to the repelling tree is bounded. 
		\item This is enough symmetry to conclude - Corollary~\ref{NotLimitTree} - that $T$ is not in the same homothety class as the repelling tree.
	\end{enumerate}
	\item We then apply North-South dynamics to $T$ (we need to know that $T$ is not in the same homothety class as the repelling tree for this to work), which combined with the previous results says that $\lim_n \frac{T\phi^n}{\lambda(\phi)^n}$ is both at finite distance from $T$, and in the same homothety class as $X_{+\infty}$, which is a contradiction. Hence this contradiction implies that $\Min_1(\phi)/\langle \phi \rangle$ is compact.
	\item As North-South dynamics are not available for general irreducible automorphisms, in Section~\ref{generalautos} we give an additional argument that is needed in order to deduce the co-compactness of a general irreducible automorphism, deducing it from the case of primitive irreducible automorphisms, where North-South dynamics are known to hold.
\end{enumerate}

Some of the results stated here are dependent on others in unexpected ways. For instance, the equivalent formulations of co-compactness, Theorem~\ref{cocompactnessdefs}, relies on the fact that $\Min(\phi)$ is uniformly close to  $\Min(\phi^{-1})$, Theorem~\ref{UnifDistFromMinSet}.

\medskip

The organisation of the paper is as follows: 

\begin{itemize}
	\item Section~\ref{sec2} sets up terminology and recalls known results. While this is largely known to experts, we do have some minor proofs which appear to be new (Lemmas~\ref{arcstab} and \ref{Arc Stabiliser}). 
	\item Section~\ref{sec3} is a fairly short section showing that  the minimally displaced for $\phi$ is uniformly close to that for $\phi^{-1}$, using results from \cite{FM20} and \cite{FM21}.   
    \item Section~\ref{Conditions} is devoted to proving the equivalent conditions for co-compactness, and also contains a discussion of the topologies on our deformation spaces. 
    \item Section~\ref{s5} is a short discussion on the North-South dynamics for primitive
      irreducible automorphisms. The material here is largely a verification, in this context,
      of results that are known in classical and/or less general cases. 
    \item Section~\ref{s6} is the most technical section. The goal of this section is the final ``discreteness'' Theorem~\ref{DiscretnessOfMinSet}. The proofs of this section are not used anywhere else, just the final result. 
    \item Section~\ref{s7} pulls everything together to prove co-compactness, first for the primitive irreducible case and then for the general irreducible case. 
    \item Section~\ref{s8} is devoted to applications, showing in particular that centralisers in $\Out(F_3)$ are finitely generated. 
\end{itemize}

\begin{ack} We are grateful to the referee for their careful and thorough feedback along with many very useful comments. 
	The second and third authors were supported by Leverhulme Trust Grant RPG-2018-058
        during the work of this paper. First author was partially founded by INdAM group
        GNSAGA, by ERC starting grant ``Definable Algebraic Topology'', and  by the European
        Union - NextGenerationEU under the National Recovery and 
        Resilience Plan (PNRR) - Mission 4 Education and research - Component 2 From research
        to business - Investment 1.1 Notice Prin 2022 -  DD N. 104 del 2/2/2022, from title
        "Geometry and topolgy of manifolds", proposal code 2022NMPLT8 - CUP J53D23003820001.  
\end{ack}

\section{Terminology and Preliminaries}
\label{sec2}

\subsection{Relative Outer Space $\O(\G)$}
Let $G$ be a group which decomposes as a finite free product $$G = G_1* \dots *G_k * F_r$$
where $F_r$ is the free group on $r \geq 0$ generators. We impose no restriction on the
$G_i$'s (in particular we do not assume that the
$G_i$'s are freely indecomposable nor non-cyclic, nor finitely generated, nor countable \dots). Any such free product decomposition
is commonly referred to as a {\em free factor system of $G$}. More precisely:

\begin{nt}
  A free factor system of $G$ is a pair $\G=(\{G_1,\dots,G_k\},r)$ such that $G=G_1*\dots*
  G_k*F_r$.  We define the rank of $\G$ as $\rank(\G)=k+r$. With $[\G]$ we denote the set of
  conjugacy classes of  the $G_i$'s, that is  $[\G]=\{[G_1],\dots,[G_k]\}$. If
  $\G'=(\{G_1',\dots ,G_s' \},m)$  is another free factor system, we say that $\G$ is bigger
  than $\G'$ if for any   $i$ there is $j$ such that $G_i'$ is a subgroup of some conjugate of
  $G_j$.  
\end{nt}

\begin{defn}
Let $G$ be a group.
\begin{itemize}
\item A $G$-tree is a tree $T$ together with an action of $G$. If the tree is
  simplicial (resp.  metric), then the action is supposed to be simplicial (resp. isometric).
        \item A $G$-tree $T$ is called minimal, if it has no proper $G$-invariant
          sub-tree.
	\item The action of $G$ on a $G$-tree is called marking (and a marked tree is a tree
          equipped with a $G$-action.)
        \item If $T$ is a minimal simplicial metric $G$-tree, we denote by $\vol(T)$ the  co-volume of 
          $T$, namely the sum of lengths of edges of the quotient graph $G\backslash T$. (A
          priori this number could be infinite).
\end{itemize}
\end{defn}
In this paper the $G$-action on a $G$-tree will always be a left-action. 
\begin{defn}
Let $\G=(\{G_1,\dots,G_k\},r)$ be a free factor system of a group $G$.
A simplicial $G$-tree is called (simplicial) $\G$-tree, if:
\begin{itemize}	\item $T$ has trivial edge stabilisers (that is to say, no $1\neq g\in G$ pointwise
  fixes an edge), and no inversions (that is to say, no $g\in G$ maps an edge to its inverse).
	\item The non-trivial  vertex stabilisers of $T$ are exactly the conjugacy classes that are
          contained in $[\G]$. More precisely, for every $i$ there is a unique vertex $v_i$
          with stabiliser $G_i$.  The vertices with non-trivial stabiliser will be called
          non-free vertices; the other vertices will be called free vertices. We use the
          notation $G_{v_i} = \stab_G(v_i)$, and we often refer to factor groups $G_i$'s as
          {\em vertex groups}. (Since there are finitely many $G_i$'s and since $F_r$ has
          finite rank, the volume of $\G$-trees is a finite number).

\end{itemize}
\end{defn}

\begin{defn}
  Let $\G=(\{G_1,\dots,G_k\},r)$ be a free factor system of a group $G$.
    The relative outer space of $\G$ --- denoted by $\O(\G)$ --- is the set of equivalence classes 
	of minimal, simplicial, metric $\G$-trees, with no redundant vertices (i.e. any free
        vertex has valence at least $3$),  where the equivalence relation is given by
        $G$-equivariant isometries. We denote by $\O_1(\G)$, the co-volume-$1$ subset of $\O(\G)$.

        There is a natural action of $\R^+$ on $\O(\G)$ given by $a:T\mapsto aT$   where $aT$
        denotes the same marked tree as $T$, but with the metric scaled by $a>0$.
	We denote by $\mathbb{P}\O(\G)$ the projectivised relative outer space (that is, the
        quotient of $\O(\G)$ by the $\R^+$-action).
\end{defn}

\subsection{Simplicial Structure of $\O(\G)$}
  Let $\G=(\{G_1,\dots,G_k\},r)$ be a free factor system of a group $G$, and consider  $X\in \O(\G)$. The (open) simplex $\Delta(X)$ is the set of points of $\O(\G)$ 
which are obtained from $X$ by just changing the lengths of (orbits of) edges in such a way
that any edge has positive length. Thus $\Delta(X)$ is parameterised by the positive cone
of $\R^{n}$, where $n$ is the number of orbits of edges in $X$. Note that the positive cone of
$\R^n$  can be naturally identified with an open $n$-simplex.

If we work in $\O_1(\G)$, then $\Delta(X)$ determines  a standard open
$(n-1)$-simplex $\Delta(X)_1 = \Delta(X) \cap \O_1(\G)$. We will often omit the subscript ``1''
and write just $\Delta$ or $\Delta(X)$ when it is clear from the context in which space we are working.

\begin{rem}
  The $\R^+$ action plus the parameterisations of $\Delta(X)$ and $\Delta(X)_1$ by convex
  subsets of $\R^n$, allow us to define Euclidean segments between pair of points $X,Y$ in the
  same simplex by the usual formula $tX+(1-t)Y$.
\end{rem}

\begin{rem}
  So far we have not mentioned topology, but all of the
  topologies we will consider induce the standard Euclidean topology on each simplex of $\O_1(\G)$.
\end{rem}

Simplicial faces of simplices of $\O(\G)$ do not always live inside $\O(\G)$, so the space is
not a simplicial complex. Any face of a simplex $\Delta = \Delta(X)$ in $\O(\G)$ is induced by
collapsing a $G$-invariant sub-forest of $X$. Such a collapse produces a
simplicial $G$-tree $Y$ with trivial edge stabilisers, and there are two cases: either $Y$
is a $\G$-tree or not (i.e. vertex stabilisers are not in $[\G]$). In the first case we say
that $\Delta(Y)$ is a {\em finitary face}, in the second that it is a {\em face at infinity}. We notice that faces at infinity correspond to free factor
system strictly bigger than $\G$.

\begin{rem}
  Given $T\in\O(\G)$, the quotient graph $G\backslash T$ comes endowed with a structure of graph of
  groups, and the the choice of a marking of $T$ corresponds to the choice of an isomorphism
  from $G$ to $\pi_1(G\backslash T)$ (where fundamental group is taken in the sense of graph of
  groups). The equivalence relation given by equivariant isometries of $G$-trees, translates to
  a notion of 
  equivalence of marked graphs, which is the usual one for the reader used to Teichmuller
  theory or classical Culler-Vogtmann outer space $CV_n$.  
\end{rem}

In this paper we will use only the tree-viewpoint, but in some case graphs are easier to
visualise.   For instance, one can easily see with graphs that if there is at least one $G_i$ which is infinite, then the simplicial structure of $\O(\G)$
is not locally finite.

\begin{ex}\label{exfig1}
Consider the simple case $G=G_1\ast \Z$ where $G_1$ is an infinite group. The simplex corresponding to a graph of groups formed by a circle with a unique
non-free vertex is a finitary face of infinitely many simplices corresponding to a graph formed by a
circle with a segment attached, ending with the unique non-free vertex (See Figure~\ref{Fig1}).
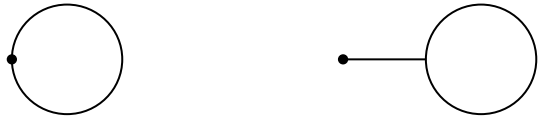
\begin{figure}[htbp]
  \centering
  \begin{tikzpicture}[x=1ex,y=1ex]
    \draw (0,0) circle[radius=4];
    \draw[fill] (-4,0) circle[radius=.3];
    \draw[fill] (20,0) circle[radius=.3];
    \draw (20,0)--(26,0);
    \draw (30,0) circle[radius=4];
  \end{tikzpicture}
  \caption{Graphs corresponding to open simplices}\label{Fig1}
\end{figure}
This is because for any $g\in G_1$, if $\Z=\langle a\rangle$, then we can define an isomorphism
$\phi_g:G\to G$ which is the indentity on $G_1$ and  maps $a$ to $ga$. It is readily checked
that all markings induced by all $\phi_g$'s on the left-side graph are equivalent, while they
are not equivalent on right-side graphs.
\end{ex}

\subsection{Action of the automorphism groups}
Let $\G=(\{G_1,\dots,G_k\},r)$ be a free factor system of a group $G$.
\begin{defn}
  The group of automorphisms of $G$ that preserve the set $[\G]$ (that is to say,
  $[f(G_i)]\in[\G]$ for all $i$) is denoted by $\Aut(G;\G)$, or simply $\Aut(\G)$. We set $\Out(\G)=\Out(G;\G) = \Aut(G;\G) / \Inn(G)$. 
\end{defn}

 There is a natural right action of $\Aut(\G)$ on $\O(\G)$, given by twisting the marking. More
 specifically, given $T \in \O(\G)$ and $\phi \in \Aut(\G)$, 
 we define the point $T\phi$ as the same metric tree as $T$, but the $G$-action on $T\phi$ is given,
 by $$x\mapsto \phi(g) \cdotp x$$ where $\cdotp$ denotes the $G$-actions on $T$. In terms of
 marked graphs this corresponds to precomposing the marking with $\phi$.

If $\alpha \in \Inn(G)$ and $T \in \O(\G)$, then it is easy to see that there is a
$G$-equivariant isometry between $T\alpha$ and $T$, i.e. they are equal as objects of
$\O(\G)$. It follows that $\Inn(G)$ acts trivially on $\O(\G)$. Thus there is an induced action
of $\Out(\G)$ on $\O(\G)$. Since the action of $\phi$ on  $\O(\G)$ depends only on
$[\phi]$, we write simply $T\phi$ to denote both actions.
 Moreover, the action preserves the co-volume of trees, so we get induced actions  on the co-volume-$1$ set $\O_1(\G)$. 

 \begin{rem}
 Since $F_r$ has finite rank, we have finitely many topological type of graphs $G\backslash T$,
 as $T$ varies in $\O(\G)$. As a consequence, there are finitely many orbits
 of simplices under the action of $\Out(\G)$. 
 \end{rem}
\subsection{Translation lengths, thickness, and boundary points}\label{slt}
Let $\G=(\{G_1,\dots,G_k\},r)$ be a free factor system of a group $G$.

For any metric $G$-tree $T$ (not necessarily in $\O(\G)$) and for any $g\in G$, we define the translation length of $g$ in $T$,
which actually depends only on the conjugacy class $[g]$, by
$$\ell_T(g) = \ell_T([g]) = \inf\{d_T(x,gx) : x \in T\}. $$

It is well known (see~\cite{CM}) that the infimum is achieved by some $x \in T$.
We have a dichotomy of elements in $G$.
If $\ell_T(g) > 0$, then $g$ is called \textit{hyperbolic} (in $T$) or $T$-hyperbolic. In this
case, the set of 
points achieving the minimum above is a line in $T$, on which $g$ acts by translations by
$\ell_T(g)$, and it is called the \textit{axis} of $g$ in $T$. Otherwise, $g$ is called \textit{elliptic} (in $T$) or $T$-elliptic .

If $T \in \O(\G)$ then elliptic elements are exactly those belonging to some
vertex group, and therefore hyperbolicity of elements does not depend on the tree $T\in\O(\G)$ but
only on $[\G]$. We denote the set of hyperbolic elements of $\G$ by $\Hyp(\G)$, and we refer to
them as $\G$-hyperbolic elements. Other elements are called $\G$-elliptic.

Let $\mathcal{C}$ be the set of conjugacy classes of elements in $G$. We can define a map $$L: \O(\G) \to \mathbb{R} ^{\mathcal{C}}$$
$$ L(T) = (\ell_T(c))_{c \in \mathcal{C}}.$$

It is proved by Culler and Morgan in \cite{CM} that in our context that map is
injective. Moreover, it induces an injective map $L:\p\O(\G)\to \p\R^{\mathcal C}$.

\begin{defn}
  The {\bf length function topology} on $\O(\G)$ and $\O_1(\G)$ is that induced by the immersion
  $L:\O(\G)\to\R^{\mathcal C}$.
\end{defn}

\begin{rem}
With respect to the length function topology, $T_n\to T$ if and only if for any $g\in G$ we
have $\ell_{T_n}(g)\to \ell_T(g)$.   
\end{rem}

It is easy to check that length function topology is Hausdorff, and agrees 
on each simplex with Euclidean one. We alert
the reader that the choice of the topology on $\p\O(\G)$ involves some subtlety, that will be
discussed in Section~\ref{topologies}. So far, the length function topology is the unique
topology we have defined.  

\begin{defn}
  We will denote by $\overline{\O(\G)}$ the closure of $\O(\G)$
  as a sub-space of $\R^{\mathcal C}$, and by $\overline{\p\O(\G)}$ the closure of $\p\O(\G)$
  as a sub-space of $\p\R^{\mathcal C}$.  
\end{defn}

It is known that
$\overline{\p\O(\G)}$ is a compact space (see~\cite{CM} for classical case and~\cite{Horbez} for
relative one). Moreover,
there is a more detailed  description of  $\overline{\p\O(\G)}$ in terms of very small trees as
follows.

\begin{defn}
Let $T$ be a metric $G$-tree such that every factor $G_i$ fixes a unique point of $T$. 
Then $T$ is called {\em small} if arc stabilizers in $T$ are either trivial,
or cyclic and not contained in any conjugate of some $G_i$. $T$ is called {\em very small} if it is small, non-trivial arc
stabilizers in $T$ are closed under taking roots, and tripod stabilizers in $T$ are trivial.
\end{defn}

\begin{thm}(Horbez, \cite{Horbez})
Let $\G$ be a free factor system of a countable group $G$,
and let $\O(\G)$ be the corresponding relative outer space.
Then $\overline{\p\O(\G)}$ is the space of projective length functions of minimal, very small trees (with repect to the free factor system $\G$).
\end{thm}

\begin{rem}
In our Arc Stabiliser Lemma~\ref{arcstab}, we prove (for completeness) that non-trivial arc stabilisers in $\overline{\O(\G)}$ are $\G$-hyperbolic, without assuming the group is countable.
\end{rem}

In analogy with Teichmuller space, we can define thick and thin parts of outer spaces.

\begin{defn} For any $\epsilon > 0$ we define the thick part $\O(\G,\epsilon)$ as the set of
  all $T \in \O(\G)$ such that all elements in $\Hyp(\G)$ have translation length more than
  $\epsilon\vol(T)$. Namely, $T \in \O(\G,\epsilon)$ if for all $g \in \Hyp(\G)$ we have $\ell_T(g)/\vol(T) > \epsilon $. We denote also by $\O_1(\G,\epsilon) = \O_1(\G) \cap \O(\G,\epsilon)$, the thick part of $\O_1(\G)$. We say that $\epsilon$ is the level of thickness (or simply the thickness) of $\O(\G,\epsilon)$.
\end{defn}

\begin{rem}
  It is immediate to see that for any simplex $\Delta$, the closure of
  $\Delta\cap\O_1(\G,\epsilon)$ is compact. Hence, since we have finitely many $\Out(\G)$-orbits of simplices,
  for any $\epsilon >0$, the quotient space $\O_1(\G,\epsilon) / \Out(\G)$ is compact.
\end{rem}

\subsection{Stretching factors and Lipschitz metrics}
\label{lipschitzsection}
Let $\G=(\{G_1,\dots,G_k\},r)$ be a free factor system of a group $G$.
For any $T \in \O(\G)$ and  $S \in \overline{\O(\G)}$, we define the (right) stretching factor as:

$$ \Lambda(T,S) = \sup_{g \in \Hyp(\G)} \frac{\ell_S(g)}{\ell_T(g)}.$$
It is immediate from the definition that  $\Lambda$ is right-multiplicative and
  left-anti-multiplicative: $$\lambda\Lambda(T,S)=\Lambda(T,\lambda S)=\Lambda(\frac{1}{\lambda}T,S).$$   

  The stretching factor is not symmetric, and in general fails to be quasi-symmetric. However, if it is restricted on any thick part $\O(\G,\epsilon)$ of $\O(\G)$, it is quasi-symmetric.

\begin{thm}[\cite{S}]\label{quasi-symmetry}
	For any $\epsilon >0$, there exists a constant $C = C(\epsilon)$ such that for all $X,Y
        \in \O_1(\G,\epsilon)$, we have $$\Lambda(X,Y) \leq \Lambda(Y,X)^C .$$ 
\end{thm}

The stretching factor can be viewed as a multiplicative, non-symmetric, pseudo-metric. It comes
with its left avatar and symmetrised version. All of them are generically referred to as
``Lipschitz metrics'' on $\O(\G)$, and have been extensively studied, for instance in 
\cite{FM15,FM12,FM11}. We list some of its basic properties.
\begin{thm}[\cite{FM15}]\label{t13}
Let $\G=(\{G_1,\dots,G_k\},r)$ be a free factor system of a group $G$, and let $\O(\G)$ its outer space. Then
\begin{enumerate}
\item $\Lambda$ is an asymmetric multiplicative pseudo-metric on $\O(\G)$, which restricts to an
  asymmetric multiplicative metric on $\O_1(\G)$:
  \begin{itemize}
  \item For all $T\in \O(\G)$, $\Lambda(T,T)=1$;  
  \item For $T,S,Q \in \O(\G) $, $\Lambda(T,S) \leq \Lambda(T,Q) \Lambda(Q,S)$;
  \item For $T,S \in \O_1(G)$, we have $\Lambda(T,S) \geq 1$, and $\Lambda(T,S) = 1$ if and only if $T = S$.
  \end{itemize}
	\item For every $T\in \O(\G)$ and  $S \in \overline{ \O(\G)}$, there is a $\G$-hyperbolic  element $g_0$ so that  $\Lambda (T,S) = \frac{\ell_S(g_0)}{\ell_T(g_0)} $.
	\item $\Out(\G)$ acts by $\Lambda$-isometries on $\O(\G)$.
        \item The symmetrised stretching factor $D(S,T)=\Lambda(S,T)\Lambda(T,S)$ satisfies the
          following. For all $T,S\in\O(\G)$
          \begin{itemize}
          \item $D(T,S)\geq 1$, and $D(T,S)=1$ if and only if there is
            $\lambda>0$ such that $T=\lambda S$;
          \item $D(T,S)=D(S,T)$;
          \item for any $Q\in\O(\G)$, $D(T,S)\leq D(T,Q)D(Q,S)$
          \end{itemize}
          In particular the function $\log D$ is a pseudo-metric on $\O(\G)$ that restricts to
          a genuine metric on $\O_1(\G)$. 
\end{enumerate}
\end{thm}

Any of these metrics induces a topology on $\O(\G),\O_1(\G)$,  and on $\p(\O(G))$ as a quotient
of $\O(\G)$, whose relation with length function topology will be discussed  in
Section~\ref{topologies}. It is however readily checked that all such topologies induces the
Euclidean one on each simplex of $\O_1(\G)$.

\subsection{Optimal maps and gate structures}
Let $\G=(\{G_1,\dots,G_k\},r)$ be a free factor system of a group $G$.

\begin{defn}
Let $X \in \O(G),Y \in \overline{\O(\G)}$. A Lipschitz continuous and $G$-equivariant map $f:X \to
Y$, is called an $\O$-map. $\Lip(f)$ denotes the best Lipschitz constant for $f$.
\end{defn}

The name ``Lipschitz metric'' when referring to stretching factor, is motivated by the fact
that $\Lambda(X,Y)$ can be viewed as the best Lipschitz constant of equivariant maps from $X$
to $Y$.

\begin{thm}[\cite{FM15,FM20}]
  For any $X,Y\in\O(G)$ we have $$\Lambda(X,Y)=\inf_f\Lip(f)$$ where $f$ runs over the set of
  $\O$-maps from $X$ to $Y$. Moreover there is at least an $\O$-map $f:X\to Y$ realising the stretching
  factor, that is, such that  $\Lambda(X,Y)=\Lip(f)$. 
\end{thm}

\begin{defn}
Let $X\in \O(\G), Y \in \overline{\O(\G)}$. An $\O$-map $f: X \to Y$ is called {\em straight},
if it is linear on edges, i.e. for any edge $e$ of $X$, there is non-negative number
$\lambda_e(f)$ so that the edge $e$ is uniformly stretched by $\lambda_e(f)$.
Given a straight map, the {\em tension graph} of $f$, is the set of maximally stretched edges:

$$
X_{\max}(f)=\{\text{edges } e : \lambda_e(f) = \Lip(f) \}.
$$
\end{defn}

\begin{defn}
Let $X \in \O(\G)$, and  let $v$ be a vertex of $X$. A turn of $X$ at $v$, is the $G_v$-orbit
of an unoriented pair of edges based at $v$.
\end{defn}

\begin{defn}
A gate structure on a simplicial metric tree $X$ is an equivalence relation on germs
of edges at vertices of $X$. If $X\in\O(\G)$, the gate structure is required to be
$G$-invariant.  Equivalence classes are called gates. Given a gate structure
$\sim$, a turn on $X$ is {\em legal}, if its germs are not in the same gate. A path in $X$ is
legal, if it crosses only legal turns. (Note that legality does depend on the chosen grate structure.)
\end{defn}

Straight maps naturally induce gate structures:
\begin{defn}
Given a straight map $f: X \to Y$, the gate structure $\sim_f$ is defined by declaring
equivalent two germs of $X$ that have the same non-collapsed image under $f$.
A turn (or a path) is called $f$-legal if it is legal with respect $\sim_f$.
\end{defn}

In case $X=Y$ there is also a different natural gate structure, that takes in account iterates,
and that will be discussed in Section~\ref{strain}.

\begin{defn}
Let $X \in \O(\G), Y \in \overline{\O(\G)}$. A straight map is called optimal, if $\Lambda(X,Y)
= \Lip(f)$ and the tension graph is at least two-gated at every vertex (with respect to
$\sim_f$). Moreover, an optimal map is called minimal, if its tension graph consists of the
union of axes of maximally stretched elements it contains.
\end{defn}

\begin{rem}\label{Optimality}
For all $X \in \O(\G)$ and $Y \in \overline{\O(\G)}$, there is always an optimal map $f: X \to Y$
(and it is usually not unique).  Moreover, there is always a minimal optimal map $f:X \to
Y$. In~\cite{FM15,FM20} these facts are proved for $Y\in\O(\G)$, but the proves clearly work without
any change for trees in $\overline{\O(\G)}$, as all technicalities take place on $X$. 
\end{rem}
\subsection{Train tracks}\label{strain}
Let $\G=(\{G_1,\dots,G_k\},r)$ be a free factor system of a group $G$.

We already seen that  straight maps $f: X \to Y$ induces a natural gate structure $\sim_f$ on
$X$. In case $X=Y$, we consider also a second natural gate structure, namely:
\begin{itemize}
	\item $\sim_f$: two germs of $X$ are $\sim_f$-equivalent, if they have the same
          non-collapsed image under $f$. 
	\item $\simfk$: two germs of $X$ are $\simfk$ if there is some $i$, so that
          they have the same non-collapsed image under $f^i$. 
\end{itemize}

Train tracks maps where introduced in~\cite{BH-TrainTracks}. The terminology we use here may
sounds different, but it is in fact equivalent. (See~\cite{FM15,FM20}).  
\begin{defn}
  Given a gate structure $\sim$ on a metric simplicial tree $X$, a train {\bf track map} $f:X\to X$ with
  respect to  $\sim$,  is a straight map such that
  \begin{enumerate}
	\item $f$ sends edges to $\sim$-legal paths, and
	\item if $f(v)$ is a vertex, then $f$ maps $\sim$-inequivalent germs at $v$ to
          $\sim$-inequivalent germs at $v$. 
\end{enumerate}
\end{defn}

It turns out  (\cite[Section 8]{FM15}) that if $f$ is train-track for some gate structure $\sim$, then in fact the
relation $\sim$ is stronger than (i.e. it contains) $\simfk$. In fact, if $f$ is
$\sim$-train track, then $f$ is $\simfk$-train track.  (Also, since $\sim_f$ is always
weaker than $\simfk$, if $f$ is $\sim_f$-train track then
$\sim_f=\simfk$.) In what follows, we generically refer to a train track map as a map
$f$, which is train track with respect to $\simfk$.

\begin{defn}\label{defin272} 
  Let $\phi\in\Aut(\G)$ and $X\in\O(\G)$. A topological representative of $\phi$ at $X$ is just an $\O$-map
  $f:X\to X\phi$. In other words, a map $f:X\to X$ such that $f(gx)=\phi(g)x$. A topological
  representative of $[\phi]\in\Out(\G)$ is a topological representative of some  $\psi\in[\phi]$. 
  A (simplicial\footnote{i.e. mapping
    vertices to vertices and edges to edgepaths}) train track representative of $[\phi]$ is a (simplicial) topological representative $f$ which is train
  track with respect to $\simfk$.
    Points admitting (simplicial) train track representatives of $[\phi]$ are called
    (simplicial) train track points of $\phi$. 
\end{defn}

\begin{rem}
  Train track representatives are always optimal maps (see~\cite{FM15,FM20}), and their
  Lipschitz constant, if bigger than one, is  the exponential growth rate of the represented
  automorphism.

  In the case where $[\phi]$ is irreducible, the tension graph of any train track representative (which always exists by Theorem~\ref{PropertiesOfIrreducibles}) is the whole graph. 
\end{rem}

\begin{rem}\label{Simplicial}
It is well known (see for instance~\cite{FM15}) that if $X$ is a train track point of 
 $\phi$, then there is a simplicial train track point $Y$ of $\phi$ such that either
$Y\in\Delta(X)$ or, at worse, $\Delta(Y)$ is a simplicial face of $\Delta(X)$. In particular, given a train track point $X$ of $\phi$, there is a simplicial train track point $Y$ of $\phi$ which is in (uniformly) bounded distance from $X$.
\end{rem}

\subsection{Bounded cancellation, critical constant, Nielsen paths}\label{s8bcc}
For any  tree $T$ and $a,b \in T$, we will denote by $[a,b]$ the unique directed reduced
(i.e. without backtracks) path in $T$  from $a$ to $b$. For a path $p$ in $T$, we denote by
$[p]$ the reduced path with same end-points of $p$. In other words, $[p]$ is ``$p$ pulled
tight''. For any reduced path $\beta=[a,b]$ in $T$, we denote by $l_T(\beta)$ its length. 

\begin{defn}[Bounded Cancellation Constant]\label{defBCC}
Given two trees $T,S$  and a continuous map $f : T \to S$, the bounded cancellation
constant of $f$, denoted by $BCC(f)$, is defined as the supremum of all real numbers $B$ with the
property that there exist $a, b, c \in T$ with $c \in [a,b]$, such that $d_{S}(f(c), [f(a),
f(b)]) = B$. 
\end{defn}

In other words, $BCC(f)$ is the best number such that for any $a,b\in T$ and $c\in[a,b]$, the
point $f(c)$ belongs to the $BCC(f)$-neighbourhood of $[f(a),f(b)]=[f([a,b])]$.

\begin{lem}[Bounded Cancellation Lemma {\cite[Proposition 4.12]{Horbez0} and\cite[Proposition 9.6]{Guirardel1998}}]\label{BCC1}
Let $\G=(\{G_1,\dots,G_k\},r)$ be a free factor system of a group $G$.  Let $T \in \O(\G)$, and $S\in  \overline{\O(\G)}$. Let $f : T \to S$ be an $\O$-map. Then $BCC(f) \leq \Lip(f)\vol(T)$.
Moreover, if $S \in \O(\G)$ we get the sharper inequality $BCC(f) \leq \Lip(f) \vol(T) -
\vol(S)$.
\end{lem}

\begin{cor}
	\label{BCC}Let $\G=(\{G_1,\dots,G_k\},r)$ be a free factor system of a group $G$.
        Let $T \in \O(\G)$, and
        $S \in\overline{\O(\G)}$. Let $f:T\to S$ be a straight map, and suppose that there is
        $\mu>0$ such that $\lambda_e(f)\geq \mu$ for any edge $e$.
        If $g\in G$ is such that its axis in $T$ can be written as a $g$-periodic concatenation of at
        most $c$ $f$-legal pieces (as, for instance, edges), then for any $B\geq BCC(f)$ we have 
	$$
	\ell_{S}(g) \geq \mu \ell_T(g) - cB.
	$$
	 In particular we can take $B = \Lambda(T, S)\vol(T)$. 
\end{cor}
\begin{proof}
This is an immediate application of Bounded Cancellation Lemma~\ref{BCC1}.
\end{proof}

\begin{defn}[Critical constant]\label{defccf}
  Given two metric trees $T,S$ and an expanding Lipschitz map $f : T \to S$ (i.e. with $\Lip(f)>1$),
  the critical constant of $f$ is defined as $cc(f)=\frac{2BCC(f)}{\Lip(f)-1}$.  
\end{defn}

\begin{lem}
For any metric tree $T$ and any expanding train track map $f:T\to T$, we have $cc(f^n)\leq cc(f)$.  
\end{lem}
\begin{proof}
  It is immediate to check by induction that $BCC(f^{n+1})\leq BCC(f)(\sum_{i=0}^n\Lip(f)^i)$, whence
  $\frac{BCC(f^{n+1})}{\Lip(f)^{n+1}-1}\leq\frac{BCC(f)}{\Lip(f)-1}$. The claim follows because,
  since $f$ is a train track map, we have $\Lip(f^{n+1})=\Lip(f)^{n+1}$.
\end{proof}

\begin{lem}\label{grow}
  Let $f:T\to T$ be a train track map defined on a metric tree $T$, with $\Lip(f)=\lambda>1$.
  Let $\gamma$ be a path in $T$, containing a legal sub-path $p$, with $l_T(p) \geq
  cc(f)$. Then for all $n>0$
  \begin{enumerate}[i)]
  \item $[f^n(\gamma)]$ contains  a legal subpath of length at least $l_T(p)$.
  \item $[f^n(\gamma)]$ contains  a legal subpath of length at least $\lambda^n(l_T(p)-cc(f))$. 
  \end{enumerate}
  In particular if $p$ is longer  than $cc(f)+1$, then $l_T(f^n(\gamma))> \lambda^n$.
\end{lem}
\begin{proof}
	Since $p$ is legal, the length of the surviving part of $f(p)$ in $[f(\gamma)]$, after
        cancellations,  is at
        least $\lambda l_T(p)-2BCC(f)=\lambda l_T(p)-cc(f)(\lambda-1)>\lambda l_T(p)-l_T(p)(\lambda-1)=l_T(p)$.

        Thus we can iterate, and we get
	$$l_T([f^n(\gamma)])>\lambda^n l_T(p)-\sum_{i=0}^{n-1}\lambda^i2BCC(f)=\lambda^n
	l_T(p)-2BCC\frac{\lambda^n-1}{\lambda-1}=$$ $$= \lambda^nl_T(p) - cc(f) \lambda^n+ cc(f)>\lambda^n( l_T(p)- cc(f)).$$
\end{proof}

\begin{defn}\label{defn287}
  Let $\G=(\{G_1,\dots,G_k\},r)$ be a free factor system of a group $G$. Let
  $X\in\O(\G)$ and $f:X\to X$ be a $G$-equivariant simplicial map.  A 
  (non-trivial) simplicial path $p$ in $X$ is called 
	\begin{enumerate}
		\item  {\em Nielsen path} (Np) if $[f(p)]=gp$ for some  $g\in G$.
		\item  {\em periodic} Nielsen path (pNp) if  $[f^n(p)]=gp$ for some $n>0$.
		\item  {\em pre-periodic} Nielsen path (ppNp) if $[f^{n+m}(p)]=gf^{m}(p)$ for
                  some $g\in G$ and integers $n,m>0$. We say that the periodic behaviour of a
                  ppNp starts before $n_0$ iterates if in the above formula we have $n<n_0$. 
		\item {\em trivial} if $[p]$ is a point, and {\em pre-trivial} if $[f^n(p)]$
                  is trivial for some positive integer $n$.
	\end{enumerate}
\end{defn}

\subsection{Candidates}Let $\G=(\{G_1,\dots,G_k\},r)$ be a free factor system of a group $G$.
As we have seen (Theorem~\ref{t13}) the  stretching factor between two trees is realised by some hyperbolic group
element. In fact, more is true.

\begin{thm}[{\cite[Theorem~9.10]{FM15} and~\cite[Lemma~7.1]{FM20}}]\label{Candidates}
For every $T
\in \O(\G)$, there is a set of hyperbolic elements $\Cand(T)$, called {\em candidates}, such
that for every $S \in \overline{\O(\G)}$ the stretching factor $\Lambda(T,S)$ is realised on a candidate, that is
$$\Lambda(T,S) = \max_{g \in \Cand(T)} \frac{\ell_S(g)}{\ell_T(g)}.$$

Moreover, the possible projections of candidates to the graph $\Gamma  =
G\backslash T$ are finitely many. Specifically, the projection of the translation axis of any
candidate has one of the following forms (possibly containing both free and non-free vertices):
\begin{itemize}
	\item A simple loop (an embedded ``O'').
	\item A figure eight, i.e. two simple loops that intersect on a point (an embedded ``8'').
	\item A non-degenerate bar-bell, i.e. a path formed by two separated simple
          loops, joined by and embedded arc (an emdedded ``O\!---\!O'').       
	\item A simply degenerate bar-bell, i.e. a path formed by a simple loop with attached
          an embedded arc ending to a non-free vertex (an embedded ``O\!---\!$\bullet$'').
        \item A doubly degenerate bar-bell, i.e. an embedded arc whose endpoints are non-free
          vertices (an embedded ``$\bullet$\!\!---\!$\bullet$''). 
\end{itemize}
\end{thm}

We will need also the following lemma.

\begin{lem}[{\cite[Theorem~4.7]{Horbez0}, see also~\cite[Lemma~2.18]{FMS}}]\label{CandFinite}
For every $T \in \O(\G)$, we can extract a {\bf finite set} from $H \subseteq \Cand(T)$, so that for every $S \in \overline{ \O(\G)}$, $$\Lambda(T,S) = \max_{g \in H} \frac{\ell_S(g)}{\ell_T(g)}.$$

Moreover $H$ does depend only on the simplex that $T$ belongs to, and not to the particular metric of $T$.
\end{lem}

\begin{cor}\label{GreenLemma}
  The stretching factor function $\Lambda: \O(\G) \times \overline{\O(\G)} \to \mathbb{R} ^{+}$
  is continuous on the second variable and lower semi-continuous on the first one, with respect to
length function topology. 
\end{cor}
\begin{proof}
  We start from   lower semi-continuity on the first variable, which does not require Lemma~\ref{CandFinite}. Let $T\in\overline{\O(\G)},X_n\in
\O(\G)$, with $X_n \to X\in\O(\G)$.

$$\liminf_{n\to\infty} \Lambda(X_n,T) = \liminf_{n\to\infty} \max_{g \in \Hyp(\G)} \frac{\ell_{T}(g)}{\ell_{X_n}(g)} \geq \max_{g \in \Hyp(\G)} \liminf_{n\to\infty} \frac{\ell_{T}(g)}{\ell_{X_n}(g)}=$$
$$
=\max_{g \in \Hyp(\G)} \lim_{n\to\infty} \frac{\ell_{T}(g)}{\ell_{X_n}(g)}= \max_{g \in \Hyp(\G)}\frac{\ell_{T}(g)}{\ell_{X}(g)} = \Lambda(X,T).$$

Now we prove
the continuity on the second variable. Let $T\in \O(\G),T_n\in \overline{\O(\G)}$, with $T_n \to T_{\infty}\in\overline{\O(\G)}$. We will show that $\Lambda(T,T_n) \to \Lambda(T,T_{\infty})$. Let denote by $H$ the finite set of Candidates of $T$ that we get from Lemma~\ref{CandFinite}. Then the following equalities hold (as $H$ is finite):

$$\lim_{n\to\infty} \Lambda(T,T_n) = \lim_{n\to\infty} \max_{g \in H} \frac{\ell_{T_n}(g)}{\ell_T(g)} = \max_{g \in H} \lim_{n\to\infty} \frac{\ell_{T_n}(g)}{\ell_T(g)}=$$
$$
= \max_{g \in H}\frac{\ell_{T_{\infty}}(g)}{\ell_T(g)} = \Lambda(T,T_{\infty}).
$$

\end{proof}
  It is easy to construct examples where continuity on first variable fails. 
\begin{ex}\label{remnotcont}
Consider graphs as in Example~\ref{exfig1} (Figure~\ref{Fig1}), with edge-lengths as
follows. On the left-side of Figure~\ref{Fig1}, the unique edge has length $1$. On the
right-side, all edges have length $1/3$. Then, 
  any infinite sequence $X_n$ of right-side graphs converges to the left-side graph $X$; this is to say that for any $g \in G$, $\lim_{n \to \infty} \ell_{X_n} (g) = \ell_X(g)$. In fact for any $g$, $\ell_{X_n}(g) = \ell_X(g)$ for all but finitely many $n$. However, 
  for every $n$ we have  $\Lambda(X_n,X)=3\neq 1=\Lambda(X,X)$.  Also, this example shows that
  the volume function in general is {\bf not} continuous with respect to the length function
  topology, as $\vol(X_n)=2/3$ while $\vol(\lim_nX_n)=\vol(X)=1$.
\end{ex}

\subsection{Displacement function and Min-Set}\label{sec210disp}
Let $\G=(\{G_1,\dots,G_k\},r)$ be a free factor system of a group $G$.
For any $\phi\in\Aut(C,\G)$  we define the displacement function with respect to $\O(\G)$ as
$$\lambda_{\phi} : \O(\G) \to \R \qquad \lambda_{\phi} (X) = \Lambda(X, X\phi) .$$
We define also the minimal $\phi$-displacement of a simplex $\Delta$ of $\O(\G)$ as
$$\lambda_{\phi}(\Delta) = \inf\{\lambda_{\phi} (X): X \in \Delta \}$$
and the minimal displacement of $\phi$ as
$$\lambda(\phi) = \inf\{\lambda_{\phi} (X) : X \in \O(\G)\}.$$
The set of minimally displaced points in $\O(\G)$ or {\bf Min-Set}, is defined as
$$\Min(\phi) = \{X \in \O(\G) : \lambda_{\phi}(X) = \lambda(\phi)\}. $$
Finally, the set of minimally displaced points with co-volume one is
$$\Min_1(\phi) = \{X \in \O_1(\G) : \lambda_{\phi}(X) = \lambda(\phi)\}. $$

We remark that the displacement actually depends only on $[\phi]\in\Out(\G)$.
We also remark that in case $[\phi]$ is reducible (see Section~\ref{secirr}) the Min-Set has to be
defined in the simplicial  bordification of $\O(\G)$, but we omit this point of view here because
in this paper we are interested in irreducible automorphisms. We just say here that in the
irreducible case, the Min-Set is  always connected, and coincides with the set of points
admitting train track representatives. We refer the interested  reader
to~\cite{FM15,FM20,FM21} for a detailed discussion on such properties. 

\subsection{Irreducible automorphisms}\label{secirr}
Let $\G=(\{G_1,\dots,G_k\},r)$ be a free factor system of a group $G$.

\begin{defn}
An element $[\phi] \in \Out(\G)$, is called $\G$-reducible (or simply reducible), if some
$\psi\in[\phi]$  admits a topological representative $f: T \to T\psi, T \in \O(\G)$, having a proper $G$-invariant,
$f$-invariant sub-forest $S$ which contains the axis of some $\G$-hyperbolic
element. $[\phi]$ is $\G$-irreducible (or simply irreducible) if it is not reducible.
\end{defn}

\begin{rem}
We can define irreducibility in terms of free factor systems. An automorphism class $[\phi]$ is
$\G$-irreducible  if $\G$ is a maximal $\phi$-invariant free factor system. (For more details,
see \cite{FM15}.) Another viewpoint of the same fact is that $[\phi]$ is reducible if and only
if there is some point $X$ in some face at infinity of some simplex  so that $\lambda_\phi(X)<\infty$. 
\end{rem}

\begin{rem}\label{Irreducibility}
  Let $G$ be a finitely generated group, and let $[\phi]\in\Out(G)$. First we note that if $\G$ is the Grushko decomposition of $G$, then $[\phi]\in\Out(\G)$. 
  Next, we observe that there always exists a free product decomposition $\G'$
  of $G$ such that $[\phi]$ is irreducible as an element of $\Out(\G')$.
  Note that in general $\G'$ is not unique. In fact, there are examples where there are
  infinitely many different spaces for which $[\phi]$ is irreducible. An example is the identity
  outer automorphism. 
\end{rem}

We summarise below some well known properties of irreducible automorphisms.

\begin{thm}[\cite{FM15,FM20}]\label{PropertiesOfIrreducibles}
  Let $[\phi]\in\Out(\G)$ be irreducible. Then:
	\begin{enumerate}
		\item It admits a train track representative $f: T \to T\phi$ (\cite[Lemma 8.17, Theorem 8.18]{FM15});
		\item the set of train track points of $\phi$ coincides with the set $\Min(\phi)$
                  of minimally displaced points (\cite[Theorem 8.19]{FM15}); 
		\item there is an $\epsilon >0$ (that depends only $\rank(\G)$ and on
                  $\lambda(\phi)$) for which $\Min_1(\phi)$ is 
                  contained in the $\epsilon$-thick part $\O_1(\G,\epsilon)$ (\cite[Proof of
                  Theorem 8.4]{FM15},\cite[Propositions 5.5 and 5.6]{FM20}). 
	\end{enumerate}
\end{thm}

Third item of Theorem~\ref{PropertiesOfIrreducibles} combined with Theorem \ref{quasi-symmetry}, implies that:

\begin{cor}\label{quasi-symmetryMinSet}
Let $[\phi]\in\Out(\G)$ be irreducible. Then there exists some constant $C=C(\phi)$, for which
for all  $X,Y \in \Min_1(\phi)$, we have $$\Lambda(X,Y) \leq \Lambda(Y,X)^C .$$
\end{cor}

The next theorem is key for our present paper:
\begin{thm}[\cite{FMS}]\label{Locally finite}
Let $[\phi]\in\Out(\G)$ be irreducible and with $\lambda(\phi) >1$. Then the simplicial
structure of $\Min(\phi)$ (as a subset of $\O(\G)$) is locally finite. In particular, the set
$\Min _1(\phi) = \O_1 \cap \Min(\phi)$ is a locally finite simplicial complex.
\end{thm}

\subsection{Primitive automorphisms}\label{TransitionMatrix}
Let $\G=(\{G_1,\dots,G_k\},r)$ be a free factor system of a group $G$.

For any $T\in\O(\G)$ and any simplicial $\O$-map $f: T \to T$, we can define the transition
matrix $M_f$ of $f$ as follows. We label orbits of edges $e_1,\dots,e_n$ and define the $(i,j)$
coefficient of $M_f$ as the number of times that $f(e_i)$ crosses the orbit of $e_j$ (in either direction).

A matrix is called {\em non-negative} if all its entries are non-negative. A non-negative matrix is
called {\em irreducible} if for any $(i,j)$ it has a power for which the $(i,j)$ entry is
positive, and it is called {\em primitive}\footnote{We notice that in~\cite{BHH-Laminations1} the
authors use the terminology {\em aperiodic} for primitive.} if it has a power so that all entries are
positive. Clearly primitive implies irreducible. It is immediate to check that $[\phi]\in\Out(\G)$ is $\G$-irreducible if and only if any
simplicial map representing $[\phi]$ has irreducible transition matrix.

\begin{defn}[Primitive Automorphism]
	An automorphism which can be represented somewhere in $\O(\G)$ by a simplicial train
        track map 
        having primitive transition matrix, is called {\em $\G$-primitive} (or simply {\em
          primitive}).
\end{defn}
      
We recall that a train track representative does not need to be simplicial and in that case, the transition matrix is not even defined. However, as we have seen before, there are always simplicial train track representatives for irreducible automorphisms.
\begin{lem}
	Suppose $[\phi]\in\Out(\G)$ is $\G$-primitive. Then for any $T\in\O(\G)$, if $f: T \to
        T$ is any simplicial train track representative of $[\phi]$, then the transition matrix
        of $f$ is  primitive.
\end{lem}
\begin{proof}
This is well known in the free case and the proof is exactly the same in the general case (see
\cite[Lemma~3.1.14]{BHH-Laminations1} for the proof).
\end{proof}

\begin{rem}
Note that a $\G$-primitive $[\phi]$ has always exponential growth,
i.e. $\lambda(\phi) >1$. Note also that for primitive automorphisms the condition on the
transition matrix is only required for train track representatives. In particular, a priori it
may happen that $[\phi]$ is primitive and reducible. We use $\G$-primitive irreducible  (or
simply primitive irreducible) to
refer to those $[\phi]$ that both $\G$-primitive and $\G$-irreducible. 

We note that in the absolute case (when $\mathcal{G}$ is trivial) that irreducibility and expanding implies primitivity. However this is not true for general free factor systems $\mathcal{G}$. 
\end{rem}

\subsection{Arc Stabiliser Lemma}\label{arcstablem}
Let $\G=(\{G_1,\dots,G_k\},r)$ be a free factor system of a group $G$.
In this section we first describe points in $\overline{\O(\G)}$ in terms of a chosen base
point, and then prove a lemma about fixed points of elliptic elements of boundary points, that
will  be used in last step of the proof of Theorem~\ref{CoCompactness}.

The standing assumption in this subsection is that any $T \in \overline{ \O(\G)}$ is the limit point of a sequence of points in $\O(\G)$. This is only certain when $G$ is countable. The issue is that $\overline{\mathbb{P}\O(\G)}$ is compact but, \textit{a priori}, may not be sequentially compact and is a subspace of a Cartesian product, which is exactly the type of topological space which can be compact without being sequentially compact. 

However, the results and proofs in this section remain true for any $G$, regardless of countability, by replacing sequences with nets (in particular,
Lemmas~\ref{Arc Stabiliser} and~\ref{arcstab} are true in general). In order to help the
reader, we give the version of this argument for sequences and explain in Remark~\ref{nets} how
to extend it to the general case. (See also~\cite{Horbez} for the countable case).

Moreover, we only use Lemmas~\ref{Arc Stabiliser} and \ref{arcstab} in the case where the tree $T$ is in fact the limit point of a sequence (Lemma~\ref{lxinf} and Theorem~\ref{CocompactnessPrimitive}).

\medskip

Let $X\in\O(G)$ be a fixed reference point. Let $T \in \overline{ \O}$ and let  $Y_n \in
\O(\G)$ be a sequence that  converges projectively 
to $T$, i.e. there is a sequence of positive numbers $\mu_n >0$ so
that: $$\lim_{n\to \infty}\frac{Y_n}{\mu_n} = T$$ 
(with respect to the length function topology).
Let $f_n : X \to Y_n$ be optimal maps. We define $$d_n: X \times X \to \mathbb{R},  \text{ where } $$
$$ d_n(x,y) =  \frac{d_{Y_n}(f_n(x),f_n(y))}{\mu_n}.$$

As $\frac{Y_n}{\mu_n}$ converges to $T$, by Corollary~\ref{GreenLemma} we get that the sequence  $\frac{\Lambda(X,Y_n)}{\mu_n}$ converges to $\Lambda(X,T)$.
In particular, this implies that  $\frac{\Lambda(X,Y_n)}{\mu_n}$ is a bounded sequence and
hence, since $\Lip(f_n)=\Lambda(X,Y_n)$,  we have
$$
0 \leq d_n(x,y) \leq \frac{\Lambda(X,Y_n)}{\mu_n} d_X(x,y) \leq \sup_n\{ \frac{\Lambda(X,Y_n)}{\mu_n} \}  d_X(x,y).
$$
At this point we would like to take a limit of the $d_n$, but boundedness is not enough to
guarantee a pointwise converging subsequence. The easiest way to do this is to
take an ultralimit (or $\omega$-limit) (see \cite{Gromov}, for the definitions and the
properties of ultralimits and ultrafilters).  We briefly recap here. 

\begin{defn} \label{omega}
	A non-principal ultrafilter on $\N$, is a function from the powerset of $\N$, $\omega: \mathcal{P}(\N) \to \{ 0,1\}$ such that: 
	\begin{itemize}
		\item $\omega(\emptyset) =0, \omega(\N)=1$
		\item $\omega(A \sqcup B) = \omega(A)+ \omega(B)$ (i.e. $\omega$ is additive on
                  disjoint subsets of $\N$) 
		\item $\omega$ is zero on any finite subset of $\N$.
	\end{itemize}
\end{defn} 

\begin{rem}
	For the reader unfamiliar with this point of view, the second point above is crucial,
        and we emphasise that $\omega$ can only take values $0$ and $1$, so the additivity on
        disjoint subsets is a strong restriction. Informally, we think of the subsets $A$  with
        $\omega(A)=1$ as ``big'' and the others small. It is then straightforward to see that
        there do not exist two disjoint subsets that are both big, and that the intersection of any
        two big subsets is again big.  
\end{rem}

Limits are then defined as follows. 

\begin{defn}\label{omegalimits} Let $\omega$ be a non-principal ultrafilter  on $\N$. 
	For any sequence  $(a_n)_{n\in \N}$ of real numbers, we say that $l\in\R$ is the
        $\omega$-limit of $a_n$ --- and we write $\lim_{\omega} a_n = l$ --- if for every $\e > 0$,
        the set $\N_{\e}= \{ n \in \N\ : \ |a_n - l| < \e \}$ satisfies $\omega(\N_{\e})=1$.  
\end{defn}

We then get the following facts whose proofs are left to the reader: 

\begin{prop} \label{propertiesofultralimits}
	Let $\omega$ be a non-principal ultrafilter on $\N$. Let $a_n$ and $b_n$ be sequences
        of real numbers.  
	\begin{itemize}
		\item If $\lim_n a_n = l$ then $\lim_{\omega} a_n = l$;
		\item If $\lim_{\omega} a_n$ exists, it is unique (it may depend on $\omega$);
		\item If $a_n$ is bounded, then $\lim_{\omega} a_n$ exists;
		\item The usual algebra of limits is valid for ultralimits: $\lim_{\omega} (a_n
                  \pm b_n) = (\lim_{\omega} a_n) \pm (\lim_{\omega} b_n)$, $\lim_{\omega}
                  a_n\cdot b_n = (\lim_{\omega} a_n)\cdot (\lim_{\omega} b_n)$;
                  \item If $a_n \geq b_n$ then $\lim_{\omega} a_n \geq \lim b_n$; 
                \item ultralimits commute with finite $\max$ and $\min$: $\lim_\omega
                  (\max\{a_n,b_n\})=\max\{\lim_\omega a_n,\lim_\omega b_n\}$.
	\end{itemize}
\end{prop}

To proceed, we apply this in our situation. We let $\omega$ be a non-principal ultrafilter on $\N$ and define:

$$
d^+(x,y) = \lim_{\omega} d_n(x,y).
$$

It is clear that this is an equivariant pseudo-metric on $X$, and we can study the associated quotient space $X^+$. Here, elements of $X^+$ are ``balls of radius $0$''. That is, elements of $X^+$ are equivalence classes $[x]:= \{ y \in X \ : \ d^+(x,y)=0 \}$.
(It may be worth to note here that $X^+$ is not a priori uniquely determined, as it depends on
the chosen ultrafilter).

\begin{lem}
	Let $\omega$ be a non-principal ultrafiler on $\N$. For all $x,y\in X$ let
	$d^{+} (x,y) = \lim_{\omega} d_n(x,y) $	
	and let	
	$X^{+} := (\frac{X}{(d^{+} = 0)}, d^+)$ be the corresponding quotient space. Then:
	\begin{enumerate}[(i)]
		\item The quotient map $X \to X^+$ is Lipschitz continuous and equivariant;
		\item $(X^+,d^+)$ is an $\R$-tree with an isometric $G$-action, $g[x]= [gx]$;
		\item $X^+$ and $T$ admit the same length function; $\ell_{X^+} = \ell_T$. In particular, $X^+$ is non-trivial and its minimal invariant subtree is equivariantly isometric to $T$.
	\end{enumerate}
\end{lem}
\begin{proof}
	For any $x,y \in X$, we have $$d_n(x,y) \leq \frac{\Lambda(X,Y_n)d_X(x,y)}{\mu_n}.$$
	It follows that $d^+([x],[y]) \leq K d_X(x,y)$, where $K = \Lambda(X,T)$ (by
	Corollary~\ref{GreenLemma}). Thus the quotient map  $X \to X^+$ is $K$-Lipschitz, and equivariance is straightforward.
	
	One can then see that $X^{+} = \frac{X}{(d^+ = 0)}$ is a path connected $0$-hyperbolic metric
	space, hence an  $\mathbb{R}$-tree equipped with an isometric $G$-action.

	It remains to show that $\ell_{X^+} =\ell_T$. For any tree $Z$, it is known that $\ell_Z(g) = \max \{0, d_Z (p,g^2 p) -d_Z(p,gp) \}$ for \textit{any} $p \in Z$ (see \cite{CM}, for more details).
	Let $x \in X, g \in G$. Then
	$$
	\begin{array}{rcl}
		\ell_{X^+}(g) & = & \max \{ 0,   d^+(x,g^2 x) - d^+(x,gx)    \} \\ \\
		& = & \max \{ 0,  \lim_{\omega} (d_n(x,g^2x) - d_n(x,gx))   \} \\ \\
		& = & \max \{ 0,   \lim_{\omega} \frac{d_{Y_n}(f_n(x),f_n(g^2x)) - d_{Y_n}((f_n(x)),f_n(gx))}{\mu_n}  \}\\ \\
		& = & \max \{ 0, \lim_{\omega} \frac{d_{Y_n}(f_n(x),g^2f_n(x)) - d_{Y_n}(f_n(x),gf_n(x))}{\mu_n} \} \\ \\
		& = & \lim_{\omega} \frac{l_{Y_n}(g)}{\mu_n} = \lim_{n}\frac{l_{Y_n}(g)}{\mu_n}  \\ \\
		& = & \ell_T(g).
	\end{array}
	$$
	
\end{proof}

\begin{lem}[Arc Stabiliser Lemma]
\label{arcstab}
Let $T \in \overline{\O(\G)}$.
If $1 \neq g\in G$ is $\G$-elliptic, then $g$ fixes a unique a point of $T$. In particular no
$\G$-elliptic, non-trivial element, point-wise fixes a non-trivial arc in $T$.
\end{lem}
\begin{proof}
As above,  $X$ is our reference point. We let $Y_n \in \O(\G)$ and $\mu_n >0$ be such that
$\frac{Y_n}{\mu_n} \to T$, and build $X^+$ as a $\omega$-limit of metrics $d_n$ on $X$. Since the minimal invariant subtree of $X^+$ is equivariantly isometric to $T$, it is sufficient to prove the theorem for arc stabilisers in $X^+$. 

Since edge stabilisers are trivial in $\O(\G)$, $g$ fixes a unique point $p$ in $X$, and a
unique point $p_n$ in each $Y_n$. Thus 
\begin{eqnarray}
  \label{eq:111}
  \forall w\in X,\ d_X(w,gw)=2d_X(w,p)\qquad  \forall w\in Y_n,\ d_{Y_n} (w,gw) = 2 d_{Y_n} (w,p_n).
\end{eqnarray}

We claim that $[p]$ is the unique point of $X^+$ fixed by $g$. Let's suppose that $[u] \in X^+$
is fixed by $g$. Thus $d^+ (u,gu) = 0$.

By equivariance of the maps $f_n$, we get that $f_n(p) = p_n$.
Then, by~\eqref{eq:111}
$$2d_n(u,p) = 2d_{Y_n} (f_n(u), p_n)/\mu_n = d_{Y_n}(f_n(u), gf_n(u))/\mu_n = d_n(u,gu) \to_{\omega} d^+(u,gu)=0.$$

As a consequence, $d^+(u,p)=0$ and so $[u]=[p]$ in $X^+$.
\end{proof}

\begin{lem}\label{Arc Stabiliser}
Let $T \in \overline{\O(\G)}$. Let $H \leq G$ be $T$-elliptic. Suppose that $H$ contains a $\G$-elliptic subgroup $A \neq 1$. Then $\operatorname{Fix}_T (H)  = \operatorname{Fix}_T (A) = \{v\}$, where $v$ is a point of $T$.
\end{lem}

\begin{proof}
	Let $1 \neq a \in A$. The group element $a$ fixes a unique point of $T$, by Lemma \ref{arcstab}.
	Given any $h \in H$, the subgroup $\langle a,h\rangle$ fixes a point $v$ of $T$ (by a
        well known result of Serre, see \cite{Serre}). Therefore, $\operatorname{Fix}_T(\langle
        a,h\rangle)
        = \operatorname{Fix}_T(\langle a\rangle) = \{v\}$, for all $h \in H$, and hence $\operatorname{Fix}_T(H)=\{ v\}$.
\end{proof}

\begin{rem}
	Note that here we are calling a subgroup elliptic if the restriction of the length function is zero on the subgroup. This is weaker than the definition - often used - that the subgroup fixes a point in the tree, although they coincide for finitely generated subgroups.
\end{rem}

\medskip

\begin{rem}
	\label{nets}
	In order to prove Lemma~\ref{arcstab} when $G$ is not countable, we argue via nets. The reason this is necessary is that we cannot guarantee that every point in $\overline{\O(\G)}$ is the limit point of a sequence. We only really use it when $T$ \textit{is} the limit point of a sequence. However, it is true in general with substantially the same proof via nets. 
	
	Concretely, one takes the directed set ${\operatorname{Fin_G}}$ of all finite subsets of $G$ ordered by inclusion. A net is then a function from ${\operatorname{Fin_G}}$ to our space and takes the place of sequences. One striking aspect of working with $\overline{ \O(\G)}$ is that we can use this set ${\operatorname{Fin_G}}$ as the universal indexing set. This is due to the fact that $\overline{ \O(\G)}$ is a subset of a Cartesian product whose indexing set is $G$, and so basic open subsets are described by finite subsets of $G$ along with open subsets of $\R$ in the corresponding coordinates.

	For this reason, a net for us will always be a function from ${\operatorname{Fin_G}}$ to our space; either $\overline{ \O(\G)}$, or the composition of such a function with one of the projection maps, resulting in a function from ${\operatorname{Fin_G}}$ to $\R$.

	A \textit{tail} of ${\operatorname{Fin_G}} $, is a subset $Tail_F = \{ E \in {\operatorname{Fin_G}} \ : \ F \subseteq E\}$ for some $F \in {\operatorname{Fin_G}}$. A net to $\R$, $x:{\operatorname{Fin_G}} \to \R$ then has limit $l$ if every open set around $l$ contains the image of a tail. Concretely this means that for every $\e > 0$, there exists an $F \in {\operatorname{Fin_G}}$ so that $|x(E) - l| < \e$ for every $F \subseteq E \in {\operatorname{Fin_G}}$. In this case we write, $x(.) \to l$. It is important to note here that the intersection of finitely many tails is again a tail; this is used repeatedly throughout.

	Similarly, a net $Y: {\operatorname{Fin_G}} \to \overline{ \O(\G)}$ has a limit $T$ if $\ell_{Y(.)}(g) \to \ell_T(g)$ for all $g \in G$. This is equivalent to saying that every open set in $\overline{ \O(\G)}$ containing $T$ contains the image of a tail.

	One can now see that every $T \in \overline{ \O(\G)}$ is the limit of a net of points - indexed by ${\operatorname{Fin_G}}$ -  in $\O(\G)$. For instance, one can do the following: for every $ F \in {\operatorname{Fin_G}} $, consider the basic open sets 
	$$
	B_F(T) = \{ S \in \overline{ \O(\G)} \ : \ |\ell_S(g) - \ell_T(g)| < 1/n \},
	$$
	where $n = |F|+1$. For each $F$, choose any $Y_F \in B_F(T) \cap \O(\G)$. Notice that if $F \subseteq E$, then $B_E(T) \subseteq B_F(T)$.	Since any open set containing $T$ contains some $B_F(T)$, we immediately see that $Y_F \to T$.

	We then introduce a ultrafilter $\omega$ on ${\operatorname{Fin_G}}$ which, as above, is a function $\omega: \mathcal{P}({\operatorname{Fin_G}}) \to \{ 0,1\}$. However, the condition we require this time is that $\omega$ is 1 on every tail - being zero on finite sets is not quite enough. Instead we require that $\omega(Tail_F)=1$ for every $Tail_F$.
	
	Concretely, the tails of ${\operatorname{Fin_G}}$ are intersection-closed and so form a filterbase. The set of supersets of tails then forms a filter, and we choose any maximal filter $\omega$ containing that filter; this will be an ultrafilter with the properties we require. This can be achieved via Zorn's Lemma. 
	
	One then defines $\omega$ limits as above, putting $\lim_{\omega} x(.) = l$ if we have for every $\e > 0$, $\omega(\{ F \in {\operatorname{Fin_G}} \ : \ |x_F-l| < \e\})=1$. As before, $x(.) \to l$ implies that $\lim_{\omega} x(.) = l$ and every bounded sequence of reals has a unique $\omega$ limit (depending on $\omega$). The rest proceeds in the same way. 
\end{rem}

\subsection{Limit trees for irreducible automorphisms}\label{LimitTreeSection}
Let $\G=(\{G_1,\dots,G_k\},r)$ be a free factor system of a group $G$.
The theory of attracting and repelling trees of a fully irreducible automorphism is well studied in
the free case. We will see in North-South dynamics Theorem~\ref{North-South1}, that such trees
exist and they are the unique fixed points, in $\overline{\O(\G)}$, of a primitive irreducible
automorphism. 

In this section we recall the construction of the attracting tree for any irreducible
automorphism with exponential growth (not necessarily primitive), starting from minimally
displaced points. In fact, the only property of irreducibility that is used here, is the
existence of train track representatives (Theorem~\ref{PropertiesOfIrreducibles}). We prove also some useful properties which can be proved using just the train track properties. Note that the repelling tree of $\phi$ will be just the attracting tree of $\phi^{-1}$, so we will focus here only on the attracting tree.

\begin{lem}[Attracting tree]\label{lxinf}
  Let $[\phi]\in\Out(\G)$ be such that there exists $X\in \Out(\G)$ supporting a train track
  map $f:X\to X$ representing $[\phi]$, with $\Lip(f)=\lambda>1$ (for example if $[\phi]$ is $\G$-irreducible with exponential growth,  and $X \in \Min(\phi)$).
Then the following  limit exists:
        $$
	X_{+\infty} = \lim_n \frac{X\phi^n}{\lambda^n}.$$
That is to say, for any $g \in G$ the following limit exists:
	$$\ell_{X_{+\infty}} (g) = \lim_n \frac{\ell_X(\phi^n(g))}{\lambda^n}.
	$$
\end{lem}
\begin{proof}
As in the construction of $X^+$ in Section~\ref{arcstablem}, we take:
	$
	d^+(x,y)= \lim_n \frac{d_X(f^n(x), f^n(y))}{\lambda^n},
	$
	and 
	$
	X_{+\infty} = \frac{X}{d^+=0}.
	$
Train track properties ensure convergence on the nose,  without use of an ultralimit. 
\end{proof}

\begin{defn}\label{stabletree}
	Let $[\phi]\in\Out(\G)$ be $\G$-irreducible with exponential growth, and let $X \in
        \Min(\phi)$. We define the {\em  attracting tree} or {\em stable tree} of $\phi$
        (with respect to $X$), as: 
        $$
	X_{+\infty} = \lim_{n\to \infty} \frac{X\phi^n}{\lambda(\phi) ^n}.
	$$
Similarly, we define the {\em repelling tree} or {\em unstable tree} of $\phi$, with respect to 
$Y\in\Min(\phi^{-1})$, as $Y_{-\infty}=\lim_{n\to +\infty}Y\phi^{-n}/\lambda(\phi^{-1})^n$. 
\end{defn}

For any $X\in\Min(\phi)$, the following facts are straightforward:

\begin{itemize}
\item $\forall Y\in\O(\G)$, $\Lambda(Y,X_{+\infty}) = \Lambda(Y\phi, X_{+\infty}\phi)$ (as $\phi$ acts by isometries).
\item $X_{+\infty} \phi= \lambda(\phi) X_{+\infty}$ (easy application of train track
  properties).
\item 	$\Lambda(X,X_{+\infty}) = 1$ (by continuity of $\Lambda$ on the second variable, Corollary~\ref{GreenLemma}).
\end{itemize}

\begin{prop}[Stable map]
	\label{stablemap}
	Let $[\phi]\in\Out(\G)$ be $G$-irreducible, of exponential growth rate
        $\lambda=\lambda(\phi) >1$. Let $X,Y \in \Min(\phi)$ and let $X_{+\infty}$ be the
        attracting tree with respect to $X$. Let $\f: Y \to Y$ be a train track representative
        of $[\phi]$. Then there is a minimal optimal map  $f_Y$ from $Y$ to $X_{+\infty}$, called the
        \textit{stable map}, with the following properties: 
	\begin{itemize}
        \item Any $f_Y$-legal periodic line $\gamma$ in the tension graph of $f_Y$, is
          $\f$-legal, and $\f(\gamma)$ is again $f_Y$-legal.
        \item The tension graph of $f_Y$ is $Y$.
        \item If $e$ is an edge of $Y$ then for any positive integer $n$, any subpath
                  of $\f^n(e)$ is $f_Y$-legal. 
	\end{itemize}
\end{prop}
\begin{proof}
	There is a minimal optimal map $f_Y: Y \to X_{+\infty}$ (Remark~\ref{Optimality}). In
        particular, $\Lip(f_Y)=\Lambda(Y, X_{+\infty})$. Moreover, since the tension graph of
        optimal maps is everywhere at least two-gated, there is some $g\in \Hyp(G)$ whose axis
        is $f_Y$-legal and contained in the tension graph of $f_Y$. 

        Let $g\in G$ be one of such $\G$-hyperbolic element. Then
        $\ell_{X_{+\infty}}(g)=\Lambda(Y,X_{+\infty})\ell_Y(g)$. On the other hand, since
        $Y\in\Min(\phi)$, we have $\ell_{Y}(\phi(g))\leq \lambda\ell_Y(g)$, with equality precisely
        when $g$ is $\f$-legal. Combining these facts, we have:
	
	$$
	\begin{array}{ccc}
	\ell_{(Y\phi)}(g)  &  \leq & \lambda \ell_{Y}(g)\\ \\
          \ell_{(X_{+\infty}\phi)} (g) & = & \lambda \ell_{X_{+\infty}}(g)\\ \\
	\frac{\ell_{X_{+\infty}}(g)}{\ell_{Y}(g)} & = & \Lambda(Y, X_{+\infty}) \\ \\
	\frac{\ell_{(X_{+\infty}\phi)} (g)}{\ell_{(Y\phi)}(g)} & \leq & \Lambda(Y\phi,
                                                                        X_{+\infty}\phi). 
	\end{array}
	$$
	
	Hence,
	$$
	\Lambda(Y, X_{+\infty})     =    \frac{\ell_{X_{+\infty}}(g)}{\ell_{Y}(g)} =
        \frac{\lambda \ell_{X_{+\infty}}(g)}{\lambda \ell_{Y}(g)} \leq
        \frac{\ell_{(X_{+\infty}\phi)}(g)}{\ell_{(Y\phi)}(g)} \leq  \Lambda(Y\phi, X_{+\infty}\phi)
        = \Lambda(Y, X_{+\infty}),
	$$
	whence we have equality throughout, and in particular
        $\ell_Y(\phi(g))=\lambda\ell_Y(g)$. Hence $f_Y$-legal axes in the tension graph of
        $f_Y$ are also $\f$-legal. Now, since the axis $\gamma$ of $g$ is $\f$-legal, and $\f$
        is train track, $\f^n(\gamma)$ remains $\f$-legal, and we have
        $$
	\Lambda(Y, X_{+\infty})= \frac{\ell_{X_{+\infty}}(g)}{\ell_{Y}(g)}=
        \frac{\lambda^n \ell_{X_{+\infty}}(g)}{\lambda^n \ell_{Y}(g)}
        = \frac{\ell_{X_{+\infty}}(\phi^n(g))}{\ell_{Y}(\phi^n(g))}  
	\leq \Lambda(Y, X_{+\infty}),
	$$
	whence the inequality is an equality and the axis of $\phi^n(g)$ --- which is
        $\f^n(\gamma)$ because of $\f$-legality--- is $f_Y$-legal and in the tension graph of $f_Y$.
        	To prove that the tension graph of $f_Y$ is the whole of $Y$, observe that
        $\cup_n\f^n(\gamma)$ is clearly $\phi$-invariant, so it must be the whole
        $Y$.

	The last claim now follows from the previous ones; as every edge can be extended to an
        $f_Y$-legal periodic line, which is $\f$-legal and all of whose iterates under $\f$ are
        both $f_Y$-and $\f$-legal. 
\end{proof}

In the next proposition we prove that the homothety class of the attracting tree doesn't depend on the train track point that we chose as base point.
\begin{prop}
  \label{weakns}
  Let $[\phi]\in\Out(\G)$ be $\G$-irreducible with exponential growth, and let  $X, Y
  \in\Min(\phi)$. Let $X_{+\infty}$ be the attracting tree for $X$, and $Y_{+\infty}$ be that
  for $Y$. Then $$X_{+\infty} = \Lambda(Y, X_{+\infty}) Y_{+\infty}.$$
\end{prop}

\begin{proof} Let $f_Y:Y\to X_{+\infty}$ be the stable map given by
  Proposition~\ref{stablemap} (in particular $f_Y$ stretches any edge by
  $\Lip(f_Y)=\Lambda(Y,X_{+\infty})$). Let $g\in G$, and represent it as a path in $G\backslash Y$ with
  $n_g$ edges. $n_g$ can be zero, for instance if $g$ is elliptic. Then by
  Proposition~\ref{stablemap}, $\phi^n(g)$ is represented as a concatenation of at most $n_g$ $f_Y$-legal pieces. Hence, by Corollary~\ref{BCC},
	$$
	\Lambda(Y, X_{+\infty}) \ell_Y(\phi^n(g)) - n_gB \leq \ell_{X_{+\infty}}(\phi^n (g)) \leq  \Lambda(Y, X_{+\infty}) \ell_Y(\phi^n(g)),
	$$
  	where $B$ is the bounded cancellation constant of $f_Y$, and the second inequality just follows from the definition of $\Lambda(Y, X_{+\infty})$.
		It follows that
	$$
	l_{X_{+\infty}}(g) = \lim_{n \to \infty} \frac{l_{X_{+\infty}}(\phi^n(g))}{\lambda^n} = \Lambda(Y, X_{+\infty}) \lim_{n \to \infty} \frac{l_{Y}(\phi^n(g))}{\lambda^n} = \Lambda(Y, X_{+\infty})l_{Y_{+\infty}}(g).
	$$

\end{proof}

Note that the uniqueness of  limit trees is a direct corollary of Theorem~\ref{North-South1} under the extra assumption of primitivity of the matrix, but the previous proposition provides us an exact description of the un-projectivised limits in the general irreducible case.

\subsection{Relative boundaries and laminations}  
Let $\G=(\{G_1,\dots,G_k\},r)$ be a free factor system of a group $G$.
For any metric tree $T$, we agree that:
\begin{itemize}
        \item a {\em half-line in $T$} is an isometric embedding $[0,\infty) \to T$; 
	\item $\overline{T}$ is the metric completion of $T$;
	\item $\partial_{\infty} T$ is the Gromov Boundary of $T$, i.e. the set of half-lines
          in $T$ up to the equivalence relation $\sim$, where two half-lines $L \sim L'$
          if and only if $L,L'$ differ only on a compact set; 
	\item $V_{\infty}(T)$ is the collection of vertices of $T$ with infinite valence (if
          $T\in\O(\G)$, this coincides with non-free vertices with infinite stabiliser); 
	\item $\partial T = \partial_{\infty} T \cup V_{\infty} (T)$;
	\item $\hat{T} = T \cup \partial T$;
	\item $\partial^2 T = \partial T \times \partial T \setminus \{(P,P) : P \in \partial T
          \} $; 
        \item a {\em direction} based at a point $P$ of $\hat{T}$, is a connected component $\hat{T}
          \setminus \{P\}$;
        \item  the {\em observer's topology} of $\hat{T}$ is the topology generated by the set of
          directions.  
\end{itemize}

It is easy to see that $\hat{T}$ is a compact set, equipped with the observer's topology. Moreover, $\partial T$ is a closed subset of $\hat{T}$ and therefore compact.
The following lemma shows that the boundary does only depend on $\G$ and not on the chosen tree $T \in \O(\G)$.
\begin{lem}[{\cite[Lemma~2.2]{GH}}]\label{lnatl}
Let $T,S \in \O(\G)$. Then any $G$-equivariant map $f : T \to S$ has a unique continuous
extension $\hat{f} : \hat{T} \to \hat{S}$. Moreover, the restriction map $h :=
\hat{f}|_{\partial T}$ is a natural homeomorphism $\partial T \to  \partial S$ (it does not depend on $f$) with $h(\partial_{\infty} T) = \partial_{\infty} S$ and $h(V_{\infty} (T)) = V_{\infty} (S)$.
\end{lem}

Therefore, the notions of $\partial (G, \G)$, $\partial_{\infty} (G,\G)$, $\partial^2(G,\G)$, $V_{\infty} (G,\G)$ can
be naturally defined as $\partial T$, $\partial_\infty T$, $\partial^2T$, $V_\infty(T)$ for a
$T\in\O(\G)$. Note that $\partial_{\infty} (G, \G)$ can be identified with the set of simple
infinite words in the free product length given by $(G,\G)$. 

In particular, for any $\G$-hyperbolic group element $g \in G$, we can define the infinite word
$g^{+\infty} = \lim_{n \to +\infty} g^n$ and $g^{-\infty}= \lim_{n \to +\infty} g^{-n}$. In
this case, $(g^{-\infty}, g^{\infty}) \in \partial^2 (G, \G)$.

There is a natural
$\Z_2$-action on $\partial^2(G,\G)$ given by flipping coordinates $(P,Q)\mapsto (Q,P)$. 

\begin{defn}
An algebraic lamination is a closed $G$-invariant, flip-invariant, subset of
$\partial^2(G,\G)$. Elements of $\partial^2(G,\G)$ are called {\em algebraic leaves}.  Given
$T\in\O(\G)$, a (bi)(infinite) line $L$ in $T$ represents an algebraic leaf $(P,Q)\in\partial^2(G,\G)$ if its
endpoints correspond to  $(P,Q)$ under the natural homeomorphism given by Lemma~\ref{lnatl}.
\end{defn}

\subsection{Attracting and repelling laminations}
Let $\G=(\{G_1,\dots,G_k\},r)$ be a free factor system of a group $G$.
Attracting and repelling laminations for irreducible automorphisms with exponential growth, can
be defined as in the classical case (see~\cite{BFH-Laminations0}  for the free case).
Classical proofs work also in the present case, as they are based only on the properties of
train-track maps.

More precisely, let $[\phi]\in\Out(\G)$ be $\G$-irreducible with $\lambda(\phi) >1$. Let $f:T \to
T$ be a train track representative of $[\phi]$, and let $e$ be an edge of $T$. Consider iterates
$f^n(e)$ and group elements $g_n\in G$ such that $g_nf^n(e)$ intersects a fixed fundamental
domain for the $G$-action on $T$. Then the limit of $g_nf^n(e)$ is a line in $T$, hence it
represents an algebraic leaf $L\in\partial^2(G,\G)$. The attracting (or stable) lamination
$\Lambda^+_\phi$ is defined as the closure of the $G$-orbit of $L$. Any line in the $G$-orbit
of $L$ is called a {\em generic line} of $\Lambda^+_\phi$.

This construction depends a priori on $T,f,e,g_n$. In fact, when 
$[\phi]$ is $\G$-primitive, it does not depend on the choices made
(see~\cite[Section~1]{BFH-Laminations0} for the proof in the free case).
We define the repelling lamination of $\phi$ as the attracting lamination of $ \phi^{-1}$, and
is denoted by $\Lambda^-_{\phi} := \Lambda^+_{\phi^{-1}}$. 

\begin{defn}
  We say that (the conjugacy class of) a subgroup $A<G$ carries $\Lambda_\phi^+$ if there is
  $T\in\O(\G)$ containing a minimal $A$-tree, which contains a line that realises a generic leaf of
  $\Lambda_\phi^+$. 
\end{defn}

\section{The distance of points of $\Min(\phi)$ from $\Min(\phi^{-1})$ is uniform}

\label{sec3}

Let's fix  a free factor system $\G=(\{G_1,\dots,G_k\},r)$ of a group $G$.
In this section, we prove a result which could be of independent interest. More specifically,
we show that if $[\phi]$ is irreducible and $\lambda(\phi)>1$, then the distance of a point of $\Min(\phi)$ from the
set $\Min(\phi^{-1})$ is uniformly bounded, by a constant depending only on $\lambda(\phi)$
(and on the dimension of the space).

\subsection{Transition vectors and spectrum discreteness}
Let $\Delta$ be a simplex of
$\mathcal O(\G)$. Let's denote by $e_1,\ldots,e_n$ the directed (orbits of) edges in
$\Delta$, and denote by $E_i$ the inverse of $e_i$, $i=1,\ldots,n$.

Let $g \in \Hyp(\G)$. If $X \in \Delta$, then (the conjugacy class of) $g$ can be written as a (cyclically) reduced loop $p(g)$ in the corresponding graph of groups $\Gamma = G \setminus X$. Note that the loop corresponding to $g$, does depend only on $\Delta$ and not on the metric of $X$.
To any $g\in\Hyp(\G)$, we  assign a \textbf{transition vector} $(a_1,a_2,\ldots,a_n)$, where $a_i$ is the number of occurrences of $e_i$'s and $E_i$'s on the loop $p(g)$.

\begin{defn}
	Let $\Delta$ be a simplex of $\O(\G)$ and $g\in\Hyp(\G)$.
	The {\em shape} of $g$ in $\Delta$ is the transition vector of $g$ with respect to $\Delta$.
\end{defn}

\begin{rem}\label{remshape}\
	\begin{enumerate}\label{Shapes}
		\item Different (conjugacy classes) of group elements may have the same
                  shape, and if so, these group elements have the same length with respect
                  to any $X \in \Delta$. So $\ell_X(\gamma)$ is defined for any shape $\gamma$.
		\item There are finitely many shapes of candidates in $\Delta$ (see Theorem \ref{Candidates}).
		\item Forall $\epsilon > 0$ and $M>0$, the set of shapes of hyperbolic elements
                  whose length is bounded by $M$ for some $X \in \mathcal{O}_1(\G,\epsilon)
                  \cap \Delta$, is finite. This follows by the fact that each of the
                  coefficients of the transition vector of such a $g$, is bounded above by
                  $M\Lambda(X,X_0)$, where $X_0$ is the centre of $\Delta\cap\O_1(\G)$ (the point where all the edges have the same length). As $X$ belongs to the $\epsilon$-thick part and has co-volume one, the distance $\Lambda(X,X_0)$ is uniformly bounded above (for instance, from $\frac{2}{\epsilon}$) and the remark follows.
	\end{enumerate}
\end{rem}

\begin{lem}[{Compare with~\cite[Theorem 7.2]{FM20}}]\label{discrete_spectrum}
	The simplex-displacement spectrum of any $\G$-irreducible $[\phi]\in\Out(\G)$, is
        discrete. That is to say\footnote{The displacement of simplices is defined in Section~\ref{sec210disp}}
	$$\operatorname{spec}(\phi)=\{\lambda_\phi(\Delta): \Delta \text{ a simplex of } \mathcal
	O(\G)\}$$ is a closed discrete subset of $\mathbb R$.
\end{lem}	

\begin{proof}
	Let $\lambda=\lambda(\phi)$.  
	We will prove the claim by showing that for any $C > \lambda$,
        $\operatorname{spec}(\phi) \cap [\lambda, C]$ is finite  (note that
        $\lambda=\inf(\operatorname{spec}(\phi))$ just by definition). 
	
	Let $\Delta$ be a simplex of $\mathcal O(\G)$. For any pair of shapes $(\gamma,\eta)$ we consider
	$\ell_X(\eta)/\ell_X(\gamma)$,  which is a function on $\Delta$ not depending on the marking, and
	for any family of pairs of shapes $B$ we define
	$$F_B(X)=\sup_{(\gamma,\eta)\in  B}\frac{\ell_X(\eta)}{\ell_X(\gamma)}.$$
	Again, $F_B(X)$ is a function on $\Delta$ that does not depend on marking, just on $B$.
	
	Since $[\phi]$ is irreducible, for any $C>\lambda$ there is $\epsilon=\epsilon(C) >0$ such
	that, for any $X\in\mathcal O(\G)$, if
	$\lambda_\phi(X) \leq C$ then $X$ is $\epsilon$-thick (see for instance~\cite[Proposition
	5.5]{FM20}).
	
	Let $S_2(C)$ be the set of shapes having length bounded by $2C\operatorname{vol}(X)$ for some
	$X$ in the
	$\epsilon(C)$-thick part of $\Delta$. The set $S_2(C)$ is finite (Remark~\ref{Shapes}).

	By candidates Theorem~\ref{Candidates},	there is a finite set $S_1$ of candidate shapes
        so that for any $X\in\Delta$, 
	$\lambda_\phi(X)=\Lambda(X,X\phi)$ is realised by $\ell_X(\phi(g))/\ell_X(g)$ for some
	$g$ having shape in $S_1$; moreover all such shapes have length at most
	$2\operatorname{vol}(X)$.
	On the other hand, for any such $g$, the shape of $\phi(g)$ has length bounded by
	$\lambda_\phi(X) \ell_X(g)$, which is bounded by $2\lambda_\phi(X)
	\operatorname{vol}(X)$. That is to say, if $\lambda_\phi(X) \leq C$, then for any $g$ with
	shape in $S_1$, the shape of $\phi(g)$ is in $S_2(C)$.
	(We remark that the set $S_1$ and $S_2(C)$ do not depend on the marking. That is to say, two
	simplices with the same unmarked underlying graph exhibit the same sets $S_1$ and
	$S_2(C)$.)
	
	It follows that there exists a family of pairs $B\subseteq S_1\times S_2(C)$ such that
	$\lambda_\phi(X)=F_B(X)$ for any  $X \in \Delta$. Note that $B$ may depend on the marking of
	$\Delta$. However, since $S_1\times S_2(C)$ is finite, there are only finitely many
        choices for 
	$B$. It	follows that the possible displacement  functions on $\Delta$ run over a finite
        sets, hence so do their minima.
	
\end{proof}

\subsection{Distance between Min-sets of an automorphism and its inverse}

\begin{lem}[{\cite[Theorem 5.3, and Lemmas 8.4, 8.5, 8.6]{FM21}}]\label{lemma_peakred}
	Given $[\psi]\in\Out(\G)$ and any $X,Y\in \mathcal O(\G)$ with
        $\lambda_\psi(X)\geq\lambda_\psi(Y)$, there is a 
 	simplicial path from $X$ to $Y$ --- that is to say, a sequence of adjacent simplices
	$\Delta_0,\Delta_1,\dots,\Delta_m$ with $X\in \Delta_0$ and $Y\in \Delta_m$ --- such that
	there exists $i_0$ such that the sequence $\lambda_\psi(\Delta_i)$ is strictly monotone
	decreasing form $0$ to $i_0$, and constant from $i_0$ to $m$.
\end{lem}

\UnifDist

\begin{proof}
	Let $X\in\O_1(\G)$ so that $\Lambda(X, X\phi)\leq L$. By Theorem~\ref{quasi-symmetry},
        the right- and left- Lipschitz 
	distances are comparable on the thick part. Since $[\phi]$ is irreducible, $X$ is
	$\epsilon$-thick (with $\epsilon$ depending on $L$ but not on $X$, see for instance~\cite[Proposition~5.5]{FM20}) and hence there is a
	constant $C_1$, not depending on $X$, such that $\Lambda(X,X\phi^{-1})=\Lambda(
	X\phi,X)<C_1$. Now we apply Lemma~\ref{lemma_peakred} with $\psi=\phi^{-1}$ and any
	$Y\in\Min(\phi^{-1})$ (which, in particular, implies
	$\lambda_{\phi^{-1}}(X)\geq\lambda_{\phi^{-1}}(Y)$). Let $(\Delta_i)$ be the sequence of
	simplices provided by Lemma~\ref{lemma_peakred}. Up to replace $Y$ with an element of
	$\Min(\phi^{-1})\cap \Delta_{i_0}$, we may assume that the sequence
	$\lambda_{\phi^{-1}}(\Delta_i)$ is strictly monotone decreasing. By
	Lemma~\ref{discrete_spectrum} there are only finitely many values in
	$\operatorname{spec}(\phi^{-1})\cap[\lambda(\phi^{-1}), C_1]$, whose cardinality
        depends only 
	on $[\phi]$. This implies that there is a uniform  bound on the length of the sequence of
	$\Delta_i$'s joining $X$ to $Y$.
	And this implies that $\Lambda(X,Y)$ is uniformly bounded depending only on
	$[\phi]$. Since both $X$ and $Y$ are $\epsilon$-thick, for some $\epsilon$ depending only on
	$[\phi]$ and $L$, then (by Theorem~\ref{quasi-symmetry}) also $\Lambda(Y,X)$ is uniformly
        bounded. 
\end{proof}
	
\section{Equivalent conditions for co-compactness of
  $\Min(\phi)$}\label{Conditions}

As customary, let $\G=(\{G_1,\dots,G_k\},r)$ be a free factor system of a group $G$.
In this section, we discuss equivalent conditions of co-compactness of the Min-Set of
$\G$-irreducible automorphisms classes $[\phi]$ with exponential growth. 

There are several topologies for our deformation spaces, but our co-compactness result is in
the strongest sense; we actually prove that any irreducible $[\phi]$ acts on $\Min_1(\phi)$
(whence on $\Min (\phi)$) with finitely many orbits of simplices, and in this sense the
topology doesn't matter, since we have a fundamental domain which is compact with respect to
any of the topologies.

However, our strategy is to show that the action is co-bounded with respect to the Lipschitz
metric, and deduce co-compactness from there and the fact that
$\Min_1(\phi)$ is locally finite (Theorem~\ref{Locally finite}). In this section we show how to
make that reduction. We start with some observations about some of the topologies commonly used for our spaces.

\subsection{Topology on deformation spaces}\label{topologies}
We have (among others) the equivariant Gromov topology (see for instance~\cite{Paulin1989}); the
length space topology (defined in Section~\ref{slt}); and the Lipschitz metric defines three topologies (see
Theorem~\ref{t13}), where the basis is given by either symmetric balls, in-balls or out-balls:

	\begin{enumerate}[(i)]
		\item The symmetric or bi-Lipschitz  ball of centre $T$ and radius $R$:
		 $$B_{sym}(T,R)=\{ S \in \O(\G) \ : \  \Lambda(T,S) \Lambda(S,T) \leq R\}.$$
		\item The Lipschitz out-ball of centre $T$ and radius $R$: $$B_{out}(T,S)=\{ S \in \O(\G) \ : \  \Lambda(T,S) \leq R\}.$$
		\item The Lipschitz in-ball of centre $T$ and radius $R$: $$B_{in}(T,R)=\{ S \in \O(\G) \ : \   \Lambda(S,T) \leq R\}.$$
	\end{enumerate}

\begin{rem} By Theorem~\ref{t13}
	        all three Lipschitz metrics are actually (multiplicative, asymmetric) metrics
                only when restricted to $\O_1(\G)$. However, the three topologies are well-defined
                also in $\O(\G)$. 
\end{rem}

\begin{rem}
	Since the Lipschitz metric is multiplicative, one should really say that the radii of
        these balls is $\log R$. This doesn't cause any problems in $\O_1(\G)$, as the
        Lipschitz metric is $1$ exactly when the points are equal, and is never less than
        that. It does cause problems in the non-symmetric case in $\O(\G)$, since 
        non-symmetric Lipschitz metrics change with scale, so one can get any positive real
        number as a value for $ \Lambda(T,S)$. The symmetric Lipschitz metric is a well
        defined multiplicative \textit{pseudo metric} on the whole $\O(\G)$. 
\end{rem}

        Let us start by proving that all topologies agree in the co-volume-one slice $\O_1(\G)$. 
\begin{lem}
	\label{outtoin}
	Let $T \in \O_1(\G)$. For any $\delta > 0$ there exists an $\epsilon > 0$ such that for
        any $S\in \O_1(\G)$, if $\Lambda(T,S) \leq 1+\epsilon$, then $\Lambda(S,T) \leq 1+ \delta$.
\end{lem}
\begin{proof}
  Consider an optimal map, $f:T \to S$. Then, by Lemma~\ref{BCC1},
  $$BCC(f)  \leq \vol(T)\Lip(f) -\vol(S) = \Lip(f) -1 = \Lambda(T,S) -1 \leq \epsilon .$$
	Hence the BCC of $f$ is bounded above by $\epsilon$. Let $a$ be the length of the
        smallest edge in $T$. Now, for any edge of $T$, if $\ell$ is its length in $T$, and $\mu$
        is how it is stretched by $f$, by looking at volumes, we get 
	$$1 =\vol(S) \leq (1+\epsilon)(1-\ell) + \mu \ell,
\qquad\text{giving}\qquad
\mu \geq 1- \frac{\epsilon(1-\ell)}{\ell} \geq 1- \frac{\epsilon(1-a)}{a}.$$
Thus, $f$ stretches all edges at least by $1-\epsilon(1-a)/a$.
By Corollary~\ref{BCC}, for any $g$,
$$
\ell_S(g) \geq (1-  \frac{\epsilon(1-a)}{a} )\ell_T(g) - \frac{\ell_T(g)}{a}\epsilon = \ell_T(g) \left( \frac{a-2 \epsilon + a\epsilon }{a} \right),
$$
where $\frac{\ell_T(g)}{a}$ is an estimate of the number of edges crossed by $g$ in $T$, and the
$\epsilon$ is just the above bound on  $BCC(f)$.
Therefore, for any $g\in\Hyp(\G)$,
$$
\frac{\ell_T(g)}{\ell_S(g)} \leq \frac{a}{a-2 \epsilon + a\epsilon }.
$$

As the upper bound tends to $1$ as $\epsilon$ tends to $0$, we have proved the result.
\end{proof}
\begin{rem}\label{remouttoin}
  Lemma~\ref{outtoin} remains true if we replace $T,S\in\O_1(\G)$ with $T,S\in\O(\G)$, modified
  as follows: $\forall T\forall\delta\exists\epsilon: \Lambda(T,S)\vol(T)<\vol(S)+\epsilon\Rightarrow
  \Lambda(S,T)\vol(S)<\vol(T)+\delta$. So the Lemma is basically true for trees with almost the same co-volume.
\end{rem}
        
We also have the reverse:

\begin{lem}
	\label{intoout}
	Let $T \in \O_1(\G)$. For any $\delta > 0$ there exists an $\epsilon > 0$ such that for
        any $S\in \O_1(\G)$, if $\Lambda(S,T) \leq 1+\epsilon$, then $\Lambda(T,S) \leq 1+
        \delta$. 
\end{lem}
\begin{proof}
	Any given $T$ is tautologically in the thick part for some appropriate level of
        thickness. Next, for any $g$,
	$$
	\Lambda(S,T) \leq K \Rightarrow \ell_S(g) \geq \frac{\ell_T(g)}{K},
	$$
	 implying that if $\Lambda(S,T) \leq K$, then $S$ will also be thick (where the thickness is divided by $K$). We can then invoke quasi-symmetry Theorem~\ref{quasi-symmetry} to immediately get the result.
\end{proof}
\begin{rem}
  As in Remark~\ref{remouttoin}, also Lemma~\ref{intoout} remains true if we replace $T,S\in\O_1(\G)$ with
  $T,S\in\O(\G)$, by suitably modifying constants.
\end{rem}

\begin{lem}
	\label{toponvolume1}
	The following topologies on $\O_1(\G)$ are the same:
	\begin{enumerate}[(i)]
		\item The equivariant Gromov topology,
		\item The length function topology,
		\item The symmetric Lipschitz topology,
		\item The out-ball Lipschitz topology
		\item The in-ball Lipschitz topology.
	\end{enumerate}
\end{lem}
\begin{proof}
	By \cite{Paulin1989}, the Equivariant Gromov topology and the length function topology
        are the same. (Paulin has a standing assumption that the group is finitely generated,
        but this is not used for this result.) Lemmas~\ref{outtoin} and \ref{intoout} imply
        that all three Lipschitz topologies are the same. 
	
	Next, if we take a sub-basic open set in the length function topology, this involves
        picking a hyperbolic group element $g$, and taking all $T \in \O_1(\G)$ such that
        $\ell_T(g)$ belongs to some open interval. Since for any $g$ the function $\ell_T(g)$
        is continuous with respect to Lipschitz metrics on $\O_1(\G)$,  such a set is open in the Lipschitz
        topology.  
	
	Conversely, by Corollary~\ref{GreenLemma}, Lipschitz out-balls are open with respect to
        the length function topology, and so Lipschitz-open sets are open in that topology.  
       \end{proof}

\begin{rem}
	One can also consider other topologies. One obvious one is the path metric obtained after giving each (open) simplex in $\O_1(\G)$ the Euclidean metric. This also turns out to be the same as the previous ones.

One also has the coherent topology, which is the finest topology (on $\O_1(\G)$ and also $\O(\G)$) which makes the inclusion maps of the simplices continuous. Care needs to be taken here, since our spaces are only a union of \textit{open} simplices, but we can take any open simplex and add all the faces that we are allowed, then insist that these inclusions are all continuous (topologising each simplex in the standard way). This is a very different topology to the one above, and we mention it only for interest.
\end{rem}

\begin{rem}
  We will always endow $\O_1(\G)$ with the topology given by Lemma~\ref{toponvolume1}.
\end{rem}
Now, we move to $\O(\G)$ and $\p\O(\G)$.

\begin{lem}[$\O_1(\G)\simeq\p\O(\G)$]
  Let $\O(\G)$ be endowed with the bi-Lipschitz topology, and consider on $\p\O(G)$ the quotient
  topology. Then $\p\O(G)$ is homeomorphic to $\O_1(G)$.
\end{lem}
\begin{proof}
Let $\pi:\O(G)\to \p\O(\G)$ be the natural projection, which is continuous by definition of
quotient topology. Since $\O_1(\G)$ is a sub-space of
$\O(\G)$, then restriction $\pi:\O_1(\G)\to \p\O(\G)$ is continuous. Also, it is clearly
bijective. It remains to prove that it is open. This is equivalent to say that for any open set
$U\subseteq \O_1(\G)$, the cone $\R^+U$ is open in $\O(\G)$ for the Lipschitz topology. Clearly
if suffices to prove it when $U$ is an open symmetric ball, say centered at $Y$ and radius
$\e$. But from
$\Lambda(Z,Y)\Lambda(Y,Z)=\Lambda(\frac{Z}{\vol(Z)},Y)\Lambda(Y,\frac{Z}{\vol(Z)})$ we get that
the symmetric ball of $\O(\G)$, centered at $Y$ and of radius $\e$, is just $\R^+U$.
\end{proof}

\begin{rem}
The bi-Lipschitz topology on $\O(\G)$ is not Hausdorff, because the symmetric metric is only a
pseudo-metric. One can naturally use on $\O(\G)=\O_1(\G)\times \R^+$ the product topology,
which is Hausdorff, agree with the Euclidean one on simplices, and for which $\p\O(\G)$ is
tautologically homeomorphic to $\O_1(\G)$. The following lemma shows in particular that both
the bi-Lipschitz and the product one are different from the length function topology.
\end{rem}
\begin{lem}[$\O_1(\G)\not\simeq\p\O(\G)$]\label{lnothomeo}
  Let $\O(\G)$ be endowed with the length function topology, and consider on $\p\O(G)$ the quotient
  topology. Then $\p\O(G)$ is {\bf not} homeomorphic to $\O_1(G)$ in general.
  In other words, the  restriction to $\O_1(\G)$ of the natural projection $\pi:\O(\G)\to
  \p\O(\G)$ is continuous, bijective, but in general is not open (for the projective length function topology).
\end{lem}
\begin{proof}
Let $X,X_n$ be as in Example~\ref{remnotcont}. The points $Y_n=\frac{3}{2}X_n$
belong to $\O_1(\G)$ and projectively converge, with respect to the length function topology,
to $X$. However $\Lambda(Y_n,X)=2$ for all $n$, in particular $\Lambda(Y_n,X)$ does not converge to $1$. In
other words, there are sets $U$ in $\O_1(\G)$ that are open for the length function topology,
such that $\R^+U$ is not open in $\O(\G)$ (again for the length function topology). An example of such set is when $U$ is the in-ball centered
at $X$ and of radius $11/10$. $\R^+U$ does not contain any of the $X_n$, while any open
neighborhood of $X$ in the length function topology, contains infinitely many of them.
\end{proof}

\begin{rem}
  The words ``in general'' in Lemma~\ref{lnothomeo} really means ``if at least one of the free
  factor groups is infinite'', as the used example is based only on this fact.
\end{rem}
\begin{rem}
Lemma~\ref{lnothomeo}  can be rephrased by saying that the
quotient length function topology on $\p\O(\G)$ is coarser than the subspace length function
topology on $O_1(\G)$. 
\end{rem}

\begin{rem}
  Example~\ref{remnotcont} shows that $\O_1(\G)$ is not closed in $\O(\G)$ with
  respect to the length function topology, as the co-volume of the limit of points in $\O_1(\G)$
  can be different from $1$. 
  
  More precisely, example~\ref{remnotcont} gives a sequence of points, $X_n$, such that $X_n \to X$ in $\O(\G)$ and $\vol(X_n) = 2/3$, whereas $\vol(X) =1$. By simply rescaling, we can set $Y_n = 3/2 X_n$ and $Y=3/2X$ so that $\vol(Y_n)=1$ and $\vol(Y) = 3/2$. We again get that $Y_n \to Y$ in $\O(\G)$ and now $Y_n$ is a sequence in $\O_1(\G)$ which converges in $\O(\G)$  but whose limit does not belong to $\O_1(\G)$. Hence $\O_1(\G)$ is not a closed subspace of $\O(\G)$. 
   
  We comment that $\O_1(\G)$ will fail to be closed in $\O(\G)$ if and only if some vertex group is infinite, as can be seen by appropriately tweaking the example above. 
\end{rem}

\begin{rem}
  As a consequence of Lemma~\ref{lnothomeo}, we get that in general  with respect to the length
  function topology,
  the closure $\overline{ \O_1(\G)}$ of $\O_1(\G)$, is not the same as the closure
  $\overline{\p\O(\G)}$ of $\p\O(\G)$. More explicitly, $\overline{\O_1(\G)}$  is exactly the
  simplicial closure of $\O_1(\G)$, which can be identified with the free splitting simplex
  (relative to the fixed free factor system) and it is exactly the space of edge-free actions
  on simplicial trees with elliptic subgroups, containing $\G$. 
On the other hand, in~\cite{CM,Horbez} it is proven that $\overline{\p\O(\G)}$ is a compact space
which contains non-simplicial trees and trees with non-trivial edge stabiliser. 
\end{rem}

Another caveat is about local compactness, as explicited  by following facts.
\begin{prop}
  In general $\O_1(\G)$ and $\O(\G)$ are not locally compact. 
\end{prop}
\begin{proof} We will present slightly different arguments in each case but with the same underlying goal; to produce inside of any neighborhood of a specially chosen point, a sequence which converges to a point \textit{outside} the space. This will show that the space is not locally compact, since our spaces are Hausdorff with respect to the length function topology.

We base our arguments on a modification of Example~\ref{remnotcont}. Referring to Figure~\ref{Fig1}, we let $X$ denote the graph in Example~\ref{remnotcont} which is topologically a circle of length 1. This is going to be the point which does not admit a compact neighborhood for both $\O_1(\G)$ and $\O(\G)$.

Let us deal with $\O_1(\G)$ first. Build points $X_{n, \epsilon}$, as in Example~\ref{remnotcont}, by assigning length $\epsilon$ to the horizontal edge ending at the non-free vertex, and  $1-\epsilon$ to the loop.

  Then, for each $n$, all the $X_{n, \epsilon}$ belong to the same simplex in $\O_1(\G)$. Hence, for each $n$, $\lim_{\epsilon \to 0} X_{n, \epsilon} =  X$. However, just as in Example~\ref{remnotcont}, $\lim_{n \to \infty} X_{n, \epsilon}= (1+\epsilon) X$. (In fact, this convergence is quite strong; for any group element the lengths with respect to the $X_{n, \epsilon}$ are eventually constant).

  Let $V$ be any neighborhood of $X$ in $\O_1(\G)$. Since $\lim_{\epsilon \to 0} X_{n, \epsilon} =  X$ there is an $\epsilon >0$ such that, for all but finitely many $n$, $V$ contains the sequence $X_{n, \epsilon}$. But now $\lim_{n \to \infty} X_{n, \epsilon}= (1+\epsilon) X \not\in \O_1(\G)$, hence $V$ cannot be compact as it admits a sequence with a limit point outside of $\O_1(\G)$ (and hence no limit point in $\O_1(\G)$ as our spaces are Hausdorff). Hence $\O_1(\G)$ is not locally compact.

  \medspace 
  
  Now we deal with $\O(\G)$, using the same point $X$. This time, define points $Y_{n, \epsilon}$ as in Example~\ref{remnotcont} by assigning the length $\epsilon$ to the horizontal edge and $1-2\epsilon$ to the loop - here $0<\epsilon<1/2$. (These are similar to the $X_{n, \epsilon}$ above, but the volume is no longer 1 so we are not looking a sequences in $\O_1(\G)$). Each $Y_{n, \epsilon}$ is obtained from $X$ by an isometric fold of length $\epsilon$. This means that, for any group element $g$, the length of $g$ in $X$ will be equal to the length of $g$ in $Y_{n, \epsilon}$ for all but finitely many $n$ and for all $\epsilon$. 
  
  Hence, since a basic open set in the length function topology only imposes restrictions on finitely many group elements, we have that if $V$ is a neighborhood of $X$ in $\O(\G)$, then $V$ contains $Y_{N, \epsilon}$ for some fixed value of $N$ and for all $0<\epsilon<1$. However, if we let $\epsilon \to 1/2$ we get an element (the loop of X) which is hyperbolic in $\O(\G)$ whose length tends to 0. Thus we have obtained a sequence in $V$ which converges to a point in $\overline{\O(\G)}$ which is not in $\O(\G)$. Therefore, as above, $\O(\G)$ is not locally compact. 
  
%
\end{proof}

Finally, note that we have no hypothesis on our
vertex groups, in particular $G$ may be not countable. However, when $G$ is countable we get:

\begin{rem}
	\label{seqcpt}
	When $G$ is countable, bounded sets in $\overline{\O(\G)}$ are sequentially
        pre-compact. Namely, if $(X_n)\subset \overline{\O(\G)}$ is a sequence such that for
        any $g\in G$ there is $K(g)$ for which $\ell_{X_n}(g)\leq K(g)$, then $X_n$ has a
        converging subsequence. 
         This follows from the fact that the product of countably many closed, bounded real
         intervals, is sequentially compact.        
\end{rem}

\subsection{Equivalent formulations of co-compactness for $\Min_1(\phi)$}
Let $[\phi]\in\Out(\G)$ be irreducible and with exponential growth, that is, $\lambda(\phi)>1$.
We know that in this case the simplicial structure of $\Min_1(\phi)$ is locally finite
(Theorem~\ref{Locally finite}).  Our aim here is to show that co-compactness of this space is
actually equivalent to co-boundedness. We first recall some terminology.

\begin{defn} A \textit{simplicial path} between $T,S \in \O(\G)$ (or $\O_1(\G)$)  is given by:
	\begin{enumerate}
		\item A finite sequence of points $T = X_0, X_1,\ldots, X_n = S$, called
                  vertices, such that for every $i = 1,\ldots, n$ there is a simplex
                  $\Delta_i$ such that the simplices $\Delta_{X_{i-1}}$ and $\Delta_{X_i}$ are
                  both faces (not necessarily proper) of $\Delta_i$. 
		\item Euclidean segments $\overline{X_{i-1} X_i} \subset \Delta_i$ called
                  edges. (The simplicial path is then the concatenation of these edges.)
                \item The simplicial length of such a path is just the number $n$.
	\end{enumerate}
\end{defn}
\begin{rem}\label{rem422cal}
The notion of {\em calibrated} path (with respect to $\phi$) is introduced in~\cite{FM21}. In the case of irreducible  automorphisms, this simply amounts to asking that any non-extremal
vertex $X_i$ realises $\lambda_\phi(\Delta)$ on its simplex, that is to say,
$\lambda_\phi(X_i)=\lambda_\phi(\Delta(X_i))$. In particular, if $[\phi]$ is $\G$-irreducible,
then any
simplicial path may be assumed to  be calibrated by just replacing any $X_i$ with a suitable
point in the closure of  $\Delta(X_i)$. (See~\cite{FM21} for more details). 
\end{rem}

It is proved in~\cite{FM21} that any two points in $\Min(\phi)$ (or $\Min_1(\phi)$) can be
joined by a simplicial path lying entirely within $\Min(\phi)$ (or $\Min_1(\phi)$,
respectively). This is done by a peak-reduction argument. We use here the following
quantitative version of that:

\begin{prop}[{\cite[Remark 8.7]{FM21}}]
	\label{singlepeakred}
	Any calibrated simplicial path $\Sigma$  connecting two points in $\O(\G)$ can be
        peak-reduced to a new calibrated simplicial path, by removing a local maximum (peak) of the function $\lambda_\phi$, via a
        peak-reduction surgery that increases the simplicial length of $\Sigma$ by at most a
        uniform amount $K$   
        depending only on $\rank(\G)$.
\end{prop}
\begin{rem}
  In~\cite{FM21}, the authors are interested in the function $\lambda_\phi$, which is scale
  invariant, and results are stated and proved for $\O(\G)$, but it is readily checked that
  the whole peak-reduction can be carried on $\O_1(\G)$.
\end{rem}

Recall that if $[\phi]$ is $\G$-irreducible,  then its simplex-displacement spectrum  $\operatorname{spec}(\phi)$ is discrete
(Theorem~\ref{discrete_spectrum}). Thus for any $x>\lambda(\phi)$ the set
$\operatorname{spec}(\phi)\cap[\lambda(\phi),x]$ is finite.
\begin{cor}\label{corpr}
  Let $[\phi]\in\Out(\G)$ be $\G$-irreducible and with $\lambda(\phi)>1$. Let $X,Y\in\O(\G)$ and let
  $X=X_0,\dots,X_L=Y$ be a simplicial path. Let $D=\max_i\lambda_\phi(\Delta(X_i))$ and
  $D_0=\max(\lambda_\phi(X),\lambda_\phi(Y))$. Let $N$ be the cardinality of
 $\operatorname{spec}(\phi)\cap[\lambda(\phi),D]$. 

  Then there exist a simplicial path
  $X=X_0',\dots,X_{L'}'=Y$ such that $\lambda_\phi(X_i')\leq D_0$ for any $i$, and
 such that  $L'\leq L(K+1)^N$, where $K$ is the constant of Proposition~\ref{singlepeakred}. 
\end{cor}
\begin{proof}
  By Remark~\ref{rem422cal} we may assume that our starting simplicial path is calibrated.
List possible simplex dispacements less than $D$,  $\lambda(\phi)=\lambda_1 < \lambda_2
<\ldots < \lambda_N\leq D$.
To any calibrated simplicial path we assign a triple $(\lambda_i, m, L)$, where
$\lambda_i$ is the maximum displacement of vertices along the path, $m$ is the number of vertices in the simplicial path which realise the maximum displacement, and $L$ is the simplicial
length. Note  that $m \leq L$.

The peak reduction process (Proposition~\ref{singlepeakred}) allows us to reduce the value of
$m$ by $1$, at a cost of increasing the value of $L$ by $K$. Hence, after at most $L$ peak
reductions, we have transformed our simplicial path to one where the maximum peak has
displacement at most $\lambda_{i-1}$. The effect on the triple is to replace it with
$(\lambda_j, m', L')$, where $j < i$ and $L' \leq L + KL = L(K+1)$. Inductively, we see that we
eventually arrive at a path with the requested bound on displacement, and of simplicial length
at most $L(K+1)^N$. 
\end{proof}

Given $X\in\O(\G)$, let $X^{c}$ denote the ``centre'' of the open simplex containing $X$ in
$\O(\G)$. That is, $X^c$ has the same action as $X$, but the edges are rescaled to all have
length $1$. (Hence $X^c$ does not have co-volume $1$, but its co-volume is uniformly bounded, since
there is an upper bound on the number of orbits of edges.)

The following proposition shows that symmetric Lipschitz balls can be connected via simplicial
paths of uniform length, if we allow ourselves to enlarge the ball slightly.

\begin{prop}
  \label{pathsinballs}
 For all $\epsilon >0$ there exist constants $M,\kappa,\alpha$ such that for any  
 $T\in\O_1(\G,\epsilon)$ and all $R>0$, any two points $S_1,S_2\in B_{sym}(T,R)\cap\O_1(\G)$ are 
  connected by a   simplicial path entirely contained in $B_{sym}(T,\kappa R^\alpha)\cap
  \O_1(\G)$, and crossing at most $MR$ simplices. Moreover, all points of such path are
  $\frac{\epsilon}{\kappa R^\alpha}$-thick. 
\end{prop}	
\begin{proof} It is sufficient to prove the claim when $S_2=T$ and $S_1=S$ is any other
  co-volume-one point in $B_{sym}(T,R)$. Also, we observe that since  $T$ is $\epsilon$-thick,
  any point in  any $B_{sym}(T,\rho)$ is 
  $\epsilon/\rho$-thick. In particular last claim follows from first one.  

  By triangular inequality, $\Lambda(S,T^c)\leq \Lambda(S,T)\Lambda(T,T^c)$ and since $\vol(S)=\vol(T)=1$,
  we have $\Lambda(S,T)\leq\Lambda(S,T)\Lambda(T,S)<R$. Moreover $\Lambda(T,T^c)$ is uniformly
  bounded by a constant depending only on $\epsilon$ (and the maximal number of edge-orbits of
  trees in $\O(\G)$).
  
So $\Lambda(S,T^c)$  is bounded by a  uniform multiple of $R$.  Let $f:S \to T^c$ be an optimal map. Subdivide the edges
of $S$ by ``pulling back'' the edge structure on $T^c$; that is, subdivide $S$  so that every
new edge maps to an edge of $T^c$ under $f$. The number of new edges 
created by this subdivision is at most $\Lambda(S, T^c)$ times the number of edges in
$S$. More concretely, for each edge $e$ in $S$, $\Lambda(S,T^c)$ is an upper bound for
the number of edges crossed by $f(e)$ because each edge in $T^c$ has length $1$. 

Now  - as in~\cite[Definition~7.6]{FM15} (see also~\cite[Theorem 5.6]{FM11}) - construct a
folding path directed by 
$f$, from $S$ to $T^c$; it is a simplicial path in $\O(\G)$. The  simplicial length of this path
is bounded  above by the number of subdivided edges of $S$. That is, it will be bounded above
by a uniform multiple of $R$.

Moreover such path, say $S=X_0,X_1,\dots ,X_n=T^c$, has the following properties
(see~\cite{FM15,FM11}\footnote{We remind
that the first step of such construction is to build $X_1$ by changing lengths to edges of $S$
so that $\Lambda(X_1,T^c)=1$, and all edges are maximally stretched. Then we proceed, as
the name suggests, by isometrically folding edges identified 
by $f$.}): $$\Lambda(X_i,X_j)=1\ \text{for any }0<i<j\leq n,
\qquad \Lambda(S,X_1)=\Lambda(S,T^c),$$
$$\vol(X_i)\leq \vol(X_1)\ \text{for any }0<i\leq n, \qquad \vol(X_1)\leq\Lambda(S,T^c).$$

We now ``correct'' this path by tracing a simplicial path which goes
through the same simplices, but whose vertices are uniformly thick and with co-volume $1$. More
precisely for 
any $i\geq 1$, we replace  each $X_i$ with $X_i^c/\vol(X_i^c)$. Now, since $T^c=X_n$, we have:
$$\Lambda(X_i^c,T)\leq\Lambda(X_i^c,X_i)\Lambda(X_i,T^c)\Lambda(T^c,T)=\Lambda(X_i^c,X_i)\Lambda(T^c,T);$$
$\Lambda(X_i^c,X_i)$ is bounded by
$\vol(X_i)\leq\vol(X_1)\leq\Lambda(S,T^c)\leq\Lambda(S,T)\Lambda(T,T^c)$. It follows
\begin{eqnarray*}
 \Lambda(\frac{X_i^c}{\vol(X_i^c)},T)
  &=&\vol(X_i^c)\Lambda(X_i^c,T)\leq \vol(X_i^c)
\Lambda(S,T)\Lambda(T,T^c)\Lambda(T^c,T)\\ &\leq& R\vol(X_i^c)\Lambda(T,T^c)\Lambda(T^c,T).  
\end{eqnarray*}
The factor $\Lambda(T,T^c)\Lambda(T^c,T)$ is a priori bounded by a constant depending on
$\epsilon$ (see for example~\cite[Lemma 6.7]{FMS}), and the claim follows from quasi-symmetry  (Theorem~\ref{quasi-symmetry}) --- because both $T$ and
all $X_i^c$ are thick --- and from the fact that centres $X_i^c$ have uniformly bounded
co-volume. Note that we can take $\alpha=1+C$ where $C$ is the constant of Theorem~\ref{quasi-symmetry}.
\end{proof}

We next show that $\Min_1(\phi)$ is not overly distorted, in the following sense.

\begin{thm}\label{undistorted}
  Let $[\phi]\in\Out(\G)$ be $\G$-irreducible with $\lambda(\phi) >1$. For any $T\in\O(\G)$ there
  are constants $C,C'$, depending only $\lambda_\phi(T)$ (and on $[\phi]$ and $\rank(\G)$), such
  that for all $R>0$:
  \begin{enumerate}
  \item Any  two points in  $\Min_1(\phi)\cap B_{sym}(T,R)$ are connected 
by a simplicial path entirely contained in $\Min_1(\phi)$ (but not necessarily in
$B_{sym}(T,R)$), and whose simplicial  length is bounded above by $CR$.
\item The ball $B_{sym}(T,R)$ intersects at most $C'R$ simplices of $\Min_1(\phi)$. 
  \end{enumerate}
\end{thm}
\begin{proof}
First claim implies in particular that if $\Delta_0$ is a simplex intersecting
$\Min_1(\phi)\cap B_{sym}(T,R)$, then any other simplex with the same property stays at bounded
simplicial distance from $\Delta_0$. Since the simplicial structure of $\Min_1(\phi)$ is
locally finite by Theorem~\ref{Locally finite}, the second claim follows. 

Let us now prove the first claim. Since the symmetric Lipschitz pseudo-metric is scale
invariant, we may assume $\vol(T)=1$. Moreover, since $[\phi]$ is irreducible, then $T$ is
$\epsilon$-thick for some $\epsilon>0$ depending on $\lambda_\phi(T)$ (but not on $T$, see for
instance~\cite[Proposition 5.5]{FM20}).

By Proposition~\ref{pathsinballs}, any two points $S_1,S_2\in B_{sym}(T,R)\cap\O_1(\G)$ can be
joined by a 
simplicial path of simplicial length at most $MR$, and lying inside $B_{sym}(T, R')$
(with $R'=\kappa R^\alpha$ and constants $M, \kappa,\alpha$ as in Proposition~\ref{pathsinballs}).

For any $S \in B_{sym}(T, R')$ and any hyperbolic group element $g$ we have
	$$
	\frac{\ell_S(\phi(g))}{\ell_T(\phi(g))} \leq R' \qquad \text{and } \qquad
        \frac{\ell_T(g)}{\ell_S(g)} \leq R',
	$$
which implies
	$$
	\frac{\ell_S(\phi(g))}{\ell_S(g)} \leq (R')^2 \frac{\ell_T(\phi(g))}{\ell_T(g)} \Rightarrow \lambda_{\phi}(S) \leq (R')^2 \lambda_{\phi}(T).
	$$
        Hence the displacements of points in $B_{sym}(T, R')$ are uniformly bounded
        by $(R')^2 \lambda_{\phi}(T)$.

        Let $N$ be the cardinality of
        $\operatorname{spec}(\phi)\cap[\lambda(\phi),(R')^2 \lambda_{\phi}(T)]$.
        By Corollary~\ref{corpr}, if $S_1,S_2\in\Min(\phi)$, they can therefore be
  connected by a simplicial path in $\Min(\phi)$ whose length is bounded by $MR(K+1)^N$ (where
  $K$ is the constant of Corollary~\ref{corpr}) and
  by scaling volumes we can get such path in $\Min_1(\phi)$.
\end{proof}

Recall (see Section~\ref{LimitTreeSection}) that the stable or attracting tree of an $X \in \Min(\phi)$ exists and is given by $X_{+\infty} = \lim_n \frac{X \phi^n}{\lambda(\phi)^n}$.

\begin{thm}\label{cocompactnessdefs}
  Consider $\O_1(\G)$ endowed with one of the equivalent topologies given by
  Lemma~\ref{toponvolume1}. Let $[\phi]\in\Out(\G)$ be $\G$-irreducible and with $\lambda(\phi) >
  1$, let $X \in \Min_1(\phi)$ and $X_{+\infty}$ be the stable tree. Then the following are
  equivalent: 
	\begin{enumerate}[(i)]
		\item\label{i} $\langle \phi \rangle$ acts on $\Min_1(\phi)$ with finitely many orbits of simplices.
		\item\label{ii} $\Min_1(\phi)/\langle \phi \rangle$ is compact.
		\item\label{iii} There exists a compact set in $\O_1(\G)$ whose $\langle \phi \rangle$-orbit covers $\Min_1(\phi)$.
			\item\label{iv} There exists a closed symmetric Lipschitz ball $B$ whose $\langle \phi \rangle$-orbit covers $\Min_1(\phi)$.
		\item\label{v} There exists a closed Lipschitz out-ball $B$ whose  $\langle \phi \rangle$-orbit covers $\Min_1(\phi)$.
                \item\label{vi} $\exists C$ such that $\forall Y \in \Min_1(\phi)$,  $1 \leq
                  \Lambda(Y, X_{+\infty}) \leq \lambda(\phi)\Rightarrow \Lambda(X,Y) \leq C$.
                \item\label{vii} $\forall D_0>1\  \exists C_0$ such that $\forall Y \in
                  \Min_1(\phi)$,  $\frac{1}{D_0} \leq \Lambda(Y, X_{+\infty}) \leq
                  D_0\Rightarrow \Lambda(X,Y) \leq C_0$.
		
		\item\label{viii} $\forall V_1 < V_2\ \exists C'$ such that $\forall Y \in
                  \Min(\phi)$, if $V_1 \leq \vol(Y) \leq V_2$ and $\Lambda(Y, X_{+\infty}) =1$,
                  then $\Lambda(X,Y) \leq C'$. 
	\end{enumerate}
\end{thm}

\begin{proof}
  Since $\lambda_\phi$ is continuous, then for every simplex $\Delta$ of $\O(\G)$, the set
  $\Min_1(\phi)\cap \Delta$ is compact. Moreover, as a consequence of
  Theorem~\ref{undistorted}, we see that (even if $\O_1(\G)$ is not locally compact)
  $\Min_1(\phi)$ is a locally compact  
space, whose compact subsets meet finitely many of its simplices and are contained in a closed
symmetric Lipschitz ball. This gives immediately the equivalence between \eqref{i},~\eqref{iii} and~\eqref{iv}.

Moreover, since $\phi$ acts by homeomorphisms on $\Min_1(\phi)$, from local
compactness we get also that \eqref{ii} is equivalent to \eqref{iii}.

Uniform thickness of $\Min_1(\phi)$ (Theorem~\ref{PropertiesOfIrreducibles}) and quasi-symmetry
(Theorem~\ref{quasi-symmetry}) gives the equivalence of \eqref{iv} and \eqref{v}.
	
	It is clear that \eqref{vii} and \eqref{viii} are equivalent. It is also easy to see that \eqref{vii} implies \eqref{vi}, by taking $D_0 = \lambda(\phi)$.
	
	We see now that \eqref{vi} implies \eqref{v}. Notice that $X_{+\infty} \phi = \lambda(\phi) X_{+\infty}$. Hence for any integer $n$ (positive or negative),
	$$\Lambda(Y \phi^n, X_{+\infty}) = \frac{\Lambda(Y, X_{+\infty})}{(\lambda(\phi))^n}.$$
	
	In particular, for every $Y \in \Min_1(\phi)$, there exists a $n$ such that $1 \leq
        \Lambda(Y \phi^n, X_{+\infty}) \leq \lambda(\phi) $, and \eqref{v} follows from \eqref{vi}.

	To summarise, we have that \eqref{i}, \eqref{ii}, \eqref{iii}, \eqref{iv} and \eqref{v}
        are equivalent, that \eqref{vii} and \eqref{viii} are equivalent,  that \eqref{vii} implies \eqref{vi} and \eqref{vi} implies \eqref{iv}. 
	
	Thus, in order to complete the proof, it suffices to show that \eqref{iv} implies
        \eqref{vii}.
        Let $B$ be a closed symmetric Lipschitz ball of radius $R$ whose translates cover
        $\Min_1(\phi)$. Without loss of generality, we can assume that its centre is $X$. Since $\Lambda(X, X_{+\infty})=1$, from multiplicative triangle
        inequality, for all $Y$ we get the following inequalities:
	$$
	\frac{1}{\Lambda(X,Y)} \leq \Lambda(Y, X_{+\infty}) \leq \Lambda(Y,X).
	$$	
	If also $Y$ has co-volume $1$, then we have both $\Lambda(X,Y), \Lambda(Y,X) \geq 1$
        (Theorem~\ref{t13}). So for any $Y \in B \cap O_1(\G)$ we get $1\leq \Lambda(X,Y),
        \Lambda(Y,X) \leq R$, whence 
		$$
	\frac{1}{R} \leq \Lambda(Y, X_{+\infty}) \leq R.
	$$
	Now suppose that we are given $D_0$ and  $Y \in \Min_1(\phi)$ such that
	$$
	\frac{1}{D_0} \leq \Lambda(Y, X_{+\infty}) \leq D_0.
	$$
	Since we know that the translates of $B$ cover $\Min_1(\phi)$, we get that $Y \phi^n \in B$ for some integer $n$. Hence, for this $n$	
	$$
	\frac{1}{R} \leq \Lambda(Y \phi^n , X_{+\infty}) = \frac{\Lambda(Y, X_{+\infty})}{(\lambda(\phi))^n} \leq R.
	$$
	Therefore,	
	$$
	\frac{1}{RD_0} \leq (\lambda(\phi))^n \leq RD_0,
	$$
	and we get a bound on $|n|$ depending only on $D_0$. But now
	$$
	\Lambda(X,Y) \leq \Lambda(X, Y\phi^n) \Lambda(Y \phi^n, Y) \leq R \Lambda(Y \phi^n, Y) 
	$$
	where $\Lambda(X, Y\phi^n) \leq R$ follows since $Y \phi^n \in B$.
	The following claim will conclude the proof.
\begin{claim*}
$ \Lambda(Y \phi^n, Y) \leq 	\max \{ (\lambda(\phi))^{|n|}, D(\lambda(\phi^{-1}))^{|n|} \}$, where $D$ is the constant from Theorem~\ref{UnifDistFromMinSet}.
\end{claim*}
\begin{proof}[Proof of the Claim]
If $n$ is negative, since $Y \in \Min_1(\phi)$ we have
$$
\Lambda(Y \phi^n, Y) = \Lambda(Y , Y \phi^{-n})=(\lambda(\phi))^{|n|}.
$$

Whereas, if $n$ is positive, then by Theorem~\ref{UnifDistFromMinSet}   $Y$ is uniformly close a point $Z \in \Min_{\phi^{-1}}$  and hence
$$
\begin{array}{rcl}
	\Lambda(Y \phi^n, Y) & = &  \Lambda(Y , Y \phi^{-n})  \\
	& \leq &  \Lambda(Y,Z) \Lambda(Z, Z \phi^{-n}) \Lambda (Z\phi^{-n},Y \phi^{-n}) \\
	& = &  \Lambda(Y,Z) \Lambda(Z, Z \phi^{-n}) \Lambda (Z,Y) \\
	& \leq & D (\lambda(\phi^{-1}))^n.
\end{array}
$$

\end{proof}

\end{proof}

\section{North-South dynamics for primitive irreducible automorphisms}\label{s5}
Let $\G=(\{G_1,\dots,G_k\},r)$ be a free factor system of a group $G$, where in
addition we require that $\rank(\G)\geq 3$.

\subsection{Statement of North-South Dynamics}
The so-called North-South dynamics for iwip automorphisms in the classical Culler-Vogtmann
Outer space $CV_n$ is a well known fact established in~\cite{LL}, and generalised
in~\cite{Gupta} in greater generality. (In~\cite{Gupta} the hypothesis $\rank(\G)\geq
3$ is required). Here we need a North-South dynamics result for our case, which
is slightly more general.

The proof of Theorem~\ref{NSdyn} below, is essentially exactly the same as the proof of~\cite[Theorem~C]{Gupta}, where the author assumes that $G$ is free and that the automorphism
is iwip (or fully irreducible). However, her proof applies for general groups,  where some
missing point can be filled using results of~\cite{GH}. Finally, we note that the iwip
property is not really used anywhere in the proof. More specifically, what is needed is an
irreducible automorphism which can be represented by some simplicial train track
map with primitive transition matrix (whence every simplicial train track
representative of $[\phi]$ has this property). Irreducibility is needed because it implies
connectedness of local Whitehead graphs -- which is in fact used -- but irreducibility of
powers is not used. 

For these reasons, we decided to  not include all details of the proof here, but for the rest of the
section, we just mention the main steps of the proof of~\cite[Theorem~C]{Gupta},  and we explain why the proofs of relevant statements still hold on our context, by giving appropriate references when needed.

\begin{thm}\label{NSdyn} Any $\G$-primitive irreducible $[\phi] \in \Out(\G)$ acts on
  $\O(\G)$ with projectivised 
  uniform north-south dynamics:
There are two fixed projective classes of trees $[T^+_{\phi}]$ and $[T^- _{\phi}]$, such that
for every compact set $K$ of $\overline{\p\O(\G)}$, that does not contain $[T^- _{\phi}]$ (resp.
$[T^+_{\phi}]$) and for every open neighbourhood $U$ (resp. $V$) of  $[T^+_{\phi}]$ (resp.
$[T^- _{\phi}]$), there exists an $N \geq 1$  such that for all $n \geq N$ we have $(K)\phi^n
\subseteq U$ (resp. $(K)\phi^{-n}\subseteq V$). 
\end{thm}

To begin with, we define $T_{\phi}^\pm$: they are the attracting and repelling tree
(Definition~\ref{lxinf}) of a given train
track point with primitive transition matrix, which exists by hypothesis. 
The proof of Theorem~\ref{NSdyn} starts with pointwise convergence, that we state separately in
following  Theorem~\ref{North-South1} and whose proof is contained in the proof~\cite[Theorem
C]{Gupta},  even if it is not written as a separate statement there. (Compare also with
Proposition~\ref{weakns}). 

\begin{thm}\label{North-South1}
Let $[\phi] \in \Out(\G)$ be $\G$-primitive irreducible.
For any $X\in\overline{\O(\G)}$, $X \notin [T^{-}_{\phi}]$ (resp. $X\notin[T^+_\phi])$, we have that for $n \to \infty$:
	$$ \frac{X\phi^n}{ \lambda(\phi)^n} \to  c T^+_{\phi}, \text{ for some } c>0 
	\qquad \text{(resp. }\frac{X\phi^{-n}}{ \lambda(\phi^{-1})^{n}} \to  d T^-_{\phi}, \text{ for some } d>0\text{)}.$$
\end{thm}
Before going into the proof of Theorem~\ref{North-South1} we show how it allows to extend
results of Section~\ref{LimitTreeSection} to general trees.

\begin{defn}
	Let $[\phi]\in\Out(\G)$ be $\G$-primitive irreducible.
	Let $X \in \overline{\O(\G)}$, which does not belong to the projective class
        $[T^+_{\phi}]$ (resp. $ [T^-_{\phi}]$). 
	We define the tree:
	$$
	X_{+\infty} = \lim_n \frac{X\phi^n}{ \lambda(\phi)^n}
	\qquad\text{(resp. }
	X_{-\infty} = \lim_n \frac{X\phi^{-n}}{ \lambda(\phi^{-1})^{n}}\text{).}$$
\end{defn}

It is very important to note here that these limits do exist and they are $\mathbb{R}$-trees
because of the previous Theorem~\ref{North-South1}. Note that this extends (when $[\phi]$ is
$\G$-primitive) the definition given in Section~\ref{LimitTreeSection} for points $X \in \Min(\phi)$.

The assumption that our irreducible automorphism is primitive is crucial in order to apply Theorem~\ref{North-South1}.
For a general irreducible automorphism (without the extra assumption of the primitive
property), we cannot ensure that the limits do exist for general points of $\overline{\O(\G)}$,
but only for points of $\Min(\phi)$, and we also may lose uniqueness.

\begin{rem}
  Given a point $X \in \overline{\O(\G)}$ which does
  not belong to $[T^+_{\phi}]$ (resp. to $[T^-_{\phi}]$),  by Theorem~\ref{North-South1}
  we have $X_{+\infty}= c T^+_{\phi}$ (resp. $X_{-\infty}=dT^-_\phi$) for some $c>0$
  (resp. $d>0$). If $X\in\Min(\phi)$ (resp. $X\in \Min(\phi^{-1})$), then by continuity of
  stretching factor (Corollary~\ref{GreenLemma}) we have
  $$ 1=\Lambda(X,\frac{X\phi^n}{\lambda(\phi)^n})\to
  \Lambda(X,cT^+_\phi)=c\Lambda(X,T^+_\phi)\qquad \text{(resp. } 1=d\Lambda(X,T^-_\phi)\text{)}.$$
(Compare also Proposition~\ref{weakns}, where the existence of $T^+_\phi$ is not used.) 
\end{rem}

We are now ready to analyse the main steps of the proof of~\cite[Theorem~C]{Gupta},  and adapt
them to prove Theorems~\ref{NSdyn} and~\ref{North-South1}. 

\subsection{The attracting tree does not depend on the chosen train track}
A key step for proving that attracting tree does not depend on the chosen train track is the
following proposition. 
\begin{prop}\label{Convergence1}
Let $T \in \overline{\O(\G)}$. Suppose there exists a tree $T_0 \in \O(\G)$, an equivariant map
$h : T_0 \to T$, and a bi-infinite geodesic $\gamma_0 $ of $T_0$, representing a generic leaf
$\gamma$ of $\Lambda^+_{\phi}$, such that $h(\gamma_0)$ has diameter greater $2BCC(h)$.
Then:
\begin{enumerate}
\item $h(\gamma_0)$ has infinite diameter in $T$;
\item there exists a neighbourhood $V$ of $T$ such that $(V)\phi^p$ converges to $T^+_{\phi}$,
uniformly as $p \to \infty$.
\end{enumerate}
\end{prop}
\begin{proof}
The proof goes as in the classical case (\cite[Lemma~3.4]{BFH-Laminations0} and
\cite[Proposition~6.1]{LL}) and it is the same as that of~\cite[Proposition~3.3]{Gupta}. The iwip property is
used there, just in order to ensure the existence of a train track representative of the
automorphism with primitive transition matrix. The rest of the argument uses this primitive
matrix in applying Perron-Frobenius theory. We conclude that our assumption that the
automorphism is $\G$-primitive is enough. Note that the assumption in~\cite{Gupta} that the group
$G$ is free is never used for the proof. 
\end{proof}

\subsection{Infinite-index subgroups do not carry the attracting lamination}
Another key step  in the proof~\cite[TheoremC]{Gupta} is the following.

\begin{lem}\label{Carries1}
Let $T\in\O(\G)$ and let $f:T \to T$ be a train track representative of a 
$\G$-primitive irreducible $[\phi] \in \Out(\G)$. 

Let $C$ be a subgroup of $G$ such that for every $[G_i] \in [\G]$,
either $C \cap G_i$ is trivial or equal to $G_i$ up to conjugation. Suppose moreover that $[\G]$
induces on $C$ a free decomposition of finite rank. If $C$ carries $\Lambda^+_{\phi}$, then $C$ has
finite index in $G$.  
\end{lem}
\begin{proof}
Again, no particular patch is needed, and the proof is exactly the same as that
of~\cite[Lemma~3.9, point~$(c)$]{Gupta}. It relies on the fact that there is one (and so every)
leaf of the lamination which crosses (the orbit of) every edge. Here we use the fact that the relative Whitehead graphs are connected at each vertex,
and  this can be ensured under our assumptions because 
 $[\phi]$ is $\G$-irreducible. The fact that the group $G$
is free is not used for this proof at all. 
\end{proof}

\subsection{$Q$-map and dual laminations of trees}
In this section we give the definition and some results about the statements on dual
laminations of trees, which are well known for free groups and they have been recently generalised for the context of free products in~\cite{GH}.

\begin{prop}[{\cite[Lemma~4.18]{GH}}]
	\label{SmallBCC1}
	Let $T \in \overline{ \O(\G)}$ be a minimal $\G$-tree with dense orbits and
	trivial arc stabilisers. Given $\epsilon > 0$, there exists a tree $T_0 \in \O(\G)$
        with co-volume $\vol(T_0) < \epsilon$, and an equivariant map $h : T_0 \to T$ whose
        restriction to each edge is isometric, and with  $BCC(h) < \epsilon$. 	
\end{prop}

The so-called $\mathcal Q$-map, which was defined in~\cite{LL} for free groups, can also be generalised
for general free products. Any $X\in\partial_\infty(G,\G)$  can be represented as the
``point at inifinity'' of a half-line in a $\G$-tree $T\in\O(\G)$. Almost the same happens for trees
$T\in\overline{\O(\G)}$, the difference is that in this case, the path representing $X$ in $T$
could have finite length. If this happens, $X$ is called $T$-bounded.

\begin{prop}[{\cite[$\mathcal Q$-map, Proposition 6.2]{GH} and \cite[Proposition 3.1]{LL}}]\label{Q-map}
	Let $T \in \overline{\O(\G)}$ be a minimal $\G$-tree with dense orbits and
	trivial arc stabilisers. Suppose $X \in \partial_{\infty} (G,\G)$ is $T$-bounded. Then
        there is a unique point   $Q_T(X) \in \overline{T}$ 
	such that for any $T_0\in\O(\G)$, any half-line $\rho$ representing $X$ in $T_0$, and
        equivariant map $h : T_0 \to T$, the
	point $Q_T(X)$ belongs to the closure of $h(\rho)$ in $T$. Also, every $h(\rho)$ is
        contained in a $2BCC(h)$-ball centered at $Q_T(X)$, except for an initial part. 
\end{prop}
      
\begin{defn}
 Let $T \in \overline{\O(G)}$. We define:
\begin{enumerate}
\item The algebraic lamination dual to the tree $T$, is defined as $L(T)=  \bigcap_{\epsilon>0}
  L_{\epsilon} (T)$ where $L_{\epsilon} (T)$ is the closure of set of pairs $(g^{-\infty},
  g^{\infty})$ where $\ell_T(g) < \epsilon$ and $g$ does not belong to some free factor of
  $[\G]$.
\item Let's further assume that $T$ has dense orbits. We define $L_Q(T) = \{(X,X')  : Q_T(X) =
  Q_T(X')\} \subset \partial^2 (G, \G)$. 
\end{enumerate}
\end{defn}

These definition are equivalent, in the case of trees with dense orbits, by~\cite[Proposition~6.10]{GH}.  Moreover, in~\cite[Remark~3.1]{Gupta} is shown that the leaves
of $L_Q(T^-_\phi)$ are either leaves of $\Lambda^+_\phi$ or concatenation of two rays, based at
a non-free vertex, obtained as iterated images of an edge via a train track map.  The latter
are called {\em diagonal leaves} (and do not arise in the classical case).

\begin{prop}[{\cite[Proposition~3.22]{Gupta}}]\label{Q-map1}
If $Y \in \overline{\O(\G)} $ is a minimal $\G$-tree with dense orbits and trivial arc
stabilisers, then at least one of the following is true: 
\begin{enumerate}
\item  There exists a generic leaf $(X, X')$ of $\Lambda^+_{\phi}$ or of
    $\Lambda^-_{\phi}$ such that $Q_Y(X) \neq Q_Y(X')$. 
  \item  There exists a diagonal leaf (i.e. the concatenation of two half-lines) $(X, X')$
    of $L_{Q}(T^{-}_{\phi})$ or $L_Q (T^+ _{\phi})$ such that $Q_Y(X) \neq Q_Y(X')$.
\end{enumerate}
\end{prop}
\begin{proof}
The proof is again the same as that of~\cite[Proposition~3.22]{Gupta}, using the
generalised version of the $\mathcal Q$-map given in~\cite{GH}, (Propositions~\ref{Q-map}
and~\ref{SmallBCC1}). All the intermediate steps still hold in our context. Connectedness of
Whitehead graphs is used here, which is safe because we are assuming $[\phi]$ is $\G$-irreducible.
\end{proof}

We also need to ensure that limit trees have dense orbit, but this is already part of
literature.
\begin{lem}[{\cite[Lemma~4.5]{Gupta}}] Let $[\phi]\in\Out(\G)$ be $\G$-primitive irreducible. Then
the trees $T^+_{\phi}$ and $T^-_{\phi}$ have dense orbits.
\end{lem}

\subsection{At least one of the laminations is long in any tree of the boundary}
The key lemma here is the following.

\begin{lem}[{\cite[Lemma~3.26]{Gupta}}]\label{LongLamination1}
  Let $T \in \overline{\O(\G)}$.  Then there exists a tree $T_0 \in \O(\G)$, an equivariant map
  $h : T_0 \to T$, and a bi-infinite geodesic $\gamma_0$ representing a generic leaf of
  $\Lambda^+_{\phi} $ or of $\Lambda^-_{\phi}$, such that $h(\gamma_0)$ has diameter greater
  than  $2BCC(h)$.
\end{lem}
\begin{proof}
As in the proof of \cite[Lemma~3.26]{Gupta} (see also~\cite{LL}), we distinguish three
cases. We just give a sketch of the proof for each case and we refer to original proof for the
details.
\begin{itemize}
\item Suppose that $T$ has dense orbits. First, we note that arc stabilisers of $T$ are trivial
  (this is true by~\cite[Proposition~5.17]{Horbez}). In this case, the conclusion is a
  consequence of Propositions~\ref{Q-map1} and~\ref{SmallBCC1} exactly as in the proof
  of~\cite[Lemma~3.26]{Gupta}. 

\item Suppose that $T$ does not have dense orbits and that it is not simplicial. This sub-case, can be
  reduced to the first case (of a tree with dense orbits), by collapsing the simplicial part,
  exactly as in~\cite{Gupta}.

\item Suppose that $T$ is simplicial. In this case, we have to show that a generic leaf of the
  attracting lamination cannot be contained in the boundary $\partial B$ of some vertex
  stabiliser $B$ in $T$. In other words, we want to prove that the lamination is not contained
  in any vertex stabiliser of a (non-trivial) tree in the boundary of
  $\O(\G)$. By~\cite[Corollary~5.5]{Horbez}, point stabilisers of trees in the boundary have
  finite rank and, more specifically, their rank is bounded above by $\rank(\G)$. It follows
  that they have infinite index and so they cannot carry the 
  lamination, by~\ref{Carries1}.
\end{itemize}
\end{proof}

\subsection{Proof of Theorem~\ref{NSdyn}} Everything flows as in the proof
of~\cite[Theorem~C]{Gupta}. The point-wise convergence of Theorem~\ref{North-South1},
follows directly from Proposition~\ref{Convergence1} and Lemma~\ref{LongLamination1}.
The locally uniform convergence then follows, because of the compactness of $\overline{\mathbb{P} \O(\G)}$.

\section{Discreteness of the product of limit trees of an irreducible automorphism}

\label{s6}

\subsection{Dynamics of train track maps}
Let $\G=(\{G_1,\dots,G_k\},r)$ be a free factor system of a group $G$.

In this section we prove the discreteness of the $G$-action on the product of the two limit
trees of irreducible automorphisms with exponential growth. We do not assume primitivity here, so powers of the automorphism may be reducible. Similar results in the free case have been proved in~\cite{BFH-Laminations0} and in the free product case in~\cite{DahLi}.

In particular, both those papers have a precise analogue of Proposition~\ref{uniform}; the argument in \cite{DahLi}, which also deals with free products, relies on a technical hypothesis of no twinned subgroups. Effectively, this allows that paper to argue that the ``angles" (the vertex group elements one encounters) to remain bounded, and hence one observes similar behaviour to that seen in \cite{BFH-Laminations0}. However, we obtain finiteness conditions in a slightly different way by observing that there are finitely many orbits of paths which occur as the train track image of an edge. However, while this idea is straightforward, it is somewhat more difficult to implement. 

We also observe that a version of Theorem~\ref{DiscretnessOfMinSet}, in the free group case, is
proved in \cite{BFH-Laminations0}, but in a slightly different way. There, the main argument
deals with the case where there is no ``closed INP", whose analogue is that no $\G$-hyperbolic
element becomes elliptic in the limit tree. (A version of Theorem~\ref{DiscretnessOfMinSet} is
also proved in \cite{DahLi}, again with the same technical assumption of no twinned subgroups.)
The other case - where there is a closed INP - is dealt with in \cite{BFH-Laminations0} via
surface theory.

Here we need to face the fact that our deformation spaces have a locally infinite simplicial
structure. For this reason we need different arguments. However, our approach allows us to deal
with both the above cases at the same time. 
We recall that we are using the square-bracket notation for reduced paths (see the start of Section~\ref{s8bcc}).

\begin{defn} Let $f:X\to X$ be a train track map representing some $[\phi]\in\Out(\G)$. 
	Let $L$ be a periodic line in $X$. The {\em number of turns} of $L$ is the number of turns
	appearing in a fundamental domain. We say that $L$ {\em splits}  as a concatenation  of
	paths, if we can write a fundamental domain of $L$ as $\rho_1\dots\rho_n$ such that for any
	$i$ we have $$[f^i(\rho_1\dots\rho_n)]=[f^i(\rho_1)]\dots[f^i(\rho_n)]$$ as a cyclically
	reduced path. 
\end{defn}

\begin{defn}  Let $f:X\to X$ be a train track map representing some $[\phi]\in\Out(\G)$. 
	An  {\em $f$-piece}, or simply a {\em piece } is an edge-path $p$ which appears as
        sub-path of $f(e)$, with $e$ edge,  or $f(e_1e_2)$ with $e_i$ edges meeting at a legal
        free turn (i.e. a turn at a free vertex). 
\end{defn}

\begin{defn}
	For a not necessarily simplicial path $p$ in a simplicial tree, its {\em simplicial
          closure} is the smallest  simplicial path containing $p$. In other words, the
        simplicial closure of $p$ is obtained by prolonging the extremities of $p$ till the
        next vertex. 
\end{defn}

We recall that we defined the critical constant $cc(f)$ of a map $f$ in
Definition~\ref{defccf}, and Nielsen paths in Definition~\ref{defn287}.

\begin{prop}\label{uniform} Let $f:X\to X$ be a simplicial train track map representing some
  $[\phi]\in\Out(\G)$, with $\Lip(f)=\lambda>1$. Let $C=cc(f)+1$. Then there exist explicit
  positive constants  $N,M\in\N$ (with $M=5N^2+N$), such that for any finite edge path, or periodic line
  $L$ in 
  $X$, one of the following holds true: 
	\begin{enumerate}
		\item $[f^M(L)]$ has less illegal turns than $L$.
		\item $[f^M(L)]$ has a legal sub-path of length more than $C$.
		\item $L$ splits (not necessarily at vertex-points) as a concatenations of paths
		$\rho_1\dots\rho_\kappa$ so that each $\rho_i$ is
		pre-periodic Nielsen path with at most one illegal turn.
	\end{enumerate}
	
	Moreover, the periodic behaviour of ppNp's starts before $N$ iterates, and  with period less
	than $N$.
	
	The same conclusion holds true for finite paths whose endpoints are not necessarily
        vertices,  with $(3)$ above replaced by:
	\begin{itemize}
		\item[$(3')$] There exists $L'$, contained in the simplicial closure of $L$, such that:
		\begin{enumerate}
			\item[$(i)$] $L'$ splits (not necessarily at vertex-points) as a concatenation of paths
			$\rho_1\dots\rho_\kappa$ so that each $\rho_i$ is  pre-periodic Nielsen path with at most one
			illegal turn;
			\item[$(ii)$] endpoints of $L$ are at distance less than   $C/\lambda^{M}$ to those of $L'$;
			\item[$(iii)$] if the initial point $x$ of $L$ is in $L'$, then $f^M(x)$ is in the same edge as
			the image   of the initial point of $L'$ (with a corresponding statement for terminal points).
		\end{enumerate}
		
	\end{itemize}
	
\end{prop}
\begin{proof}
	Since powers of train track maps are also train track, by replacing $f$ with some
        power,  we may freely assume  that $l_X(f(e)) >C$ for any edge. (Note that if $f$ represents an irreducible $[\phi]$, then $f^n$ represents $[\phi^n]$ which might not be irreducible. However, it will still be the case that $f^n$ is a train track map).
We now set constants (whose role will become clear along the proof): 
\begin{itemize}
	\item $M_0$ is one plus the number of orbit of pieces (which is finite);  
	\item $m=(M_0^2+M_0)^2$;
	\item $Q_m$ is the number of orbit of turns at non-free vertices that appears in iterates
	$f^n(p)$ where  $p$ runs   over the set of pieces, and $n=1,\dots, m$ ($Q_m$ is a finite
	number);
	\item $N_0=m(Q_m+2)^2+1$;
	\item $N=mN_0$;
        \item $M=5N^2+N$.
\end{itemize}

	We give the proof in case $L$ is a finite edge path, by analysing what happen to
	maximal legal  sub-paths of $L$; the case where $L$ is a periodic line
	follows by applying our reasoning to a fundamental domain (and make cyclic reductions). Also,
	note that the case where $L$ is legal easily reduces to $(3)'$ with no ppNp appearing (hence $L$
	is just in the neighbourhood of a point), so we
	may assume $L$ contains at least one illegal turn.
	
	In case $L$ is {\em not simplicial}, we refer to non-simplicial maximal legal sub-paths of
	$L$ (that possibly arise only at its extremities), as {\em tails}.
	
	\medskip
	
	First, we observe that if $(1)$ holds for some iterate $f^n$ with $n\leq M$, then
	by train track properties, it holds also for  $f^M$. The same is true for $(2)$
        by Lemma~\ref{grow}.

        Now, we suppose that we have a path $L$ for which $(1)$ fails (in particular it fails for any
	$n\leq M$). Then the  number of illegal turns in $f^n(L)$ remains constant.
	This implies that, in calculating
	$[f(L)]$,  we apply $f$ to each maximal legal sub-path of $L$, then cancel, and we are
	assured that some portion of  the  image of that path survives, and that the new turns
	formed after cancellations are illegal ones.
	
	For any $\alpha$ maximal legal sub-path of $L$, which is not a tail, we denote by $\alpha_n$
	the corresponding  maximal sub-path in $[f^n(L)]$, i.e. the portion of $f^n(\alpha)$ that
	survives after  cancellations (for $1\leq n\leq M$).
	If $\alpha$ is a tail, then we define $\alpha_n$ to be the simplicial closure of
	the surviving portion. So $\alpha_n$ is a simplicial path in any case.  Note
	that since $f$ is simplicial and expanding, then
	$$f^{-1}(\alpha_n)\subseteq \alpha_{n-1}$$
	also in case of tails.

	\medskip

	Now we assume that also $(2)$ fails, and prove that in that case $(3)$ is true.
	Since $f$-images of edges are longer than $C$, the
	$f$-preimages of  legal paths  we see in $[f^n(L)]$  (for $1 \leq n\leq M$) crosses at most two
	edges. In particular any maximal legal sub-path of $[f^n(L)]$ consists of at most $2$
	pieces. Note that a legal path may a priori be divided in  pieces in different ways.  Here we
	consider the subdivision of $\alpha_n$ given by $f^{-1}(\alpha_n)$. Note that from the
        definition of piece it follows that if
        $\alpha_n$ consists of two pieces, then they meet at a non-free vertex.

	For each $1\leq n\leq M$, and any
	$\alpha_n$ maximal legal sub-path of $[f^n(L)]$, we define
	$\surv(\alpha_n)=f^{-M+n}(\alpha_{M})$ (the portion of $\alpha_n$ that survives all $M$ iterates.)
	
	Any such $\alpha_n$ therefore splits in three (not-necessarily simplicial)
	sub-paths $$\alpha_n=\lef(\alpha_n)\surv(\alpha_n)\ri(\alpha_n).$$

	\begin{rem}\label{rm1}
		We observe that:
		\begin{itemize}
			\item Since $f$-images of edges have length more than $C$, then any $\surv(\alpha_n)$ contains
			at most one vertex.
			\item If $\surv(\alpha_n)$ contains a vertex $v$, then $\surv(\alpha_{n+j})=f^j(\surv(\alpha_n))$
			contains the vertex $f^j(v)$.
			\item If $\alpha\beta$ are consecutive maximal legal sub-paths, hence forming an illegal
			turn, then cancellations between $f(\alpha_n)$ and $f(\beta_n)$ occur in the sub-path
			$f(\ri(\alpha_n)\lef(\beta_n))$.
			\item $\surv(\alpha_n)$ is never involved in cancellations.
			\item $\ri(\alpha_n)$ and $\lef(\alpha_n)$ are
			eventually cancelled by $f^M$, unless $\alpha_n$ is a tail.
			\item If $\alpha$ is a non-tail extremal maximal legal sub-path of $L$, say on the left-side
			(the start of $L$), then $\lef(\alpha_n)$ is empty, because no cancellations occur on its
			left-side (same for the end of $L$).
			\item If $\alpha$ is a tail, say on the left-side, then $\lef(\alpha_n)$ is just the portion
			of the edge containing the beginning of $[f^n(L)]$, but  which is not in $[f^n(L)]$.
		\end{itemize}
	\end{rem}

	Next we focus our attention on iterates till $N$. Pick two consecutive such maximal legal sub-paths $\alpha,\beta$ and look at
	$\alpha_n,\beta_n$ (for $1\leq n\leq N$).
	
	\begin{claim}\label{cl1}
		There exist $1\leq s<t\leq N$ and points
		$a_s\in \alpha_s, a_t\in \alpha_t, b_s\in \beta_s, b_t\in \beta_t$, such that
		\begin{itemize}
			\item $f^{t-s}(a_s)=a_t,  f^{t-s}(b_s)=b_t$ (hence $[a_t,b_t]=[f^{t-s}([a_s,b_s])]$);
			\item there is $h\in G$ so that $[a_t,b_t]=h[a_s,b_s]$; (so $[a_s,b_s]$ is a pNp
			of period $t-s<N$, containing a single illegal turn: that formed at the concatenation point
			of $\alpha_s\beta_s$);
			\item $a_t$ is the unique fixed point of the restriction to $\alpha_t$ of $hf^{s-t}$; $b_t$
			is the the unique fixed point of the restriction to $\beta_t$ of $hf^{s-t}$.
			\item $a_s$ is not internal to $\ri(\alpha_s)$ and $b_s$ is not internal to $\lef(\beta_s)$.
		\end{itemize}
	\end{claim}
	\begin{proof}
		The proof is based on pigeon hole principle. As mentioned,  any $\alpha_n$ consists of either one
		or two     pieces.  In case $\alpha_n$ consists of two pieces,
		we denote by $v_n$ the non-free vertex  separating the pieces of $\alpha_n$, and similarly
		we define $w_n$ as the vertex separating the  pieces of $\beta_n$, if any.
		
		By definition of constants, we have $M_0-1$ orbit of pieces. So the possible configurations of
		orbit of pieces that we read in a maximal legal sub-paths are less than
		$M_0^2+M_0$. Consequently,  the configuration  of orbit pieces that we  read in paths
		$\sigma_n=\alpha_n\beta_n$, runs over a set of cardinality strictly less than
		$m=(M_0^2+M_0)^2$. Let $\mathcal T$ be the set of orbit of turns at non-free vertices that
		appears in iterates of pieces up to power $m$ (the cardinality of this family is $Q_m$ by
		definition).
		
		Now, we subdivide the
		family $\Sigma=\{\sigma_n,\  1\leq n\leq N\}$  in $N_0$ subfamilies
		$\Sigma_\nu$ each made of $m$  consecutive elements. By pigeon hole  principle any such $\Sigma_\nu$
		contains a pair of paths $\sigma_i,\sigma_j$ (with $i<j$) with the same configuration of
		orbit of pieces. To any such pair we associate a tag $(Conf,Turn_\alpha,Turn_\beta)$ as
		follow: $Conf$ is just the configuration of orbit of pieces. We define now $Turn_\alpha$,  the
		other being defined in the same way.
		\begin{itemize}
			\item $Turn_\alpha=1$ if $\alpha_i$ consists of a single piece.
			\item $Turn_\alpha=Per$ if $\alpha_i$ is made of two pieces and
                          $f^{j-i}(v_i)=v_j$. ($Per$ is just a label.)
			\item Finally, if none of the above occur, we set $Turn_\alpha$ to be the orbit of
			$\tau_\alpha$, the turn we read in
			$\alpha_j$ at $v_j$. Note that
			if $\alpha_i$ consists of two pieces and $v_j\neq f^{j-i}(v_i)$, then $\tau_\alpha$
			belongs to $\mathcal T$.
			
		\end{itemize}
		
		In last case, the possibilities for the orbit of $\tau_\alpha$ are at most
		$Q_m$ (same for $\tau_\beta$.)
		It follows that the cardinality of the set of possible tags $(Conf,Turn_\alpha,Turn_\beta)$ is
		$m(Q_m+2)^2$. Since $N_0=m(Q_m+2)^2+1$ we must have at least one repetition. That is to say, we
		find two pairs $(\sigma_{i_0},\sigma_{j_0}),(\sigma_{i_1},\sigma_{j_1})$
		(with $1\leq i_0\leq j_0<i_1<j_1\leq N$) with same tags. Now we have three cases:
		
		\medskip
		
		\noindent{\bf Case 1:} Both $Turn_\alpha$ and $Turn_\beta$ are different from $Per$. In this case we
		set $s=j_0$ and $t=j_1$. Let's focus on $\alpha$-paths. If $Turn_\alpha=1$ then $\alpha_s$ and
		$\alpha_t$ both consist of single pieces, and in the same orbit.
		If $Turn_\alpha=\tau_\alpha$, then $\alpha_s$ and  $\alpha_t$ both consist of two pieces in the
		same  respective orbit, and whose middle turns are also in the same orbit. So, also in this
		case we have that $\alpha_s$ and $\alpha_t$ are in the same orbit. The same
		reasoning shows  that $\beta_s$ and $\beta_t$ are in the same orbit.

		Thus, there exist
		$h,h'\in G$ such that $\alpha_t=h\alpha_s$ and $\beta_t=h'\beta_s$. Both turns we read
		at (the concatenation points of) $\alpha_s\beta_s$ and $\alpha_t\beta_t$ are illegal. Since
		legality of turns is
		invariant under the action of $G$, we have that the turn we read in $(h\alpha) (h\beta)$ is
		illegal. On the other hand the illegal turn at $\alpha_t\beta_t$ is
		$(h\alpha_s)(h'\beta_s)$.  This forces $h=h'$, and in particular the whole path
		$\alpha_s\beta_s$ is in the same orbit of $\alpha_t\beta_t$.
		
		Now, we set $a_t$ to be the unique fixed point of
		the restriction to $\alpha_t$ of the contraction $hf^{s-t}$ and set
		$a_s=f^{s-t}(a_t)$. Similarly we define  $b_t$ and $b_s$. Thus $a_t=f^{t-s}(a_s)=ha_s$, and
		the same holds for $b$-points.

		\medskip
		
		\noindent{\bf Case 2:} $Turn_\alpha=Turn_\beta=Per$. In this case we set $s=i_0,t=j_0$ (note that
		this choice is different from that of Case $1$). The paths
		$[v_s,w_s]$ and $[v_t,w_t]$ both
		consists of two pieces in the
		same respective orbit, meeting at illegal turns. As in Case $1$, we deduce that in fact
		the whole $[v_s,w_s]$ is in the same orbit of $[v_t,w_t]$.       In this case we set
		$a_s=v_s, a_t=v_t, b_s=w_s,b_t=w_t$.   Note that if $[v_t,w_t]=h[v_s,w_s]$ then $v_t$ is
		the unique fixed point of the restriction to $\alpha_t$ of $hf^{t-s}$ (and similarly for
		$w_t$).

		\medskip
		
		\noindent{\bf Case 3:} One of $Turn$'s, say $Turn_\alpha$, is $Per$ and the other, $Turn_\beta$, is
		different. In this case we set $s=i_0,t=j_0$ (as in Case $2$). As above, we see that there is
		$h\in G$  so that the concatenation of the right-side piece of $\alpha_t$ with the left-side
		piece of $\beta_t$ is the $h$-translate of the concatenation of corresponding pieces in
		$\alpha_s,\beta_s$.
		If $Turn_\beta=1$, then as above we see that $\beta_t=h\beta_s$, and  we define
		$b_t$ as the unique fixed point of  $hf^{s-t}$ in $\beta_t$, $b_s=f^{s-t}(b_t)$, $a_s=v_s$,
		and $a_t=v_t$.
		
		So we are left with the case $Turn_\beta=\tau_\beta$. Let
		$\tau_s=(e,e')$ be the turn that we read in $\beta_s$ at $w_s$, and let $\tau_t$ the turn we
		read at $w_t$. Since the configurations of
		pieces are the same at iterates $s,t$, we know that there is $h'\in G$ so that
		$\tau_t=(he,h'e')$ (note that $h^{-1}h'$ is in the stabiliser of $w_s$).
		Now we define $H:\beta_s\to \beta_t$ to be $h$ on the left-side piece,
		and $h'$ on the right-side one; and set
		$b_t$ to be the unique fixed point of contraction $Hf^{s-t}:\beta_t\to\beta_t$, and $b_s=f^{s-t}(b_t)$.
		
		In order to have $[a_t,b_t]=h[a_s,b_s]$, we have to prove that if $w_s$ is in $[a_s,b_s]$,
		that is to say if $w_s$ is on the left side of $b_s$, then $h=h'$. In this case, since
		$f^{t-s}(w_s)\neq w_t$ and since $f$ is expanding, then
		$f^{t-s}(w_s)$ is on the left side of $w_t$, possibly on the cancelled region.
		Now we  iterate $f$ for $(t-s)$ more times (note that since $t=j_0$ we have enough room to
		iterate $(t-s)$ times).
		
		If $f^{t-s}(w_s)$ is in $\beta_t$ (that is to say, it is not in a cancelled region), then
		$f^{t-s}(w_t)$ is in $\beta_{t+t-s}$, and  from
		$[v_t,w_t]$ to $[v_{t+(t-s)},w_{t+(t-s)}]$ we see the same cancellations we had from
		$[v_s,w_s]$ to $[v_t,w_t]$. It follows that $f^{t-s}(\tau_t)$ is in the same orbit of
		$f^{t-s}(\tau_s)$ and this forces $h=h'$.
		
		Similarly, if $f^{t-s}(w_s)$ is cancelled, then $f^{t-s}(w_t)$ must also be cancelled ---
		overlapping a turn in the image of $\alpha_t$ being in the same orbit as $f^{t-s}(\tau_s)$ ---
		otherwise the turn we read at concatenation point of $\alpha_{t+(t-s)}\beta_{t+(t-s)}$ would
		become legal, contradicting the fact that the number of illegal turns stay constant (and cancellations on the right-side of $\beta$ and on the left-side of $\alpha$ never
		touch    the illegal turn between $\alpha$ and $\beta$. Remark~\ref{rm1} point four).
		Again, $f^{t-s}(\tau_t)$ and $f^{t-s}(\tau_s)$ are in the same orbit and thus
		$h=h'$.

		\medskip
		
		In all three cases, we proved first three properties. We check now last one.  We prove that
		$a_s$ is not in the interior of $\ri(\alpha_s)$, the same  reasoning proving that $b_s$ is not in the interior of $\lef(\beta_s)$.
		
		We already
		proved that $[a_s,b_s]$ is a pNp.
		A priori $a_s$ could belong to $\ri(\alpha_s)$. If
		$b_s\in\lef(\beta_s)\cup\surv(\beta_s)$ then it is
		clear (Remark~\ref{rm1} point three) that cancellations we see in subsequent iterations are
		the same we
		see from $[a_s,b_s]$ to $[a_t,b_t]$ and in particular $a_n$ is never cancelled, so a
		posteriori $a_s$ would belong to $\surv(\alpha_s)$. But now note that the very same holds
		true also if
		$b_s\in\ri(\beta_s)$. Indeed, in this case $f^n(b_s)$ may, a priori, eventually disappear from
		$\beta_{s+n}$; but still,
		cancellations with $\alpha$-paths arise in a sub-path which is in the image of
		$f^{s+n}[a_s,b_s]$ because $b_s$
		is on the right side of $\surv(\beta_s)$ which is never involved in cancellations.
				The claim is proved.
		
	\end{proof}

	If $\alpha\beta$ are as in the claim, then by pulling back $a_s,b_s$ to (the simplicial
	closures of) $\alpha,\beta$ we find
	a ppNp in (the simplicial closure of) $L$. We set $a=f^{-s}(a_s)$ and $a_n=f^n(a_s)$ for any
	$1\leq n\leq M$. Similary we
	define $b$-points. The paths $[a_n,b_n]$
	evolves till $n=s<N$, then starts with a (orbit) periodic behaviour with period $p=t-s<N$. The
	idea is that this provide the requested splitting of $L$.
	
	Let $\gamma$ be the maximal legal sub-path of $L$ on the left-side of $\alpha$ (if any). Claim~\ref{cl1} can then be applied to the subpath $\gamma \alpha$. We wish to show that the point in $\alpha$ obtained from that process is the same as the one obtained by applying Claim~\ref{cl1} to $\alpha \beta$.

	Let $c,a'$ in $\gamma,\alpha$ respectively, be the points provided by Claim~\ref{cl1}, so that $[c,a']$ is
	ppNp.  As above we denote $c_n=f^n(c)$ and $a'_n=f^n(a')$.
	
	\begin{claim}
		$a=a'$. That is, applying Claim~\ref{cl1} locally results in well-defined points globally. 
	\end{claim}
	\begin{proof}
		Let $s'$ be the iterate where periodicity of $[c_n,a'_n]$ starts, and let $p'$ be the
		period. Without loss of generality we may assume $s'\leq s$. Since $[c_{s'},a'_{s'}]$ is a
		pNp, then  also $[c_{s},a'_{s}]$ is a pNp with the same period $p'$. Let $P=pp'$, note that
		$P\leq N^2$. Both $[c_s,a'_s]$ and $[a_s,b_s]$ are $P$-periodic. At time $s$ the segment
		$[a_s,b_s]$ is contained in $[f^s(L)]$ but a priori its extremities may get cancelled from
		$[f^n(L)]$ in
		subsequent iterations. The same holds for $[c_{s'},a'_{s'}]$.
		
		Suppose that images of $a,a'$ are not cancelled till the next three iterations of $f^P$ (note
		that $s+3P\leq N+3N^2\leq M$).  Since $f$-images of edges have length more than $C$ and $L$
		contains no
		legal sub-path of that length, then for $n=s,s+P,s+2P$  the segment $[a_n,a'_n]$
		--- which is  the pre-image of $[a_{s+3P},a'_{s+3P}]$ --- contains at  most one
		vertex. Therefore there are two iterates in the first three steps so that $[a_n,a'_n]$
		contains the same number   of vertices which is either zero or one. Now, since $f$ is
		expanding and $a_n$ and $a'_n$ are   orbit-periodic, this forces $a_n=a'_n$, so $a=a'$
		(because $[a,a']$ is a legal path). (To be precise here we don't use only the periodicity of
		$a,a'$ but the periodicity of the pNp's $[c_s,a'_s]$ and $[a_s,b_s]$ because we need the orbit
		periodicity of the oriented edges containing $a_n$ and $a'_n$).
		
		We end the proof by proving that $a$ is not cancelled in iterations till $s+3P$. (The same
		reasoning will work for $a'$). Suppose the contrary. Let $\tau_n$ be the illegal turn of
		$[c_n,a'_n]$ ($\tau_n$ is in the orbit of $\tau$, but $\tau_n\neq f^n(\tau)$).  Since $a$
		is cancelled, then  $a\in\lef(\alpha)$ (because we know it  is not in
		$\ri(\alpha)$) and in particular $[a_{s+3P},b_{s+3P}]$ contains $\tau_{s+3P}$ on the
		left-side of $\surv(\alpha_{s+3P})$. By periodicity $[a_{s+4P},b_{s+4P}]$ contains the
		segment $[f^P(\tau_{s+3P}),\tau_{s+4P}]$, still in the left-side of
		$\surv(\alpha_{s+4P})$. In particular it contains a whole edge, hence $[a_{s+5p},b_{s+5P}]$
		contains a legal segment of length more than $C$ on the left-side of
		$\surv(\alpha_{s+5P})$. By periodicity, $[a_s,b_s]$ contains a legal sub-path of length more
		than $C$ which contradicts the fact that $(2)$ fails (note that $s+5P\leq N+5N^2=M$).

	\end{proof}

	Now, if $L$ is simplicial, then we have provided a splitting of $L$ in ppNp's, as required. In the
	general case, we have a splitting of the simplicial closure of $L$ so that interior paths are
	ppNp's.
	
	Now let's focus on tails. Let $\alpha$ be a tail, say the starting one, and let $a$ be the
	point in the simplicial  closure of $\alpha$ given by Claim~\ref{cl1}. So $a$ is the
        starting 
	point of our $L'$. Let $x$ be the starting
	point of $L$ (note that $x$ is never cancelled till $M$ iterates because $(1)$ fails). Point
	$a$ may lie either on the left or the right side of $x$.	
	If $a$ is on the
	left-side of $x$, i.e. $x\in L'$, then the image $f^M(x)$ and $f^M(a)$ are in the same edge
	(and edges have length less then $C$ because $(2)$ fails), and in particular at distance less
	than $C$ apart.	
	If $x$ is not in  $L'$, then the segment
	$[f^M(x),f^M(a)]$ is shorter than
	$C$ because $[x,a]$ is legal and never affected by cancelations, and $(2)$ fails.  In both cases $d_T(x,a)\leq C/\lambda^M$.
	
	\medskip
	
	The proof of Proposition~\ref{uniform} is now complete.

\end{proof}

The following is now immediate, since we may iterate Proposition~\ref{uniform}, bearing in mind Lemma~\ref{grow}.

\begin{cor}
	\label{growconst}
	In the hypothesis of Proposition~\ref{uniform},
	for any $C_1 > 0$  there exists an $M_1 \in \N$ such that:
	For any finite edge path or periodic line $L$ in $X$, one of the following holds true:
	\begin{enumerate}
		\item $[f^{M_1} (L)]$ has less illegal turns than $L$.
		\item $[f^{M_1}(L)]$ has a legal sub-path of length more than $C_1$.
		\item $L$ splits (not necessarily at vertex-points) as a concatenations of paths
		$\rho_1\dots\rho_\kappa$ so that each $\rho_i$ is
		pre-periodic Nielsen path with at most one illegal turn.
	\end{enumerate}

	The same conclusion holds true for finite paths whose endpoints are not necessarily vertices,
	with $(3)$ above replaced by:

	\begin{itemize}
		\item[$(3')$]
		\begin{enumerate}
			\item[$(i)$] $L$ splits as a concatenations of paths
			$\delta_0 \rho_1\dots\rho_\kappa \delta_1 $ so that each $\rho_i$ is  pre-periodic Nielsen path with at most one
			illegal turn;
			\item[$(ii)$]  $\delta_0, \delta_1$ each cross at most one illegal turn;
			\item[$(iii)$] $\delta_0, \delta_1$ each have length at most $2cc(f)$.
		\end{enumerate}
		
	\end{itemize}
\end{cor}

\begin{defn}
	Let $f:X \to X$ be a train track map. A path $\beta$ in $X$ is called {\em pre-legal} if
        for some $n \in \N$, $[f^n(\beta)]$ is legal. 
\end{defn}

\begin{lem}[The $2/3$-lemma]
	\label{twothirds} Let $f:X\to X$ be a simplicial train track map representing some
        $[\phi]\in\Out(\G)$, with $\Lip(f)=\lambda>1$.
	Let $L$ be  either a finite path or a periodic line in $X$. Let $M=5N^2+N$ be the
        constants of 
        Proposition~\ref{uniform}. Suppose that no legal
        sub-paths of length more than $C=cc(f)+1$
	appears in iterates $[f^n(L)]$ for $1\leq n\leq M$, but that $[f^n(L)]$ becomes eventually
	completely legal for some $n>M$ (that is, $L$ is pre-legal). Then
	
	$$ \#\{\text{illegal turns of } [f^M(L)]\}\leq \frac{2}{3}(\#\{\text{illegal
		turns of } L\}+1).$$
              In particular, if $\#\{\text{illegal turns of } L\}>3$, then\footnote{Because if $x>3$,
                then $\frac{2}{3}(x+1)<\frac{8}{9}x$}
        $$  \#\{\text{illegal turns of } [f^M(L)]\}< \frac{8}{9}\#\{\text{illegal turns of } L\}.$$
\end{lem}
\begin{proof}
We order $L$ and $f^{M}(L)$ accordingly. Let $\sigma_0\dots\sigma_h$ be the subdivision of $[f^{M}(L)]$ in maximal legal sub-paths. Let
$S_i$ be the starting point of $\sigma_i$, which coincides with the ending point of
$\sigma_{i-1}$.  Let $V_1=\min f^{-M}(S_1)$ and $W_1=\max f^{-M}(S_1)$. (We may have $V_1=W_1$
if for example $f^{-M}(S_1)=V_1$ and the turn at $V_1$ is legal, which can happen because our
maps are train track  for $\simfk$, not necessarily for $\sim_f$.) Define $\gamma_1=[V_1,W_1]$. The set $f^{-M}(S_2)\cap
\{x>W_1\}$ is non empty just because $f^M(\gamma_1)$ can be retracted to $S_1$ in
$[f^{M}(L)]$ (also, note that $f^M(\gamma_1)$ is a tree which is rooted at $S_1$, but a priori
it could contain the segment $\sigma_2$, so one could possibly see some pre-image of $S_2$ in $\gamma_1$) . Let $V_2,W_2$ be respectively the min and max of  $f^{-M}(S_2)\cap \{x>W_1\}$,
and define $\gamma_2=[V_2,W_2]$. Recursively define $V_i,W_i,\gamma_i$  in the same way, and
define sub-paths
$\xi_0=\{x\leq V_1\}$, $\xi_i=[W_i,V_{i+1}]$ for $i<h$, and $\xi_h=\{x\geq W_h\}$. (So $\xi_i$
is a pre-image of $\sigma_i$ in a broad sense).

Since any $\gamma_i$ gets cancelled in
$[f^{M}(L)]$, then it contains at least one illegal turn ($\gamma_i$ may be a single point at an
illegal turn of $L$, in this case we abuse notation and still say that $\gamma_i$ is a path
containing one illegal turn). 
 Suppose that $\gamma_i$ contains only one illegal turn. For any  $x_{i}\in \xi_{i-1},
 y_i\in\xi_i$, Proposition~\ref{uniform} applies to the path $[x_i,y_i]$, and taking those
 points sufficiently close to $\gamma_i$, we may assure that we are in the situation $(iii)$ of
 $(3)'$, for both $x_i,y_i$ (so the images of endpoints of the ppNp provided by
 Proposition~\ref{uniform},  are in $\sigma_{i-1},\sigma_i$
 respectively). Proposition~\ref{uniform} in particular implies that we
 find $z_i\in \xi_{i-1},t_i\in\xi_i$ so that $[z_i,t_i]\subseteq \xi_{i-1}\gamma_i\xi_i$ is
 ppNp, with periodicity starting before $N$ iterates\footnote{See Definition~\ref{defn287}}, and with period less than $N$, where $N$
 is as in Proposition~\ref{uniform}. Note that by periodicity, and since $f$ is expanding,
 then either $t_i$ coincides or it is  on the left of $z_{i+1}$. (Indeed, since $f$ is
 expanding, the $f^n$-images of $[z_i,t_i]$ are longer and longer paths, in which some central
 portion cancels because of periodicity. Any point $x$ in the interior of $[z_i,t_i]$ eventually cancels in
 iterates because otherwise the distance of $f^n(x)$ from $f^n(z_i)$ would exceed the distance
 between $f^n(z_i)$ and $f^n(t_i)$. By periodicity, $z_{i+1}$ does not cancel, and so it
 must stay on the right of $t_i$.)

 We say that $\gamma_i$ is:
\begin{itemize}
\item {\em Periodic}, if $\gamma_i$ contains only one illegal turn, and such turn will be
  called periodic;
\item {\em Non-periodic}, if $\gamma_i$ contains at least two illegal turns.
\end{itemize}
Since $[f^n(L)]$ becomes eventually legal, all illegal turns must disappear, and they can
disappear in two ways: either they became legal after some iteration-cancellation of the turn itself, or they are
cancelled by overlapping the image of some other illegal turn (since $f$ is train-track, no new
illegal turns are created). 
	
Since illegal turns in the periodic $\gamma_i$'s remain illegal forever (because they are part of a
pNp), then they must cancel by
overlapping or just because they are at extremities of $L$ and the ppNp provided by
Proposition~\ref{uniform} strictly contains that extremity of $L$. There are at most 2 illegal turns of the last kind, at the extremities of $L$.
	
\begin{claim}
  The illegal turns of iterates of two different periodic $\gamma_i$'s never overlap.
\end{claim}
\begin{proof}
  Suppose the contrary. Let $\gamma_i$ and $\gamma_j$ be two periodic paths whose illegal turns
  eventually overlap. Let $\tau$ be the illegal turn of $\gamma_i$ and $\omega$ that of
  $\gamma_j$. Let  $\tau_n=f^n(\tau)$ and 
  $\omega_n=f^n(\omega)$. Let $s<N$ be such that both $\gamma_i$ and $\gamma_j$
  become periodic from step $s$ on, and let $p_i,p_j$ their periods, whose product
  $p=p_ip_j$ is less than $N^2$ (by Proposition~\ref{uniform}). Let $n_0$ be the first iterate when
  $\tau_{n_0}=\omega_{n_0}$. Let $N\leq q<N+p$ such that $q\equiv n_0\  (\operatorname{mod}
  p)$ (note that $N+p\leq N+N^2<M$). By
  periodicity $\tau_q$ is in the same orbit as $\tau_{n_0}$ and $\omega_q$ in the same orbit as
  $\omega_{n_0}$. Thus there is $g\in G$ such that $\omega_q=g\tau_q$. Note that $g\neq id$ because $\gamma_i$ and  $\gamma_j$ remains disjoint till iterate $M$. But now
  $f^{n_0-q}(\tau_q)=\tau_{n_0}=\omega_{n_0}=f^{n_0-q}(\omega_q)=f^{n_0-q}(g\tau_q)=g\tau_{n_0}$ forces $g=id$, a  contradiction. 
\end{proof}

Therefore, except possibly for the two extremal $\gamma_i$'s, we can associate 
to any periodic $\gamma_j$ the illegal turn $\tau_{\operatorname{canc}}(\gamma_j)$ of some other $\gamma_k$ which eventually cancels the unique illegal turn of $\gamma_j$. By the claim above,  the turn $\tau_{\operatorname{canc}}(\gamma_j)$ is one of the illegal turns of $[f^M(L)]$ that comes from a non-periodic
$\gamma_k$. Again by the claim, different periodic $\gamma_i$'s have associated different non-periodic
	turns, that is $\tau_{\operatorname{canc}}(\gamma_i)\neq
        \tau_{\operatorname{canc}}(\gamma_j)$ for $i\neq j$. That is,
        $\tau_{\operatorname{canc}}$ is injective.
	
	Let $A,B$ be respectively the number of periodic and non-periodic
	illegal turn we see in $[f^M(L)]$. Since $\tau_{\operatorname{canc}}$ is injective and
        defined possibly except for at most two $\gamma_i$'s,  we have $A-2\leq B$. It follows that
	
	$$\frac{(A-2)+B}{(A-2)+2B}\leq\frac{2}{3}$$
	
	hence
	$$A+B\leq \frac{2}{3}(A-2+2B)+2=\frac{2}{3}(A-2+2B+3)=\frac{2}{3}(A+2B+1).$$
	The number of illegal turns in $[f^M(L)]$ is $A+B$ by definition. Any non-periodic illegal turn
	in $[f^M(L)]$ contributes to at least two\footnote{Because any illegal turn of
          $[f^M(L)]$ which is not in a periodic $\gamma_i$, comes from a non-periodic
          $\gamma_i$, each of which contains at least two illegal turns.} illegal turns in $L$, so the number of illegal turns
	of $L$ is at least $A+2B$ and the lemma is proved.
\end{proof}

\begin{rem}
	The statement is sharp, as you can build a path with two illegal turns that survives till $M$ but
	then disappears, just by a concatenation of two pNp's to which we cut a suitable portion near
	the ends.
\end{rem}

\begin{defn}
For a simplicial $G$-tree $X$ denote by $a_X$ the length of the shortest edge of $X$.
\end{defn}

\begin{lem}
	\label{ratio}
	Let $[\phi]\in\Out(\G)$. Let $Y \in \O(\G)$ such that there exists a simplicial train
        track map 
        $f_Y: Y \to Y$ representing $[\phi^{-1}]$. (For example if $[\phi]$ is $\G$-irreducible
        and  $Y\in\Min(\phi^{-1})$ admits a simplicial train track). For
        any  constant $C_1>0$, and any $X\in\O(\G)$, set 
	$$ D=\frac{C_1}{a_Y} + 1 \qquad \qquad D' =D \Lambda(X,Y) \Lambda(Y,X).$$
	Then, with these constants, the following holds true for any $g\in\Hyp(\G)$:
        
	If $\frac{\ell_X(\phi^n(g))}{\ell_X(g)} < 1/D'$, then the axis of $g$ in $Y$ contains
        an $f_Y$-legal subpath of length at least $C_1$. 	 
\end{lem}

\begin{proof}
  First, observe  that if $\frac{\ell_X(\phi^n(g))}{\ell_X(g)} < 1/D'$, then
  $$
  \begin{array}{rcl}
	\frac{\ell_Y(\phi^{-n} \phi^n (g))}{\ell_Y(\phi^n(g))} &  =  &   \frac{\ell_Y(g)}{\ell_Y(\phi^n(g))} 
	 =   \left( \frac{\ell_X(g)}{\ell_X(\phi^n(g))} \right) \left( \frac{\ell_Y(g)}{\ell_X(g)} \right) \left(  \frac{l_X(\phi^n(g))}{l_Y(\phi^n(g))} \right)\\ \\
	& \geq  & \frac{l_X(g)}{l_X(\phi^n(g))} \frac{1}{\Lambda(Y,X)} \frac{1}{\Lambda(X,Y)}
                  >  D. 
	\end{array}
	$$

 Let $L$ be the axis of $\phi^n(g)$ in $Y$. Let $n_g$ be the number of $f_Y$-illegal turns in
a fundamental domain of $L$. Then if $[f_Y^{-n} (L)]$ would not contain any legal subpath of length at least $C_1$,
 since $f_Y$ is train track, we would get
			$$
			\ell_Y(g) \leq C_1 n_g.
			$$
			
			But also (since $f_Y$ is simplicial),
			
			$$
			\ell_Y(\phi^n(g)) \geq n_g a_Y.
			$$
			
			Hence, $\frac{\ell_Y(g)}{\ell_Y(\phi^n(g))} \leq \frac{C_1}{a_Y} < D$,
                        contradicting the above inequality.
\end{proof}

\begin{lem}
  \label{legalinverse}
 Let  $[\phi]\in\Out(\G)$ with $\lambda(\phi)=\lambda>1$. Let $X,Y\in\O(\G)$ such that there exist
 simplicial train track maps $f_X: X \to X, f_Y: Y \to Y$ representing $[\phi], [\phi^{-1}]$
 respectively.   Let $h: X \to Y$ be a straight $G$-equivariant map.

Then, for any  $C_2>0$, there exist $b \in \N$ and  $0 \leq L_0 \in \R$ such that, for any path
$\beta$ in  $X$, if $\beta$ is pre-legal (for $f_X$) and $l_X(\beta) \geq L_0$, then either:
\begin{enumerate}[(i)]
\item $[f_X^b(\beta)]$ contains an $f_X$-legal subpath of length at least $C_2$ or; 
\item $[h(\beta)]$ contains an $f_Y$-legal subpath of length at least $C_2$. \footnote{We note that this result is an analogue of Lemma 2.10 of \cite{BFH-Laminations0}. However, that Lemma requires the automorphism to be non-geometric, as it is otherwise false. We do not make the non-geometric assumption, and it is this that requires us to restrict to pre-legal paths.}
\end{enumerate}
\end{lem}

\begin{proof}
	We start by setting some constants.
	
	{\bf The constant $K$:}
	There exists a uniform constant $K$ (depending only on $X$) such that for any  path
        $\beta = [u,v]$ in $X$ there exists $g\in G$  such that
	\begin{enumerate}[(i)]
        \item $u$ is in the axis of $g$,
        \item $v \in [u, gu]$, and
        \item the distance from $v$ to $gu$ is bounded above by $K$.
	\end{enumerate}
	
	That is, if we look at the quotient graph of groups, we can complete the image of any
        $\beta$ to a cyclically reduced loop by adding a path of length at most $K$.
	
{\bf The constant $C_1$:} Let $C=cc(f_X)+1$ and 
	$$C_1=\max \{ C, 2 C_2 + 2\Lip(h)K \}.$$

	{\bf The constant $M$:}	
	Our constant $M$ is the same one that appears in Proposition~\ref{uniform},
        Corollary~\ref{growconst}, and Lemma~\ref{twothirds}.

	We then set, bearing in mind Lemma~\ref{ratio},

	{\bf The constants $D, D', D''$:}
$$
D = \frac{C_1}{a_Y} + 1 \qquad\qquad
D' = D \Lambda(X,Y) \Lambda(Y,X) \qquad\qquad
D'' =  2 D'.
$$

{\bf The constants $b, b_1, b_2\in\N$:} they are so that
$$
\left( \frac{8}{9} \right)^{b_1} C_1 \leq  \frac{a_X}{D''} \qquad\qquad
\lambda^{b_2}  \geq C_2 \qquad\qquad
b  =  b_1 M + b_2 
$$

Finally, {\bf The constants $L_1,L_0>0$} are defined by:
$$
  \frac{1}{D''}+ \frac{\lambda^{b}K}{L_1} =  \frac{1}{D'} \qquad\qquad
L_0  =  \max \{ L_1, 4C_2 D''   \}.
$$

We note that for any $L \geq L_1$ (in particular if $L\geq L_0$) we have,

$$
\frac{1}{D''}+ \frac{\lambda^{b}K}{L} \leq \frac{1}{D'}.
$$

We now argue as follows. If $[f^b_X(\beta)]$ contains a $f_X$-legal subpath of length at least
$C_2$, we are done. Otherwise, consider paths $\beta, [f_X^M(\beta)], \ldots, [f_X^{M
  b_1}(\beta)]$. If any of these paths contain an $f_X$-legal subpath of length at least
$C=cc(f_X)+1$, then by Lemma~\ref{grow}, $[f_X^{M b_1+b_2}(\beta)]=[f_X^{b}(\beta)]$  contains an
$f_X$-legal subpath of length at least $\lambda^{b_2}\geq C_2$ and we are done.

Hence we may assume that maximal legal subpaths in each of these paths have length less than
$C$. In particular $l_X([f_X^b(\beta)])$ is bounded by the number of its illegal turns times
$\min\{C,C_2\}$. Let $\eta_\beta$ be the number of illegal turns of $\beta$.

By Lemma~\ref{twothirds}, either the number of illegal turns in $[f_X^{b}(\beta)]$ is at most
$\left( \frac{8}{9} \right)^{b_1}\eta_\beta$, or the number of illegal turns in some $[f_X^{\kappa
  b_1}(\beta)]$, and hence also in $[f_X^{b}(\beta)]$, is at most $3$.

In either cases, our choice of constants guarantees that
$$
\frac{l_X([f_X^b(\beta)])}{l_X(\beta)} \leq \frac{1}{D''}.
$$
(Indeed, in latter case we easily get $\frac{l_X([f_X^b(\beta)])}{l_X(\beta)} \leq 4C_2/L_0
\leq\frac{1}{D''}.$ In former case we have $l_X(\beta)\geq a_X\eta_\beta$, whence
$\frac{l_X([f_X^b(\beta)])}{l_X(\beta)}\leq \left( \frac{8}{9} \right)^{b_1} C/a_X\leq
C/C_1D''\leq 1/D''$.)

Next we complete $\beta$ to a path $\beta \gamma$, as in the definition of our constant
$K$. That is,  $\beta \gamma$ is cyclically reduced, it is the fundamental domain of a hyperbolic
element $g\in G$, and $\gamma$ has length at most $K$. 
Note that
$$
\ell_X(\phi^b(g))\leq l_X([f_X^b(\beta \gamma)]) \leq l_X([f_X^b(\beta)]) + \lambda^bK,
$$
and therefore, (since $\ell_X(g)=l_X(\beta\gamma)$)
$$
\frac{\ell_X(\phi^b(g))}{\ell_X(g)} \leq\frac{l_X([f^b(\beta \gamma)])}{l_X(\beta \gamma)} \leq \frac{1}{D''} + \frac{\lambda^b K}{l_X(\beta \gamma)} \leq \frac{1}{D'}.
$$

Hence, by Lemma~\ref{ratio}, the axis of $g$ in $Y$ contains an $f_Y$-legal subpath of length
at least $C_1$. This means that $[h(\beta \gamma)]$ contains an $f_Y$-legal subpath of length
at least $C_1/2$ (the number $C_1/2$ arises from the fact that the legal subpath is really a
subpath of the axis, and not necessarily of $[h(\beta\gamma)]$.)
We then get  that $[h(\beta)]$ contains an $f_Y$-legal subpath of length at least $C_1/2 - \Lip(h)K \geq C_2$, because the length of $[h(\gamma)]$ is at most $\Lip(h)K$.

\end{proof}

\begin{lem}\label{lpnpell}
Let $f:X\to X$ be a simplicial train track map representing some $[\phi]\in\Out(\G)$. If $p$ is
any periodic line in $X$ which is a concatenation of pNp's, then any group element acting
periodically on $p$ is elliptic in $X_{+\infty}$. 
\end{lem}
\begin{proof}
  This follows directly from  the construction of $X_{+\infty}$.
\end{proof}

\Discretness

\begin{proof} 
	Without loss of generality (by Proposition~\ref{weakns} and Remark~\ref{Simplicial}), we may assume that each of $X, Y$ supports a simplicial
        train track map -- let's call these $f_X$ and $f_Y$ -- representing $[\phi]$
        and $[\phi^{-1}]$ respectively, so that we may apply Corollary~\ref{growconst}. Let
        $C_1$ be larger than the critical constants for $f_X, f_Y$ plus one, and then apply
        Corollary~\ref{growconst} to get constants $M_X, M_Y$ of which we take the larger, and
        call this $M_1$. By Lemma~\ref{grow}, for any path $\rho$ in $X$, containing a legal segment of length $C_1$, $[f_X^n(\rho)]$ contains a legal segment of length at least $(\lambda(\phi))^n$. Similarly for $Y$.

	We first show that $X_{+\infty}$ and $Y_{-\infty}$ have the same elliptic (and hence hyperbolic) elements. 	
	To this end, suppose that $g$ is $X_{+\infty}$-elliptic. This implies that
        $\ell_X(\phi^n(g))$ is bounded for all $n \geq 1$, by some constant $A$ (depending on
        $g$). (No legal segment in the axis for $\phi^n(g)$ can be longer than the critical
        constant, and the number of illegal turns is bounded.) 
        Hence $\ell_Y(\phi^n(g)) \leq A \Lambda(X,Y)$ is also bounded for all $n \geq 1$. In
        particular, if we realise $\phi^n(g)$ as a periodic line in $Y$, then the number of
        illegal turns in its period is uniformly bounded for all $n \geq 1$ (just because length of paths is
        bounded below by $a_Y$ times the number of turns). Let $A_1$ be greater than this
        number of illegal turns.
	
	Suppose $n$ is large and apply $A_1$ times $f_Y^{M_1}$ to the axis $L$ of $\phi^n(g)$
        in $Y$. By Corollary~\ref{growconst}, we 
        have either reduced the number of illegal turns by $A_1$, which is impossible, or we
        have a legal segment of length $C_1$ in $[f_Y^{A_1 M_1}(L)]$, or $L$ is a product of
        ppNp's in $Y$ and hence $Y_{-\infty}$ elliptic (Lemma~\ref{lpnpell}). We argue that
        only the last can occur, 
        since otherwise the length of $[f_Y^{A_1 M_1 j}(L)]$ in $Y$ would be at least
        $(\lambda(\phi^{-1}))^j$. But we know that $\ell_Y(g) \leq A \Lambda(X,Y)$, and we can take
        $j$ so that $(\lambda(\phi^{-1}))^j>A\Lambda(X,Y)$ and   $n=A_1M_1j$, getting contradiction. (Note that if $L$ is the axis for
        $\phi^n(g)$ in $Y$, then $[f_Y^n(L)]$ is the axis for $g$ in $Y$.)
	
	Therefore $l_{X_{+\infty}}(g) = 0$ implies that $l_{Y_{-\infty}}(g) = 0$, and vice versa, by symmetry.

	\medskip

	Next we set our notation and constants:
	
	\begin{enumerate}[(i)]
		\item $h: X \to Y$ is a straight $G$-equivariant map,
		\item $C_2 = \max \{ cc(f_X)+1, cc(f_Y)+ 1 + 2BCC(h) \}$,
		\item $L_0, b$ are the constants from Lemma~\ref{legalinverse} with the previous value of $C_2$,
		\item $J$ is any integer greater than $7L_0/a_X$,
		\item $M$ is the constant from Proposition~\ref{uniform},
		\item $\epsilon = \min \{ 1, 1/\lambda(\phi)^b, 1/\lambda(\phi)^{JM} \}$.
	\end{enumerate}
	
	We are going to show that for any  $X_{+\infty}$-hyperbolic $g$
	$$
          \max \{   \ell_{X_{+\infty}}(g) , \ell_{Y_{-\infty}}(g)        \} \geq \epsilon.
	$$

Consider an $X_{+\infty}$-hyperbolic $g$, which we represent as a periodic line $L$ in $X$.
First note that the statement here is really one about long group elements. More precisely, if
$\ell_X(g)$ is bounded above by $7L_0$, then we may apply Proposition~\ref{uniform} $J$
times. Since $J$ is greater than the number of edges --- whence that of illegal turns --- in a period of $L$ we cannot reduce the number of illegal turns $J$ times. We also cannot write $L$ (or $[f_X^{JM}(L)]$) as a product of ppNp's, as that would imply that $g$ was $X_{+\infty}$-elliptic. Hence, $[f_X^{JM}(L)]$ must contain an $f_X$-legal segment of length at least $cc(f_X)+1$. Then, by Lemma~\ref{grow}, $[f_X^{JM+j}(L)]$ must contain an $f_X$-legal segment of length at least $\lambda(\phi)^j$, and so $\ell_{X_{+\infty}}(g) \geq 1/\lambda(\phi)^{JM} \geq \epsilon$.

\medskip

Next,  since $g$ is $X_{+\infty}$-hyperbolic, there exists some $n>JM$ such that $[f_X^n(L)]$ contains a legal subpath longer than the critical constant for $f_X$. Hence, by increasing $n$, we can assume that $[f_X^n(L)]$ contains arbitrarily long legal segments. In particular we shall assume that $[f_X^n(L)]$ contains a legal subpath of length at least $2cc(f_X)+1$.

\medskip

	Let $\beta$ be a subpath of $L$ such that $[f_X^n(\beta)]$ is a legal subpath of
        $[f_X^n(L)]$. (Constructively, take two points $p,q$ in $[f_X^n(L)]$ that bounds a
        legal sub-path, and consider any two pre-images in $L$. These are the endpoints of $\beta$. Thus, the endpoints of $\beta$ lie in $f_X^{-n}([f_X^n(L)])$.) 
	
	If it were the case that $l_X(\beta) \geq L_0$, we could apply Lemma~\ref{legalinverse} to conclude that either $[f_X^b(\beta)]$ contains an $f_X$-legal subpath of length at least $C_2$, or $[h(\beta)]$ contains an $f_Y$-legal subpath of length at least $C_2$.
		In the former case, $[f_X^b(L)]$ contains $[f_X^b(\beta)]$ as a subpath, and hence an $f_X$-legal subpath of length at least $C_2 \geq cc(f_X) + 1$. Hence, by Lemma~\ref{grow} the length of $[f_X^{b+j}(L)]$ is at least $\lambda(\phi)^j$ and thus, $\ell_{X_{+\infty}}(g) \geq 1/\lambda(\phi)^b \geq \epsilon$.
		Similarly, in the latter case, $[h(L)]$ contains an $f_Y$-legal subpath of length at least $C_2 - 2BCC(h)$ (cancellation is possible on applying $h$). Since  $C_2 - 2BCC(h) \geq cc(f_Y) + 1$, we conclude as before that $\ell_{Y_{-\infty}}(g) \geq 1 \geq \epsilon$.
	Hence, we may assume that all such $\beta$ have length less than $L_0$.

\medskip

Finally, we conclude as follows. Choose four points, $p_0, p_1, p_2, p_3$ in $[f_X^n(L)]$, such
that $[p_0,p_3]$ splits as $[p_0,p_1][p_1,p_2][p_2,p_3]$, where $[p_1, p_2]$ is a maximal legal
subpath of $[f_X^n(L)]$ of length   at least $2cc(f_X) + 1$, and $[p_0, p_1]$ and
$[p_2, p_3]$  each consists of three maximal legal subpaths of $[f_X^n(L)]$. Then, for
        any $0\leq i<n$ we backward-recursively choose pre-images $f_X^{i-n}(p_s)$ of these
        points in each $[f_X^i(L)]$. Thus, for all $i$ the path
        $[f_X^{i-n}(p_0),f_X^{i-n}(p_{3})]$ is a sub-path $[f_X^i(L)]$ which splits as         $$[f^{i-n}_X(p_0),f^{i-n}_X(p_1)][f^{i-n}_X(p_1),f^{i-n}_X(p_2)][f^{i-n}_X(p_2),f^{i-n}_X(p_3)].$$

Moreover, for all $j$ we have $f_X^j(f_X^{-n}(p_s))=f_X^{j-n}(p_s)$.
        Let $\gamma_0=[f_X^{-n}(p_0),f_X^{-n}(p_1)]$,
        $\beta=[f_X^{-n}(p_1),f_X^{-n}(p_2)]$, $\gamma_1=[f_X^{-n}(p_2),f_X^{-n}(p_3)]$, and
        $\gamma=\gamma_0\beta\gamma_1$.  Since we are assuming that  
        pre-images of legal subpaths of $f_X^n(L)$ have length at most $L_0$ (as otherwise we
        are done), we have $l_X(\beta)<L_0$, $l_X(\gamma_0),l_X(\gamma_1)<3L_0$, and
        $l_X(\gamma)<7L_0$.

	As before, we apply $J$ times $f_X^{M}$ to $\gamma$, where $J$ (defined above) is a bound on
        the number of illegal turns in $\gamma$. We analyse the behaviour of this path using
        Proposition~\ref{uniform}. We know that if we get a long legal segment for $f_X$, this
        would bound the length $\ell_{X_{+\infty}}(g)$ from below (by $1/\lambda(\phi)^{JM}
        \geq \epsilon$). Also, we cannot reduce the number of illegal turns $J$
        times. Therefore, we are left with the case where for some $j\leq J$ we have that $[f^{jM}(\gamma)]$ splits as 
	$$
	[f^{jM}\gamma] = \delta_0 \rho_1 \ldots \rho_\kappa \delta_1,
	$$
        where each subpath has at most one illegal turn and the $\rho_i$'s are ppNps.
        Moreover,  $[f^{jM}_X(\gamma)]$ has sub-paths $[f_X^{jM}(\gamma_0)]$,
        $[f_X^{jM}(\beta)]$ and  $[f_X^{jM}(\gamma_1)]$. We also know that the first and last
        of these cross at least two illegal turns, because $[p_0,p_1]$, and $[p_2,p_3]$ crosses
        two illegal turns. In particular, $[f_X^{jM}(\beta)]$ is a sub-path of $\rho_1 \ldots
        \rho_\kappa$. However, this is a splitting and therefore $[f_X^{n}(\beta)]$ is a
        sub-path of $[f_X^{n-jM} (\rho_1 \ldots \rho_\kappa)]$, which is impossible since
        $[f_X^n(\beta)]$ is a legal path of length at least $2cc(f_X)+1$, and no ppNp can
        contain a legal subpath of length greater than $cc(f_X)$. 
\end{proof}

\section{Co-compactness of the Min-Set}

\label{s7}

We will prove our main theorem, first under the extra hypothesis that our automorphism is primitive. More specifically, we will prove that the Min-Set of a primitive irreducible automorphism is co-bounded under the action of $\langle\phi\rangle$, which implies that it is co-compact (see Section~\ref{Conditions}). The general result is proved in the next section, where we drop the primitivity hypothesis.

Let us fix a group $G$ and a free factor system  $\G=(\{G_1,\dots,G_k\},r)$ of $G$.
We are going to use North-South dynamic stated in Section~\ref{s5}, where
we had the additional assumption of $\rank(\G)=k+r\geq 3$. We remark that in case of lower
rank, either we are in the classical $CV_2$ case, and the co-compactness result is known (a
proof can be found  for example in~\cite{FMS2}) or the result is trivial.

\subsection{Ultralimits}

At this stage, our strategy is as follows: we will argue by contradiction so that if
$\Min_1(\phi)$ is not co-compact, then Theorem~\ref{cocompactnessdefs} provides us with a
sequence of minimally displaced points $Z_i$ which stay at constant Lipschitz distance from the
attracting tree $X_{+\infty}$, but which are at unbounded distance from some basepoint. We can
find scaling constants $\mu_i$ so that $Z_i/\mu_i$ is bounded, and we would like to take the
limit of a sub-sequence of $Z_i/\mu_i$. The problem with this is that a priori we do not have
sequential compactness unless $G$ is countable, Remark~\ref{seqcpt}.

This is a minor issue, as what we really use is the existence of some adherence point of the
above sequence. The easiest way to deal with this is to turn to $\omega$-limits (or ultralimits).
Of course, the main point of interest is exactly when $G$ is countable, and in this case
$\omega$-limits are not needed.

So the reader may wish to simply read all the
$\omega$-limits as usual limits, and the arguments stay essentially the same (up to taking
subsequences appropriately, and consider suitable $\liminf$ or $\limsup$ in some
inequalities). We refer to Section~\ref{sec2} definitions and some basic properties for
$\omega$-limits, especially Definitions~\ref{omega}, \ref{omegalimits} and Proposition~\ref{propertiesofultralimits}. 
In our situation, we make the following definition of $\omega$-limit of elements in $\overline{\O(\G)}$. 

\begin{defn} Let $\omega$ be a non-principal ultrafilter on $\N$.  
	Let $(Y_i) \subset \overline{\O(\G)}$ be a sequence and let $Y_\infty\in\overline{\O(\G)}$. We say
        that $Y_\infty$ is the $\omega$-limit of $Y_i$, and write $Y_\infty=\lim_\omega Y_i$ if
        for any $g\in G$ we have  
	$$
	\ell_{Y_{\infty}}(g)= \lim_{\omega} \ell_{Y_i} (g).
	$$ 
\end{defn}

\begin{rem}
  Suppose that for any $g\in G$, the $\omega$-limit of $\ell_{Y_i}(g)$ exists. Then the corresponding
  $\omega$-limit length function is indeed the length function of an element in
  $Y_\infty\in\overline{\O(\G)}$. This is because the conditions defining length functions of
  trees are closed under $\omega$-limits (by Proposition~\ref{propertiesofultralimits} and~\cite{Parry}). 
\end{rem}

\begin{prop} \label{nontrivial}
  Let $\omega$ be a non-principal ultrafilter on $\N$. Consider a sequence of points $Z_i \in
  \O(\G)$, and let $X \in \O(\G)$. Set $\mu_i = \Lambda(X,Z_i)$. 		
	Then $\lim_{\omega} \frac{Z_i}{\mu_i}$ exists, is unique (depends on $\omega$) and non-trivial. 	
\end{prop}
\begin{proof}
	Notice that for any $g \in G$, 
	$
	\ell_{Z_i}(g) \leq \mu_i \ell_X(g) 
	$. 
	So each sequence $\frac{\ell_{Z_i}(g)}{\mu_i}$ is a bounded sequence and therefore
        has a unique $\omega$-limit $l_T(g)$ for some $T$.  	
	Furthermore, Lemma~\ref{CandFinite} provides a finite set $H \subseteq G$ such that, for any $i$, 
	$$
	\Lambda(X,Z_i) = \max_{h \in H} \frac{\ell_{Z_i}(h)}{\ell_X(h)}.
	$$
	Hence, 
	$$
	1 = \lim_{\omega }\Lambda(X,Z_i/\mu_i) = \lim_{\omega} \max_{h \in H} \frac{\ell_{Z_i}(h)/\mu_i}{\ell_X(h)} 
	=  \max_{h \in H} \frac{\lim_{\omega}\ell_{Z_i}(h)/\mu_i}{\ell_X(h)} = \max_{h \in H} \frac{\ell_{T}(h)}{\ell_X(h)}, 
	$$
	which shows that the limiting tree is non-trivial. Here we have used that the
        $\omega$-limit commutes with a finite maximum (Proposition~\ref{propertiesofultralimits}) 
\end{proof}

\subsection{Co-compactness of the Min-Set of primitive irreducible automorphisms}\label{s7.1}
For this, and the next section, we fix once and for all a non-principal ultrafilter $\omega$ on
$\N$.  In the subsequent results we refer to $\Lambda(T,-)$, where $T$ may be a tree in
$\overline{\O(\G)}$, rather than just in $\O(\G)$. We intend the following: 

\begin{defn}
	\label{extendlambda}
	Let $T, W \in \overline{ \O(\G)}$. We set $\Lambda(T,W)$ to be the  supremum of the ratios, $\frac{\ell_W(g)}{\ell_T(g)}$, over all elements which are hyperbolic in $T$. This is possibly infinite. We also set $\Lambda(T,W)= +\infty$ if there is a $T$-elliptic group element which is hyperbolic in $W$.  
	
\end{defn}

\begin{lem}
  \label{displace}
  Let $[\phi]\in\Out(\G)$ be $\G$-irreducible and with $\lambda(\phi)=\lambda>1$. Let
  $X\in\Min(\phi)$ and let $X_{+\infty}$ be the corresponding attracting tree.
  Suppose we have a sequence $(Z_i)\subset\Min_1(\phi)$ for which there is $\mu_i$ such
        that $\lim_{\omega}  \frac{Z_i}{\mu_i}$ exists and is a non-trivial tree $T$. Then, for
        any positive integer 
        $n$ we have  
	\begin{itemize}
		\item $\Lambda(T, T\phi^n) \leq \lambda^n$, and
		\item if $\Lambda(T, X_{+\infty})$ is finite, then $\Lambda(T, T\phi^n) = \lambda^n$.
	\end{itemize} 
\end{lem}
\begin{proof} 
	The first claim is straightforward, since for any $T$-hyperbolic $g$
	$$
	\frac{\ell_{T\phi^n}(g)}{\ell_T(g)} = \lim_{\omega} \frac{\ell_{Z_i\phi^n}(g)/\mu_i}{\ell_{Z_i}(g)/\mu_i} = \lim_{\omega} \frac{\ell_{Z_i\phi^n}(g)}{\ell_{Z_i}(g)} \leq \lambda^n.
	$$
	For the second claim note that, for any $T$-hyperbolic $g$ and any positive integer $m$,

           \begin{eqnarray*}
    	\Lambda(T, X_{+\infty}) &=& \Lambda( T\phi^{mn}, X_{+\infty}\phi^{mn}) \geq
    	\frac{\ell_{X_{+\infty}}(\phi^{mn}(g))}{\ell_T(\phi^{mn}(g))}
    	=\frac{\ell_{X_{+\infty}}(g)\lambda^{mn}}{\ell_T(\phi^{mn}(g))} 
    	\\ \\
    	&\geq&
    	\frac{\ell_{X_{+\infty}}(g)\lambda^{mn}}{\ell_T(g)(\Lambda(T,T\phi^n))^m}
    	=\frac{\ell_{X_{+\infty}}(g)}{\ell_T(g)}
    	\left( \frac{\lambda^n}{\Lambda(T, T\phi^n)}\right)^m.
    \end{eqnarray*}

We note that at no stage are we dividing by zero here. Indeed, if $g$ is $T$-hyperbolic, then it is also $\G$-hyperbolic and:
$$
\frac{l_T(\phi (g))}{l_T(g)} = \lim_{\omega} \frac{l_{Z_i}(\phi (g))}{l_{Z_i}(g)} \geq \lim_{\omega} \frac{1}{\Lambda(Z_i \phi, Z_i)} \geq \lim_{\omega} \frac{1}{\Lambda(Z_i, Z_i \phi)^C} = \frac{1}{(\lambda(\phi))^C}, 
$$
where the last inequality follows from quasi-symmetry, Lemma~\ref{quasi-symmetry} and the fact
that $\Min_1(\phi)$ is uniformly thick (Theorem~\ref{PropertiesOfIrreducibles}). 
In particular, this says that if we start with a $T$-hyperbolic group element $g$, then
$\phi(g)$ is also $T$-hyperbolic (hence $\phi^{mn}(g)$ is also $T$-hyperbolic).  
\end{proof}

\begin{lem}
	\label{estimate}
  Let $[\phi]\in\Out(\G)$ be $\G$-irreducible and with $\lambda(\phi)=\lambda>1$. Let
  $X\in\Min(\phi)$ and let $X_{+\infty}$ be the corresponding attracting tree.
  Let $Y\in\Min(\phi)$ such that it admits a simplicial train track representative for $[\phi]$,
  and let $f_Y : Y \to X_{+\infty}$ be an optimal map. 
	For any $g\in\Hyp(\G)$, and any integer $n$, we have	
	$$
	\ell_{X_{+\infty}}(\phi^n(g)) \geq \Lambda(Y,X_{+\infty}) \ell_{Y} (\phi^n(g)) - \ell_{Y}(g) \frac{B}{a_Y},
	$$
	where $B=\Lambda(Y,X_{+\infty})\vol(Y)$ is the BCC of $f_{Y}$ and $a_Y$ is the length
        of the shortest edge in $Y$. 
\end{lem}
\begin{proof} For $g\in\Hyp(\G)$, let $\eta_g$ be the number of edges in a reduced loop
  representing $g$ in
  $G\backslash Y$. Clearly $\ell_Y(g)\geq \eta_g a_Y$. By
  Proposition~\ref{stablemap}, the axis of $\phi^n(g)$ in $Y$ can  be written as a
  concatenation of at  most $\eta_g$ $f_Y$-legal paths. By Corollary~\ref{BCC},   
	$$
	\ell_{X_{+\infty}}(\phi^n\gamma) \geq \Lambda(Y,X_{+\infty}) \ell_{Y} (\phi^n \gamma) -
        \eta_g B\geq \Lambda(Y,X_{+\infty}) \ell_{Y} (\phi^n(g)) - \ell_{Y}(g) \frac{B}{a_Y}.$$
\end{proof}

\begin{lem}\label{laW}
  Let $[\phi]\in\Out(\G)$ be $\G$-irreducible, with $\lambda(\phi)>1$. Then there is a constant
  $\epsilon>0$ such that for all $W\in\Min_1(\phi)$, admitting a simplicial train track
  representative of $[\phi]$, we have $$a_{W}>\epsilon$$ where $a_W$
  denotes the length of the shortest edge in $W$. 
\end{lem}
\begin{proof}
   Note that uniform thickness of minimally displaced points is not enough in order to have a
   lower bound on the lengths of the edges, but for points supporting simplicial train track
   representatives, there is such a bound, as there are finitely many transition matrices
   of simplicial train tracks representing $[\phi]$\footnote{Such transition matrices have
     integer non-negative entries, and spectral radius $\lambda(\phi)$. Hence each entry of the matrix is bounded by $\lambda_{\phi}$. So for any given spectral radius
     $\lambda$ there are only finitely many matrices with non-negative integer coefficients
     and spectral radius not exceeding $\lambda$.},  and the lengths for edges are given by
   eigenvectors of the Perron-Frobenius eigenvalue of these matrices.
\end{proof}

\begin{prop}
	\label{Tfar}
  Let $[\phi]\in\Out(\G)$ be $\G$-irreducible and with $\lambda(\phi)=\lambda>1$. Let
  $X\in\Min(\phi)$ and let $X_{+\infty}$ be the corresponding attracting tree. Let
  $T\in\overline{\O(\G)}$ be a non-trivial tree which is the $\omega$-limit of a sequence $Z_i/\mu_i$ with the
  following properties:

  \begin{enumerate}
  \item The $Z_i$ are uniformly thick and with co-volume $1$; that is, $\exists\epsilon_0>0 
    \forall i, Z_i \in \O_1(\epsilon_0)$;
  \item $\mu_i\to \infty$;
  \item there is $\delta>0$ such that $\Lambda(Z_i,X_{+\infty}) \geq \delta $, for all $i$;
  \item there is a sequence  $W_i \in \Min_1(\phi)$ and $K>0$ such that $\Lambda(Z_i,W_i)
    \leq K$, for all $i$. (For example if the $Z_i$ themselves belong to $\Min_1(\phi)$). 
   \end{enumerate}
	Then $\Lambda(T, X_{+\infty}) = \infty$.
\end{prop}
\begin{proof} Note that by assumption  $T$ is non-trivial.   Without loss of generality, we may assume that each $W_i$ supports a simplicial train track
  representing $[\phi]$ (Remark~\ref{Simplicial}).

 Note also that the points $W_i$ are uniformly thick because they are minimally displaced
 (Theorem~\ref{PropertiesOfIrreducibles})  and the same is true for $Z_i$, by
 assumption. Therefore, all the points $W_i,Z_i$ belong to some uniform thick part and since
 the stretching factor $\Lambda$ is multiplicatively quasi-symmetric when restricted on any
 thick part  $\O_1(\epsilon)$ (Theorem~\ref{quasi-symmetry}) it follows that there is some
 uniform constant $C$ such that
	$$ \Lambda(Z_i,W_ni)^{1/C} \leq  \Lambda(W_i,Z_i) \leq \Lambda(Z_i,W_i)^C.$$
In particular, it follows that for the constant $K_1 = K^C$, we get that for any $i$:
	$$\Lambda(W_i,Z_i) \leq \Lambda(Z_i,W_i)^C \leq K^C = K_1.$$

	We will prove now that there is a non-trivial tree $S\in\overline{\O(\G)}$ so that
        $W_i/\mu_i$ $\omega$-converges to $S$. We first
        observe that for every hyperbolic element $g \in G$ and positive integer $i$,
        we have that: 
	$$
	0 \leq  \ell_{\frac{W_i}{\mu_i}} (g) =  \frac{\ell_{W_i}(g)}{\mu_i} \leq K\frac{\ell_{Z_i}(g)}{\mu_i}.
	$$
        It follows that the sequence $W_i/\mu_i$ is bounded, and so $S = \lim_{\omega}
        W_i/\mu_i$ exists. Moreover, $T,S$ have finite distances to each other (in particular,
        $S$ is non-trivial 
        since $T$ is non-trivial, and they admit the same hyperbolic elements) because of the
        inequalities: 
	
	$$
	\frac{\ell_{S}(g)}{\ell_T(g)} = \lim_{\omega} \frac{\ell_{W_i}(g)/\mu_i}{\ell_{Z_i}(g)/\mu_i} \leq K
	\qquad
	\text{ and similarly }
	\qquad
	\frac{\ell_{T}(g)}{\ell_S(g)} = \lim_{\omega} \frac{\ell_{Z_i}(g)/\mu_i}{\ell_{W_i}(g)/\mu_i} \leq K_1.
	$$

	Therefore, it is enough to prove that $\Lambda(S,X_{+\infty})$ is infinite. We argue by
        contradiction, assuming that $$\Lambda(S,X_{+\infty}) < \infty.$$ Then, by Lemma~\ref{displace}, $\Lambda(S, S\phi^m) = \lambda^m$ for any positive integer $m$. 
For all such $m$, we then choose a $S$-hyperbolic element  $g_m$ such that $\ell_S(\phi^m (g_m)) / \ell_S(g_m)\geq \lambda^m/2$.
Now we apply Lemma~\ref{estimate} to $W_i$ and get constants
$B_i=\Lambda(W_i,X_{+\infty})$ (because $\vol(W_i)=1$), and
$\epsilon_i:=a_{W_i}$. By Lemma~\ref{laW}, there is a uniform $\epsilon >0$ so that
$\epsilon_i>\epsilon$ for all $i$. On the other hand, by the properties of $Z_i$ and the triangle
inequality, we get that all the distances $\Lambda(W_i,X_{+\infty})$ are uniformly bounded from
below by $\frac{\delta}{K}$. Thus
	$$
	\ell_{X_{+\infty}}(\phi^m(g_m)) \geq \Lambda(W_i, X_{+\infty})(\ell_{W_i} (\phi^m(g_m)) - \ell_{W_i}(g_m) \frac{1}{\epsilon}) \geq \frac{\delta}{K}(\ell_{W_i}(\phi^m(g_m)) - \ell_{W_i}(g_m) \frac{1}{\epsilon}) .
	$$

        Therefore,	
        \begin{eqnarray*}
          \frac{\ell_{X_{+\infty}}(\phi^m(g_m))}{\ell_{W_i}(\phi^m(g_m))} & \geq & \frac{\delta}{K}  \left(1 - \frac{ \ell_{W_i}(g_m) }{\epsilon \ell_{W_i}(\phi^m(g_m))}\right)  \\ \\ &  \xrightarrow[\lim_{\omega}]{} & \frac{\delta}{K} \left( 1 -  \frac{ \ell_{S}(g_m) }{ \epsilon \ell_{S}(\phi^m(g_m))}\right) \geq \frac{\delta}{K} \left(1 - \frac{2}{\epsilon \lambda^m}\right).
	\end{eqnarray*}

	For any $0<\delta_0 <1$, choose a $m$ such that $1 - \frac{2}{ \epsilon \lambda^m} \geq
        1- \delta_0$. For this choice of $m$, let $c_i = \frac{\ell_{X_{+\infty}}(\phi^m
          (g_m))}{\ell_{W_i}(\phi^m(g_m))}$. Then the calculation above shows that $\lim_{\omega} c_i \geq \frac{\delta}{K} (1 - \delta_0) > 0$. On the other hand, $\Lambda(S, X_{+\infty}) \geq \lim_{\omega} c_i \mu_i$. Hence, $$\Lambda(S, X_{+\infty}) \geq  \lim_{\omega}  \frac{\mu_i \delta (1-\delta_0)}{K} = \infty,$$
	
contradicting the assumption $\Lambda(S,X_{+\infty})<+\infty$.	
\end{proof}

Remember that if $T,S\in\Min_1(\phi)$ we have (Proposition~\ref{weakns} and basic properties of
stable trees):
$$T_{+\infty} = \Lambda(S, T_{+\infty}) S_{+\infty}\qquad
\Lambda(S_{+\infty},T_{+\infty}) = \Lambda(S, T_{+\infty})\qquad T_{\infty}\phi^{\pm
  n}=\lambda(\phi)^{\pm n}T_{\infty}.$$

\begin{prop}
	\label{unstabledist}
  Let $[\phi]\in\Out(\G)$ be $\G$-irreducible and with $\lambda(\phi)=\lambda>1$. Let
  $X\in\Min_1(\phi),Y\in\Min_1(\phi^{-1})$, and let $X_{+\infty}, Y_{-\infty}$ be the
  corresponding attracting and repelling trees.

Then,  for every constant $\nu > 0$, there is a constant $\delta=\delta(\nu,\lambda,X,Y)>0$ so that for every $Z \in \Min_1(\phi)$
  we have
  $$\Lambda(Z,X_{+\infty}) \leq \nu \quad \Longrightarrow\quad\Lambda(Z,Y_{-\infty}) \geq \delta.$$
\end{prop}

\begin{proof} Let $\epsilon>0$ be the constant given by Theorem~\ref{DiscretnessOfMinSet}.
  For any positive number  $\nu$, fix a (for instance the smallest) positive integer $n_0$ for which	
	$$
	\lambda ^{-n_0} \leq \frac{\epsilon}{4 \nu}.
	$$

 Let now $Z \in \Min_1(\phi)$ be such that $\Lambda(Z,X_{+\infty})\leq\nu$, and let
 $Z_{+\infty}$ be the corresponding attracting tree. Let $g$ be a candidate that realises
  $\Lambda(Z,X_{+\infty})$. In particular $g$ and any of its power are not $X_{+\infty}$-elliptic.
  Moreover, the length of $g$ with respect to both $Z$ and $Z_\infty$ is bounded above by $2$
  (as the volume of $Z$ is $1$). If we set now $h = \phi^{-n_0}(g)$, we get 
  \begin{eqnarray*}
    \ell_{X_{+\infty}}(h)
    &=& \ell_{\Lambda(Z,X_{+\infty}) Z_{+\infty}}(h)  =
        \Lambda(Z,X_{+\infty})\ell_{Z_{+\infty}}(\phi^{-n_0}(g))\\
    &\leq&  \nu \ell_{Z_{+\infty}}(\phi^{-n_0}(g))  =  \nu \ell_{(Z_{+\infty}\phi^{-n_0})}(g)
           = \nu \ell_{(\lambda^{-n_0}Z_{+\infty})}(g)\\
    &=& \nu \lambda^{-n_0}\ell_{Z_{+\infty}}(g)
            \leq 2 \nu \lambda^{-n_0} \leq \frac{\epsilon }{2} < \epsilon.
	\end{eqnarray*}
Therefore, by Theorem~\ref{DiscretnessOfMinSet}, it follows that $\ell_{Y_{-\infty}}(h)\geq \epsilon$. 

By multiplicative quasi-symmetry of $\Lambda$ restricted on the thick parts of $\O_1(\G)$
(Theorem~\ref{quasi-symmetry}), there exists a constant $C$  such that $\Lambda(T,S) \leq
\Lambda(S,T)^C$, for any $T,S \in \Min_1(\phi) \cup \Min_1(\phi^{-1})$ (note that $C$ depends
only on $[\phi]$ because elements in Min-Sets are uniformly thick because $[\phi]$ is
irreducible). In particular,  
$$\Lambda(Z, Z\phi^{-n_0}) = \Lambda(Z\phi^{n_0}, Z) \leq \Lambda(Z, Z\phi^{n_0})^C =
\lambda^{Cn_0}.$$ 
Therefore, as the length of $h$ with respect to $Z$ is  at most $2$, we get that:	
	$$
	\ell_Z(h) = \ell_Z (\phi^{-n_0} (g)) \leq \ell_Z(g) \Lambda(Z, Z\phi^{-n_0}) \leq 2 \lambda^{Cn_0}.
	$$
which implies 
\begin{eqnarray*}
\Lambda(Z, Y_{-\infty})\geq
	\frac{\ell_{Y_{-\infty}}(h)}{\ell_Z(h)} \geq \frac{\epsilon}{2\lambda^{Cn_0}} = \delta
	\end{eqnarray*}
where the quantity $\delta$ does not depend on $Z$.	
\end{proof}

\begin{cor}\label{NotLimitTree}
  Let $[\phi]\in\Out(\G)$ be $\G$-irreducible and with $\lambda(\phi)=\lambda>1$.
Let  $X\in\Min_1(\phi),Y\in\Min_1(\phi^{-1})$, and let $X_{+\infty}, Y_{-\infty}$ be the
  corresponding attracting and repelling trees.

 Let $T\in\overline{\O(\G)}$ which is the $\omega$-limit of a sequence $Z_i/\mu_i$ with the following
 properties:
 \begin{enumerate}
 \item $Z_i\in\Min_1(\phi)$;
 \item $\mu_i\to\infty$;
 \item there is $\nu>1$ such that $1 \leq \Lambda(Z_i, X_{+\infty}) \leq \nu$. 
 \end{enumerate}
	Then $$\Lambda(T, X_{ +\infty}) = \infty = \Lambda(T,Y_{-\infty}).$$
\end{cor}
\begin{proof}
The equality for $X_{+\infty}$, follows by just applying Proposition~\ref{Tfar} directly on the
sequence $Z_i=W_i$, with $\delta=1$.

For $Y_{-\infty}$, we first apply Theorem~\ref{UnifDistFromMinSet}, which provides us a
sequence of minimally displaced points $W_i\in \Min_1(\phi^{-1})$, with the property that for
some uniform  constant $M$, $$\max\{\Lambda(W_i,Z_i),\Lambda(Z_i,W_i)\} \leq M.$$ 

Next, we want to apply Proposition~\ref{Tfar}, for $\phi^{-1}$. Conditions
$(1)$ and $(2)$ are satisfied by our assumptions. By the choice of $W_i$'s, Condition $(4)$ is
satisfied, too.

For property $(3)$, we apply Proposition~\ref{unstabledist} for every $i$, to the point
$Z_i$. By hypothesis $\Lambda(Z_i,X_{+\infty})<\nu$, and Proposition~\ref{unstabledist}
provides the $\delta>0$ such that $\Lambda(Z_i,Y_{-\infty})>\delta$, as required.
\end{proof}

\begin{thm}\label{CocompactnessPrimitive}
	Let $[\phi]\in\Out(\G)$ be $\G$-primitive irreducible (that is, a relatively
        irreducible automorphism with a train track with primitive transition matrix -- and hence exponential
        growth). Then $\Min_1(\phi) = 
        \Min(\phi) \cap \O_1$ is co-compact, under the action of $\langle\phi\rangle$. 
\end{thm}
\begin{proof}
  Let $X\in\Min_1(\phi), Y\in\Min_1(\phi^{-1})$, and $X_{+\infty},Y_{-\infty}$ be the corresponding
  attracting/repelling trees. By Theorem~\ref{North-South1}, $[X_{+\infty}]=[T_\phi^+]$ is the unique
  attracting class of trees for $\phi$, and $[Y_{-\infty}]=[T_\phi^-]$ is the unique repelling
  class. 
  
  We now argue by contradiction and suppose that $\Min_1(\phi)/\langle\phi\rangle$ is not compact. By
  Theorem~\ref{cocompactnessdefs} (point $(v)$), there is a sequence of points $Z_1,Z_2,\ldots,Z_i,\ldots$ of
  $\Min_1(\phi)$ for which $\Lambda(X,Z_i\phi^m) \geq i$, for every $m\in\Z$.
  Note that for all  $m\in \Z$ and  $W\in\O(\G)$ we have  $$\Lambda(W\phi^m,X_{+\infty}) = \Lambda(W,  X_{+\infty}\phi^{-m}) =
  \Lambda(W,\lambda(\phi)^{-m}X_{+\infty}) = \lambda(\phi)^{-m} \Lambda(W,X_{+\infty}).$$
 Therefore, we can replace points $Z_i$ with some points of their $\langle\phi\rangle$-orbits (which we will still denote by $Z_i$) with the extra property $$1 \leq \Lambda(Z_i, X_{+\infty}) \leq \lambda(\phi).$$
  We set $\mu_i = \Lambda(X,Z_i)$.	
	Now $\lim_{\omega} Z_i/\mu_i=T$ for some
        (non-trivial) tree $T$ on the boundary of $\O(\G)$, by Proposition~\ref{nontrivial}. From
        Corollary~\ref{NotLimitTree}, we know
        \begin{eqnarray}
          \label{eq:f}
        \Lambda(T, Y_{-\infty}) = \infty=\Lambda(T,X_{+\infty}).          
        \end{eqnarray}
        In particular, $T$ does not belong to $[T_\phi^-]$. On the other hand, by
        Lemma~\ref{displace}, it follows that for every positive integer 
        $j$ $$\Lambda(T,T\phi^j) \leq \lambda(\phi)^j,$$ or equivalently 
$$\Lambda(T,\frac{T\phi^j}{ \lambda(\phi)^{j}}) \leq 1.$$ 
	By applying the North-South dynamics Theorem~\ref{North-South1}, on $T
        \notin [T_\phi^-]$, we get that $\frac{T\phi^j}{ \lambda(\phi)^{j}}$ projectively
        converges to $T_\phi^+$, which is in the same projective class as $X_{+\infty}$, so
        there is $c>0$ such that $\frac{T\phi^j}{\lambda(\phi)^{j}}$ converges to $cX_{+\infty}$.
        But in that case, $\Lambda(T,X_{+\infty})$ would be finite, contradicting~\eqref{eq:f}.
\end{proof}

\begin{rem}
  The primitivity assumption is used only in applying North-South dynamics in last theorem, and
  not in previous results of this section. In Proposition~\ref{unstabledist} and
  Corollary~\ref{NotLimitTree}, if one is allowed to  use  North-South dynamics (for instance for primitive automorphisms), then one can  replace any instance of  $Y_{-\infty}$ with $X_{-\infty}$.  
\end{rem}

\subsection{Co-compactness of the Min-Set of general irreducible automorphisms}
\label{generalautos}

In this subsection, we will prove the co-compactness of the Min-Set for irreducible
automorphisms of exponential growth. 
For this section we fix: A free factor system $\G=(\{G_1,\dots,G_k\},r)$ of a group $G$; an
element $[\phi]\in\Out(\G)$ which is $\G$-irreducible, with $\lambda(\phi)=\lambda>1$; an
element $X\in\Min_1(\phi)$ supporting a simplicial train track map $f:X\to X$ representing
$[\phi]$. In particular there is genuine automorphism $\psi\in[\phi]$ represented by $f$ 
(Definition~\ref{defin272}) and, up possibly to replace $\phi$ with $\psi$, we may assume that $f$
represents $\phi$ (that is, $f(gx)=\phi(g)f(x)$).

We denote by $M_f$, the transition matrix of $f$ (see Section~\ref{TransitionMatrix}). If $M_f$
fails to be primitive, then we can partition the edge orbits into blocks so that, for some
positive integer $s$, $M_{f^s}=M_f^s$ is a block diagonal matrix, which is strictly positive
matrix when restricted to a block. 
Correspondingly, we can define sub-forests, $X_1, \ldots, X_l$, of $X$ consisting of edges, and their incident vertices, belonging to a single block. The following two lemmas are straightforward:

\begin{lem}Let $f$, $X$ and $X_i$ be defined as above. Then
	\begin{enumerate}[(i)]
		\item $f$ permutes the $X_i$'s.
		\item Each $X_i$ is a $G$-forest (i.e. a forest which is $G$-invariant).
		\item The union of the $X_i$ is $X$.
	\end{enumerate}
\end{lem}

We can then define cylinders.

\begin{defn}
	A cylinder is a connected component of some $X_i$.
\end{defn}
      
\begin{rem}
We note that it is possible for two cylinders to intersect at a vertex, as long as the
cylinders belong to different sub-forests $X_i\neq X_j$. 
\end{rem}

\begin{lem}
	\label{actioncylinders}
	If $C$ is a cylinder, then $f(C)$ is also a cylinder. Moreover, for any $g\in G$, also
        $g(C)$ is a cylinder, belonging to the same $X_i$ as $C$.  
\end{lem}

We also have the following:

\begin{lem}
	\label{cylinderstabs}
	For any cylinder $C$, and any vertex $v \in C$:
	\begin{enumerate}[(i)]
		\item $\stab_G(C)$ contains a $\mathcal{G}$-hyperbolic element;
		\item $\stab_G(v) \leq Stab_G(C)$.
	\end{enumerate}
\end{lem}
\begin{proof} Without loss of generality, the map $f^s$ has a block diagonal transition matrix
  where the positive entry in every block is at least $3$.
  Choose an edge $e$ in $C$. Since $f^s$ is train track, $f^s(e)$ is a legal path. The
  condition on $f^s$ means that $f^s(e)$ crosses the 
  orbit of $e$ at least $3$ times, and it is contained in the same $X_i$ as $C$. This means that $C$
  contains a path crossing $e$, $ge$ and $he$ for some $g, h \in G$; where these $3$ edges are
  distinct. Clearly, $g,h \in Stab_G(C)$.
  Since the action of $G$ on $X$ is edge-free, this implies that if both $g$ and $h$ are
  elliptic, then $gh$ is hyperbolic (as $e,ge,he$ are all in the legal path $f^s(e)$). Hence
  $\stab_G(C)$ contains a hyperbolic element.
  Finally, an element of $\stab_G(v)$ must send $C$ to another cylinder containing $v$ but  belonging to the same subforest as $C$. This means that it must preserve $C$.
\end{proof}

We now define a new tree $\mathcal{T}$ from this information, which remembers the construction
of the dual tree of the partition of $X$ in cylinders. We note
that this is not a $\G$-tree because Lemma~\ref{cylinderstabs} tells us that vertex stabilisers are too big and in general edge stabilisers are not trivial. More precisely:

\begin{defn}
	We define a $G$-tree $\mathcal{T}$ as follows. This is a bi-partite tree:
	\begin{itemize}
		\item Type I vertices are the cylinders of $X$.
		\item Type II vertices are the vertices of $X$ which belong to at least two
                  distinct cylinders. 
	\end{itemize}
The edges of $\mathcal{T}$ are the pairs $(C,v)$ where $C$ is a Type I vertex, and $v$ is a
Type II vertex contained in  $C$.
\end{defn}

It is an easy exercise to see that $\mathcal{T}$ is a $G$-tree.

\begin{prop}
	\label{edgestabs}
	We get the following:	
	\begin{enumerate}[(i)]
		\item The stabiliser of an edge $(C,v)$ of $\mathcal{T}$ is equal to $Stab_G(v)$;
		\item $f$ induces a map, $F:\mathcal{T} \to \mathcal{T}$ representing $\phi$
                  (that is $F(gx)=\phi(g)F(x)$), which sends vertices to vertices -- preserving
                  type -- and edges to edges;
		\item The irreducibility of $[\phi]$ implies that all the edge stabilisers of $\mathcal{T}$ are non-trivial.
	\end{enumerate}
\end{prop}
\begin{proof}
  The first point follows from the second part of Lemma~\ref{cylinderstabs}. The second point
  follows from Lemma~\ref{actioncylinders}, and the fact that $f$ maps vertices to vertices.

  For the final point, note that if $\mathcal{T}$ had an edge with trivial stabiliser, we could
  collapse all the edges with non-trivial stabiliser, and get a new $G$-tree
  $\mathcal{\overline{T}}$ and a new map $\overline{F}$ on this tree representing $\phi$. Since
  the action of this tree is edge-free and non-trivial, this would correspond to a proper free
  factor system for $G$, which would be $\phi$-invariant. However, Lemma~\ref{cylinderstabs}
  implies that this free factor system properly contains $[\mathcal{G}]$. Therefore, we would
  obtain a $\phi$-invariant proper free factor system properly containing $[\mathcal{G}]$, a
  contradiction to the irreducibility of $[\phi]$. 
\end{proof}

\Cocompactness
\begin{proof}
  We shall deduce this theorem from the primitive case. We have our base-point $X \in
  \Min(\phi)$ which supports our simplicial train track map $f$, representing $\phi$, but with
  (potentially) imprimitive transition matrix. Let $X^{+\infty}$ be the attracting tree corresponding to $X$. We use here the
notation $X^{+\infty}$ instead of $X_{+\infty}$, as we did in the rest of the paper, for
notational reasons of this proof.

We argue by contradiction, and  suppose that the action is not co-compact. Then, by
Theorem~\ref{cocompactnessdefs} (point $(viii)$), we may find a sequence of points $Y_i\in\Min(\phi)$ such that:
\begin{enumerate}[(i)]
\item $\vol(Y_i)$ are uniformly bounded,
\item $\Lambda(Y_i, X^{+\infty} )=1$,
\item $\mu_i:=\Lambda(X,Y_i)$ is unbounded.
\end{enumerate}

We then define $T=\lim_{\omega }Y_i/\mu_i$ which exists and  is non-trivial by
Proposition~\ref{nontrivial}. In order to reach the desired contradiction, we will show  that
such non-trivial $T$ is trivial.

Note that by Proposition~\ref{weakns}, the first and second points imply that $$Y_i^{+\infty} =
\lim_{m\to\infty} \frac{Y_i\phi^m}{\lambda^m}= X^{+\infty}.$$

Consider a cylinder $C$ in $X$, with stabiliser $H=\stab_G(C)$. We note that $H$ is a free
factor of $G$ and $[\phi^s]$ induces by restriction an automorphism class
$[\phi_H^s]$ of $H$. Also, $\G$ induces a free factor system $\HH$  of $H$. 
The restriction of $f^s$ induces a train track representative of $[\phi^s_H]$ with primitive
transition matrix. Moreover, $[\phi_H^s]$ is $\HH$-irreducible: $f$ permutes the forests $X_i$,
and if $f^s$ shows an invariant free factor system on $H$, this can be translated to others
$\stab_G(C')$ by iterating $f$, and so producing a global free factor system which is $[\phi]$-invariant, which cannot
exist because $[\phi]$ is irreducible. In Particular, Theorem~\ref{CocompactnessPrimitive}
applies to $[\phi_H^s]$.

Next, for each of the $G$-trees above, we may form the minimal invariant $H$ subtree. We denote
this invariant subtree with a subscript $H$, for instance $Y_{i,H}^{+\infty}$.
The fact that $Y_i^{+\infty}= X^{+\infty}$ implies that $Y_{i,H}^{+\infty}= X^{+\infty}_H$ and
hence  $\Lambda(Y_{i,H}, X_H^{+\infty} ) = 1$. We still get that $X_H, Y_{i,H}$ are minimally
displaced points for $\phi^s_H$, whose volumes are uniformly bounded. By
Theorem~\ref{CocompactnessPrimitive}, 
$\phi_H^s$ acts co-compactly on its minimally displaced set and this, by
Theorem~\ref{cocompactnessdefs} point $(vii)$, means that 
 	$$
	\Lambda(X_H, Y_{i,H}) = \sup_{h \in H, l_X(h) \neq 0} \frac{\ell_{Y_i}(h)}{\ell_X(h)}
        \text{ is bounded}. 
	$$
	
	But since
	
	$$
	\ell_T(h) = \frac{\lim_{\omega} \ell_{Y_i}(h)}{\mu_i},
	$$
	and $\mu_i$ is unbounded, we deduce that $\ell_T(h) = 0$ for all $h \in H$.
        By Lemma~\ref{Arc Stabiliser},
        this implies that $H$ fixes a unique point of $T$, and that this is the same point
        fixed by any  $\stab_G(v)$, for  $v \in C$ with nontrivial stabiliser. 
	
	In particular, we may define a $G$-equivariant map from $\mathcal{T}$ to $T$, by
        mapping each vertex to the unique point of $T$ which is fixed by the corresponding (and
        non-trivial) stabiliser. By Lemma~\ref{edgestabs} and Lemma \ref{Arc Stabiliser}, each
        edge is actually mapped to a point. This means that the whole $G$-tree $\mathcal{T}$ is
        mapped to a point. In this case, as the map from $\mathcal{T}$ to $T$ is
        $G$-equivariant, there would be a fixed point for the whole group $G$. But this would
        imply that $T$ is trivial in the sense of translation length functions, contradicting the
        non-triviality of $T$.
\end{proof}

\quasiline

\begin{proof}
	The idea is to simply apply the Svarc-Milnor Lemma. The action of $\langle \phi \rangle$ is clearly properly discontinuous and Theorem~\ref{CoCompactness} gives us cocompactness. 
        The only obstacle is that the symmetric Lipschitz metric $d_{sym}$ is not geodesic (or, even,
 a length metric). Define the intrinsic metric $d_I$ to be the infimum of lengths of paths between any two points. Notice that since $\Min_1(\phi)$ is thick, quasi-symmetry implies that the asymmetric Lipschitz metric $d_{out}$ and $d_{sym}$ are bi-Lipschitz equivalent functions. Since $d_{out}$ is a geodesic asymmetric metric, we deduce that $d_I$ and $d_{sym}$ are also bi-Lipschitz equivalent, and we are done.  
	\end{proof}

\section{Applications}

\label{s8}

\subsection{Relative Centralisers}
In this section we give an application of our main result, regarding relative centralisers of
relatively irreducible automorphisms with exponential growth.

\begin{thm}\label{RelativeCentraliser}
Let $G$ be a group, $\G$  a non-trivial free factor system for $G$, and $\O_1(\G)$ be the
corresponding co-volume $1$ section of relative Outer Space. Let $[\phi]\in\Out(\G)$ be
$\G$-irreducible with exponential growth, and let $X\in\Min_1(\phi)$. Let $C(\phi)$ be the
relative centraliser of $[\phi]$ in $\Out(G,\G)$. Then there is finite index subgroup $C_0(\phi)$ of $C(\phi)$, such that $C_0(\phi)$ is the (interal) direct product  
$$C_0(\phi) = C_X(\phi) \times \langle [\phi]  \rangle,$$

where  $C_X(\phi) =\{[\psi]\in C_0(\phi): X\psi=X\}=\stab(X)\cap C_0(\phi)$.
\end{thm}
\begin{proof}
First, note that $C(\phi)$ preserves $\Min(\phi)$. By Theorem~\ref{CoCompactness}, there is a
fundamental domain $K$ for the action of $\langle[\phi]\rangle$ on $\Min_1(\phi)$, which consists of finitely
many simplices. Without loss of generality we assume that $X \in K$. For any $[\psi] \in C(\phi)$ we have that $X \psi \in K \langle [\phi] \rangle$, since $K$ is a fundamental domain. 

Define $C_0(\phi):=\{ [\psi]  \in C(\phi) \ : \ X\psi \in X \langle [\phi] \rangle \}$. Since $K$ is finite, $C_0(\phi)$ is a finite index subgroup of $C(\phi)$ containing $[\phi]$. By definition, for every $[\psi] \in C_0(\phi)$, there is an $n \in \Z$ such that $X \psi = X \phi^n$. Hence $[\psi] = [\alpha \phi^n]$ for some $[\alpha] \in \stab(X) \cap C_0(\phi) = C_X(\phi)$. Therefore $C_0(\phi) =  C_X(\phi)\langle [\phi] \rangle$. 

(Morally, one could take $C_0(\phi)$ to be the subgroup which acts trivially on $\Min_1(\phi)/\langle [\phi] \rangle$. This is a finite index subgroup since automorphisms preserve simplices and metric structures, so therefore each point in the quotient space has a finite orbit under the action of $C(\phi)$.)  

Moreover, since $[\phi]$ acts without periodic points in $\O(\G)$, $\langle [\phi] \rangle \cap C_X(\phi) = \{ 1\}$. As $\langle [\phi] \rangle$ commutes with $C_X(\phi)$ we get that $C_0(\phi)$ is the direct product, $C_0(\phi) = C_X(\phi) \times \langle [\phi]  \rangle$.

\end{proof}

Note that the previous result generalises a well known result for free groups, that centralisers of irreducible automorphisms with irreducible powers are virtually cyclic (see \cite{BFH-Laminations0}). It also generalises a result of the third author who proved a similar result for relative Centralisers of relatively irreducible automorphisms, with the extra hypothesis that all the powers of the automorphism are irreducible (\cite{S2}).

\subsection{Centralisers in $\Out(F_3)$}
In this section we study centralisers of automorphisms in $\Out(F_3)$. The main result of this
section is the following. 

\Centralisers

Before going into the proof, we need to quote some preliminary fact. Our proof is based on Remark~\ref{Irreducibility}: Any automorphism $[\phi] \in \Out(F_3)$ is irreducible with respect to some relative outer space $\O(\G)$, for some free factor system $\G$ of the free group $F_3$. Equivalently, in the language of free factor systems, $\G$ is a maximal $[\phi]$-invariant free factor system.
However, a maximal free factor system for $[\phi]$ is not necessarily unique. In fact, there are automorphisms with infinitely many different maximal invariant free factor systems.
The following theorem shows that under the extra assumption that $[\phi]$ does not act periodically on any free splitting (i.e. point of some relative outer space), there are finitely many maximal invariant free factor systems. This is proved by Guirardel and Horbez in~\cite{GH2}.

\begin{prop}[{\cite{GH2}}]\label{NoFixedFS}
Let $[\phi]\in\Out(F_n)$. Suppose that there is no free splitting of $F_n$ which is preserved by some power of $[\phi]$. Then there are finitely many maximal $[\phi]$-invariant free factor systems $\G_1,\G_2,\ldots,\G_K$. As a consequence, the relative centraliser $C_{\G_i}(\phi)$ has finite index in $C(\phi)$, for $i=1,\ldots,K$.
\end{prop}
\begin{proof}
The first part is a special case of~\cite[Corollary~1.14]{GH2}.
For the second part, we note that $C(\phi)$ preserves the finite set of maximal $[\phi]$-invariant free factor systems $\{\G_1,\ldots,\G_K\}$. As the relative centraliser with respect to the free factor system $\G_i$, is simply $C_{\G_i}(\phi) = \Out(F_n,\G_i) \cap C(\phi)$, the result follows.
\end{proof}

On the other hand, we need to understand the complementary case of an automorphism that acts periodically on a free splitting.
This case has also been studied in~\cite{GH2}. If $T,S$ are two $F_n$-trees with trivial edge
stabilisers (i.e. free splittings), then we say that $T$ dominates $S$ if point stabilisers in
$T$ are elliptic in $S$. In other words, if $T \in \O(\G_1), S \in \O(\G_2)$, then $T$
dominates $S$ if and only if $\G_1 \leq \G_2$. Alternatively, $T$ dominates $S$ if $\Lambda(T,S)<\infty$.
From~\cite{GH2} we can also extract the following proposition.

\begin{prop}\label{FixexFS}
Let $[\phi]\in\Out(F_n)$. Let's assume that there is a power of $[\phi]$ fixing a free
splitting. Then there is a maximal (with respect to domination) $\langle[\phi]\rangle$-periodic free
splitting $T \in \O(\G)$, for some free factor system $\G$.  All such maximal free splittings,
belong to the same relative outer space $\O(\G)$. 
Moreover, if $[\phi]$ has infinite order, then the centraliser $C(\phi)$ preserves the free factor system $\G$.
\end{prop}
\begin{proof}
The first part is~\cite[Proposition~6.2]{GH2} for the cyclic subgroup $H = \langle[\phi]\rangle$. The second part follows by~\cite[Theorem~8.32]{GH2}.
\end{proof}

\begin{rem}
We recall that {\em maximal}, invariant, free factor systems are defined to be maximal with
respect to the natural ordering $\leq$ on free factor systems of $F_n$. It is important to
mention here a {\em maximal} free splitting means that it belongs to the {\em minimal}, in
terms of the ordering, relative outer space! 
\end{rem}

The linear growth case cannot be really studied using the methods that are presented in this paper, so we need the following result:
\begin{thm}[\cite{AM}]\label{linear growth} 
Centralisers of linearly growing automorphisms in $\Out(F_n)$ are finitely generated.
\end{thm}

We need also the following well known result for $\Out(F_2)$.
\begin{thm}\label{Out(F_2)}
Centralisers of infinite order elements in $\Out(F_2)$ are virtually cyclic.
\end{thm}
\begin{proof}
This is clear as $\Out(F_2)$ is virtually $F_2$ and centralisers of non-trivial elements in $F_2$ are cyclic. 
\end{proof}

Now we are in position to start the proof of the main result of this section.
\begin{proof}[Proof of Theorem~\ref{Out(F_3)}]
  
The possible free factor systems of $F_3$ have one of the following types (for some free basis $\{a,b,c\}$):

\begin{enumerate}
	\item $\G$ = $\emptyset$. Note that in this case $\O(\G) = CV_3$.
	\item $\G = \{ \langle a\rangle \}$.
	\item $\G = \{ \langle a\rangle, \langle b\rangle \}$.
	\item $\G = \{ \langle a\rangle, \langle b \rangle, \langle c\rangle \}$.
	\item $\G = \{ \langle a,b\rangle \}$.
	\item $\G = \{\langle a,b\rangle, \langle c\rangle \}$.
\end{enumerate}

\begin{rem}\label{SubgroupsOfStabilisers}
The stabilisers of points of a relative outer space of a free product are described in \cite{GuirardelLevitt}, in terms of the elliptic free factors $G_i$ of $\G$ and the automorphisms groups $\Aut(G_i)$. In cases $(1)-(4)$, stabilisers of points are virtually $\mathbb{Z}^k$ for some uniformly bounded $k$. In particular, any subgroup of the stabiliser in these cases is finitely presented.
\end{rem}

Let $[\phi] \in \Out(F_3)$.
Let's first assume that our automorphism and all of its powers do not fix a point of some
relative outer space of $F_3$.
As noticed (Remark~\ref{Irreducibility}), there is some relative outer space $\O(\G)$ for which
$[\phi]$ is irreducible. Note that under our assumption that no power of $[\phi]$ fixes a free
splitting, we get that $\lambda_{\O}(\phi) > 1$. Therefore, cases $(5)$ and $(6)$ of above
list cannot appear under our assumptions, as the corresponding relative outer spaces are consisted by a single point (assuming volume equal to 1) and so there are no automorphisms of $\Out(\G)$ with $\lambda_{\G}(\phi) >1$ (all such automorphisms fix a point of $\O(\G)$).

In any other case, by Theorem~\ref{RelativeCentraliser}, $C_{\G} (\phi)$ has a finite index
subgroup which is a $\mathbb{Z}$-extension of $C_X(\phi)$, where $C_X(\phi)$ is the subgroup of
$C_{\G}(\phi)$, acting trivially on some $X \in \O(\G)$. By Remark~\ref{SubgroupsOfStabilisers},
$C_X(\phi)$ is finitely presented. Therefore, $C_{\G}(\phi)$ is finitely presented, as a
$\Z$-extension of a finitely presented group. By Proposition~\ref{NoFixedFS}, the centraliser
$C(\phi)$ of $[\phi]$ in $\Out(F_3)$ has as a finite index subgroup which is finitely presented
the (the group $C_{\G}(\phi)$), and therefore $C(\phi)$ is finitely presented itself. In particular, $C(\phi)$ is finitely generated.

We now assume that our automorphism has a power that fixes a point of some relative outer
space of $F_3$.  There is a maximal such free splitting with respect to domination, by Proposition~\ref{FixexFS}, and there is a free factor system $\G$ such that all such maximal free splittings belong to the same relative outer space, $\O(\G)$. We deal with the possible cases for $\G$ as enumerated above.

In case $(1)$, $[\phi^k]$ fixes a point of $CV_3$, then $[\phi]$ has finite order, and by \cite{Kal}, $C(\phi)$ is finitely presented.

In cases $(2)-(4)$, $[\phi^k]$ fixes a point $T$ of the corresponding relative outer space. By the description of stabilisers of points in~\cite{GuirardelLevitt}, it is easy to see that $[\phi^k]$ (and so $[\phi]$) has linear growth as an automorphism of $\Out(F_3)$ and so the result follows by Theorem~\ref{linear growth}.

For case $(5)$, note that the corresponding relative outer space is a single point, hence preserving this kind of free factor system is equivalent to preserving the corresponding splitting. Hence, by Proposition~\ref{FixexFS}, since all periodic maximal free splittings belong to the same relative outer space, there is a unique one. Since the $[\phi]$ image of a periodic splitting is again periodic, this means that the splitting is actually fixed by $[\phi]$.


We switch now to  elements of $\Aut(F_3)$. Let $\Phi\in[\phi]$ which actually fixes $H$ (not just up to conjugacy). In other words, $\Phi(H) = H$.
Consider the restriction $\Phi_H$ of $\Phi$ on $H$, which induces an element of $\Out(H)$.  If $\Phi_H$ has finite order as an outer automorphism, then it is easy to see that $\Phi$ has linear growth and so, as before, $C(\phi)$ is finitely generated by Theorem~\ref{linear growth}.

So, let's assume now that $\Phi_H$ has infinite order as an outer automorphism.
The subgroup of $\Aut(F_3)$ projecting to the centraliser $C(\phi)$ in $\Out(F_3)$, is $C =
\{\Theta \in \Aut(F_3): [\Theta, \Phi] \in \Inn(F_3)\}$. We will show that $C$ is finitely
generated, which will implies that $C(\phi)$ is finitely generated.

By Proposition~\ref{FixexFS}, if $\Theta\in C$, then $[\Theta]\in\Out(F_3)$ fixes the
conjugacy class of $H$, so we have a well defined homomorphism $\pi:C\to \Out(H)$.
It is easy to see that the image of $\pi$ is in fact contained in the centraliser of $[\Phi_H]$ in
$\Out(H)$, which, by Theorem~\ref{Out(F_2)}, is virtually
cyclic. Therefore $C^0=\pi^{-1}\langle [\Phi_H]\rangle$ is a finite index subgroup of $C$. Hence,
it is enough to show that $C^0$ is finitely generated.

Let $\Theta \in C^0$. We assume without loss, up to composing with an inner automorphism of
$F_3$, that $\Theta(H) = H$. As $\Theta \in C^0$, the restriction of $\Theta$ on $H$, which we
denote by $\Theta_H$, is of the form $\Theta_H = \Phi_H ^k ad(h)$, where $ad(h) \in \Inn(H)$,
for some $k\in\Z$ and $h\in H$. Therefore, if we denote by $C^1$ the subgroup of $C^0$ of those
automorphisms acting as the identity on $H$, we get that $C^0$ is generated by the generators of $C^1$, $\Phi_H$, and the generators of $\Inn(H)$ (which is clearly finitely generated).
In particular, it is enough to show that $C^1$ is finitely generated.

Recall that we are working with a free basis $\{a,b,c\}$, with $H=\langle a,b\rangle$. Since $\Phi(H)=H$, we
must have $\Phi(c) = h_1c^{\epsilon}h_2$, where $\epsilon \in \{-1,1\}$ and $h_1,h_2 \in H$,
and a similar equation holds for  elements of $C^1$. Up to passing to a finite index subgroup $C^2$ of $C^1$, we can assume that $\Theta(c) = xcy, x,y \in H$. As we pass to finite index subgroup, it is clear that it is enough to prove that $C^2$ is finitely generated.

Since $\Theta\in C^2<C^1$, hence $\Theta_H = Id_H$, we get that $\Phi \Theta (a) = \Theta \Phi(a) $ and $\Phi \Theta (b) = \Theta \Phi(b) $. The remaining part of the proof is to write down the equations corresponding to $\Phi \Theta (c) = \Theta \Phi(c) $, which is equivalent to the fact that $\Phi$ and $\Theta$ commute (under our assumptions that $\Theta$ acts as the identity of $H$ and $\Phi$ preserves $H$, it is clear that $\Phi$ and $\Theta$ commute up to inner automorphism if and only if they genuinely commute).

We have:
$$\Phi \Theta (c) = \Phi(xcy) = \Phi(x) h_1 c^{\epsilon} h_2 \Phi(y) \qquad \qquad \Theta \Phi(c) = \Theta (h_1c^{\epsilon}h_2) = h_1 (xcy)^{\epsilon} h_2.$$

Let's first assume that $\epsilon =1$. In this case, the automorphisms $\Phi,\Theta$ commute if
and only if $$\Phi(x) h_1 = h_1x \ \text{ and }\  \Phi(y) h_2 = h_2y \quad \Longleftrightarrow
\quad \Phi(x) = h_1xh_1^{-1} \ \text{ and }\  \Phi(y) = h_2^{-1}yh_2.$$

Note that it is well known that the subgroups $S_{\Phi,h} = \{z: \Phi(z) = hzh^{-1}\}$ is finitely
generated for every $\Phi$ and every $h \in H$ (for example see \cite{BH-TrainTracks} -- since $S_{\Phi,h}$ is just the fixed subgroup of $\Phi$ composed with an inner automorphism). In our case, as any
$\Theta$ with the requested properties is uniquely determined by $x\in S_{\Phi,h_1},y\in
S_{\Phi,h_2^{-1}}$, the above equations identify the subgroup $C^2$ with the product of $S_{\Phi,h_1}$
and $S_{\Phi,h_2^{-1}}$, which means that it is finitely generated. Therefore, the proof concludes in this case.

In case $\epsilon =-1$, the automorphisms commute if and only if
$$\Phi(x)h_1 = h_1 y^{-1} \ \text{ and }\  h_2\Phi(y) = x^{-1}h_2$$ which is equivalent to
$$
\left\{\begin{array}{l}
         \Phi^2(y) = \Phi(h_2)h_1^{-1}y h_1 (\Phi(h_2))^{-1}\\
         x=h_2\Phi(y^{-1})h_2^{-1}
       \end{array}\right.
\Longleftrightarrow
\left\{\begin{array}{l}
         y\in S_{\Phi^2,\Phi(h_2)h_1^{-1}}\\
         x=h_2\Phi(y^{-1})h_2^{-1}
\end{array}\right.
$$
and the thesis follows as above, since $S_{\Phi^2,\Phi(h_2)h_1^{-1}}$ is finitely generated.

Case $(6)$ is similar to case $(5)$ and so we skip the details.
\end{proof}


\begin{thebibliography}{99}

\bibitem{AM} N. Andrew and A. Martino, \textit{Centralisers of linear growth automorphisms of free groups}, in preparation. 		
		
\bibitem{BFH-Laminations0} M. Bestvina, M. Feighn, M. and Handel. \textit{Laminations,
    trees, and irreducible automorphisms of free groups.} Geometric and Functional Analysis
  GAFA, 7(2), 215-244, 1997.

\bibitem{BHH-Laminations1} \bysame, \textit{The Tits Alternative for $Out(F_n)$ I:
    Dynamics of Exponentially-Growing Automorphisms}. Annals of Mathematics, 151(2), second
  series, 517-623, 2000.

\bibitem{BH-TrainTracks} M.Bestvina and M. Handel. \textit{Train Tracks and Automorphisms of
    Free Groups}. Annals of Mathematics, vol. 135, no. 1, 1992, pp. 1–51. 

\bibitem{Carette2014} Mathieu Carette, Stefano Francaviglia, Ilya Kapovich, and Armando
  Martino, \textit{Corrigendum: ``{S}pectral rigidity of automorphic orbits in free groups''},
  Algebraic \& Geometric Topology \textbf{14}, no.~5, 3081--3088. 

\bibitem{Carette2012} \bysame, \textit{Spectral rigidity of automorphic orbits in free groups},
  Algebraic \& Geometric Topology \textbf{12}, no.~3, 1457--1486. 

	
\bibitem{Cohen1999}
Marshall~M. Cohen and Martin Lustig.
\newblock The conjugacy problem for {D}ehn twist automorphisms of free groups.
\newblock {\em Commentarii Mathematici Helvetici}, 74(2):179--200, 1999.

\bibitem{Coulbois2015} Thierry Coulbois and Martin Lustig, \textit{Index realization for
    automorphisms of free groups},  Illinois J. Math. 59(4): 1111-1128.
                  
\bibitem{CM} M. Culler and J. W. Morgan: \textit{Group actions on $\mathbb{R}$-trees},
  Proc. London Math. Sot. 55 (1987), 571-604. 

\bibitem{DahLi}	Francois Dahmani and Ruoyu Li, \textit{Relative hyperbolicity for automorphisms
    of free products and free groups}, Journal of Topology and Analysis, 2020, pages: 1-38. 


\bibitem{Dowdall2015}
Spencer Dowdall, Ilya Kapovich, and Christopher~J. Leininger,
\textit{Dynamics on free-by-cyclic groups}, Geometry \& Topology, 19(5):2801--2899, 2015.

\bibitem{FM11} S. Francaviglia and A. Martino, \textit{Metric properties of outer space},
  Publ. Mat. \textbf{55} (2011), no.~2, 433--473.

\bibitem{FM12} \bysame, \textit{The isometry group of outer space}, Adv. Math. \textbf{231}
  (2012), no.~3-4, 1940--1973. 
  
\bibitem{FM15} \bysame, \textit{Stretching factors, metrics and train tracks for free products.} Illinois J. Math. 59 (2015), no. 4, 859--899. 

\bibitem{FM20} \bysame, \textit{Displacements of automorphisms of free groups I: Displacement
    functions, minpoints and train tracks}, Trans. of AMS, 74 (5) 2021, 3215---3264.

\bibitem{FM21}
  \bysame, \textit{Displacements of automorphisms of free groups II: Connectedness of level
    sets}, Trans. of AMS, 375,  2022, 2511--2551 

\bibitem{FMS} S. Francaviglia, A. Martino and D.Syrigos  \textit{The minimally displaced set of
    an irreducible automorphism is locally finite.} Glasnik Matematicki, 55 (75) 2020, 301-336.

\bibitem{FMS2} \bysame, \textit{The minimally displaced set of an irreducible automorphism of
    FN is co-compact.}, Arch. Math. 116, 369–383 (2021).

\bibitem{Gromov} M.Gromov, \textit{Asymptotic invariants of infinite groups.} London
  Math. Soc. Lecture Note Ser., vol. 182, Cambridge Univ. Press, Cambridge, 1993, ISBN
  0-521-44680-5. 
 	
\bibitem{Guirardel1998} V. Guirardel, \textit{Approximations of stable actions on {${\mathbb R}$}-trees}, Comment. Math. Helv. 73, 89–121 (1998). 

\bibitem{GH}  V. Guirardel and C. Horbez. \textit{Algebraic laminations for free products and
    arational trees.} Algebr. Geom. Topol. 19 (2019), no. 5, 2283--2400. 

\bibitem{GH2} \bysame, \textit{Measure equivalence rigidity of $\mathrm{Out}(F_N)$}, preprint,
  arXiv:2103.03696.  

\bibitem{GuirardelLevitt}  Vincent Guirardel and Gilbert Levitt, \textit{The outer space of a
    free product},  Proc. Lond. Math. Soc. (3) \textbf{94} (2007), no.~3, 695--714. 
		
  
\bibitem{Gupta} R. Gupta, \textit{Loxodromic Elements in the relative free factor complex},
  Geometriae Dedicata, volume 196, 91--–121 (2018).  
	 
\bibitem{HM} Michael Handel and Lee Mosher, \textit{Axes in outer space}, American Mathematical
  Society Providence, RI, Memoirs of the American Mathematical Society, 2012. 

\bibitem{Horbez0} C. Horbez, \textit{Hyperbolic graphs for free products, and the Gromov boundary
of the graph of cyclic splittings} J. of Topology. 9, 401--450 (2016).

\bibitem{Horbez} \bysame, \textit{The boundary of the outer space of a free product.}
  Isr. J. Math. 221, 179–234 (2017).


  
\bibitem{Kal} Sašo Kalajdžievski, \textit{Automorphism group of a free group: Centralizers and
    stabilizers}, Journal of Algebra, Volume 150, Issue 2, 1992, Pages 435-502.


\bibitem{KrsticLustigVogtmann}
Sava Krsti\'{c}, Martin Lustig, and Karen Vogtmann.
\newblock An equivariant {W}hitehead algorithm and conjugacy for roots of
{D}ehn twist automorphisms.
\newblock {\em Proc. Edinb. Math. Soc. (2)}, 44(1):117--141, 2001.

\bibitem{LL} G. Levitt and M. Lustig, \textit{Irreducible automorphisms of  $F_{n}$ have
    North-South Dynamics on Compactified Outer Space},  Journal of the Institute of Mathematics
  of Jussieu, Volume 2, Issue 01, January 2003, pp 59---72.

\bibitem{Meinert2015} Sebastian Meinert, \textit{The {L}ipschitz metric on deformation spaces
    of {$G$}-trees},  Algebr. Geom. Topol. \textbf{15}, no.~2, 987--1029.

\bibitem{Mut}
Jean~Pierre Mutanguha,
\textit{Irreducible nonsurjective endomorphisms of {{\(F_N\)}} are
hyperbolic}, Bulletin of the London Mathematical Society, 52(5):960--976,
2020.


\bibitem{Parry} W. Parry, \textit{Axioms for Translation Length Functions.} In: Alperin
  R.C. (eds) Arboreal Group Theory. Mathematical Sciences Research Institute Publications, vol
  19. Springer, New York, NY, 1991.
  
\bibitem{Paulin1989}
Fr\'{e}d\'{e}ric Paulin, \textit{The {G}romov topology on {$\mathbb{R}$}-trees}, Topology and
its Applications \textbf{32}, 1989, no.~3, 197--221. 

\bibitem{Serre} J.P. Serre, \textit{Trees.} (Translated from the French by John
Stillwell). Springer-Verlag, 1980. ISBN 3-540-10103-9.  

\bibitem{S} D.Syrigos \textit{Asymmetry of outer space of a free product}, Communications in
  Algebra, 46,8, 3442-3460, (2018).

\bibitem{S2} \bysame, \textit{Irreducible laminations for IWIP Automorphisms of free products and Centralisers}, 2014, preprint arXiv:1410.4667

\end{thebibliography}
 \end{document}